\documentclass[english]{amsart}

\usepackage{amsmath,amssymb} 
\usepackage{mathrsfs} 
\usepackage[lite,alphabetic]{amsrefs} 
\usepackage[usenames,dvipsnames]{xcolor} 
\usepackage{bbm} 
\usepackage[scr=boondoxo]{mathalfa} 
\usepackage{dsfont} 
\usepackage[
	colorlinks,
	plainpages,
	citecolor=red!50!black,
    filecolor=Darkgreen,
    linkcolor=blue!40!black,
    urlcolor=cyan!50!black!90]{hyperref} 
\usepackage{enumitem}
\usepackage{version} 

\usepackage{babel}  
\usepackage{csquotes}
\MakeOuterQuote{"}
\usepackage[all]{xy} 
\usepackage{tikz-cd} 
\usepackage{accents} 
\usepackage{stmaryrd} 
\usepackage{upgreek}

\newtheorem{theorem}{Theorem}[section]
\newtheorem{proposition}[theorem]{Proposition}
\newtheorem{corollary}[theorem]{Corollary}
\newtheorem{lemma}[theorem]{Lemma}

\newtheorem{theoremintro}{Theorem}

\theoremstyle{definition}
\newtheorem{definition}[theorem]{Definition}
\newtheorem{remark}[theorem]{Remark}

\newcommand{\id}{\mathrm{id}}

\newcommand{\End}{\mathrm{End}}

\newcommand{\Hom}{\mathrm{Hom}}

\newcommand{\iso}{\xrightarrow{\,\smash{\raisebox{-0.5ex}{\ensuremath{\scriptstyle\sim}}}\,}}

\newcommand{\into}{\hookrightarrow}
\newcommand{\onto}{\twoheadrightarrow}

\makeatletter  
\newcommand{\sbullet}{%
  \hbox{\fontfamily{lmr}\fontsize{.4\dimexpr(\f@size pt)}{0}\selectfont\textbullet}}

\makeatother

\newcommand{\mfg}{\mathfrak{g}}
\newcommand{\mfh}{\mathfrak{h}}

\newcommand{\mfn}{\mathfrak{n}}

\newcommand{\mfsl}{\mathfrak{s}\mathfrak{l}}

\newcommand{\mcF}{\mathcal{F}}

\newcommand{\mcL}{\mathcal{L}}

\newcommand{\mcQ}{\mathcal{Q}}
\newcommand{\mcR}{\mathcal{R}}

\newcommand{\mcZ}{\mathcal{Z}}

\newcommand{\mbE}{\mathbf{E}}

\newcommand{\mbI}{\mathbf{I}}

\newcommand{\mbbV}{\mathbb{V}}

\newcommand{\C}{\mathbb{C}}
\newcommand{\Q}{\mathbb{Q}}

\newcommand{\Z}{\mathbb{Z}}

\newcommand{\msh}{\mathsf{h}}

\newcommand{\msq}{\mathsf{q}}

\newcommand{\msz}{\mathsf{z}}

\newcommand{\al}{\alpha}

\newcommand{\veps}{\varepsilon}

\title[Braid group actions and affine quantum groups]{
Braid group actions, Baxter polynomials,\break and affine quantum groups
}

\author[N. Friesen]{Noah Friesen}
\address{Department of Mathematics and Statistics, University of Saskatchewan}
\email{noah.friesen@usask.ca}

\author[A. Weekes]{Alex Weekes}
\address{Department of Mathematics and Statistics, University of Saskatchewan}
\email{weekes@math.usask.ca}

\author[C. Wendlandt]{Curtis Wendlandt}
\address{Department of Mathematics and Statistics, University of Saskatchewan}
\email{wendlandt@math.usask.ca}


\allowdisplaybreaks
\numberwithin{equation}{section}
\newcommand {\curtiscomment}[1]{\footnote{\textcolor{magenta}{C:\,#1}}}

\newcommand {\noahcomment}[1]{\footnote{\textcolor{teal}{N:\,#1}}}

\newcommand{\g}{\mfg} 
\newcommand{\Yhg}[1][]{Y_\hbar^{#1}(\mfg)} 
\newcommand{\Root}{\Delta} 
\newcommand{\UqLg}[1][]{U_q^{#1}(L\g)}

\newcommand{\Bg}{\mathscr{B}_\mfg} 
\newcommand{\Wg}{\mathscr{W}_\mfg}     
\newcommand{\Monic}{\mathscr{M}} 
\newcommand{\qweight}{\Phi}

\newcommand{\braid}{\uptau}    				
\newcommand{\mbraid}{\mathsf{T}}         
\newcommand{\M}[1][]{#1^\Monic}     
\newcommand{\qshift}{\mathsf{q}}

\newcommand{\Poles}{\sigma}

\newcommand{\tminor}{\mathscr{T}} 
\newcommand{\bt}{\mathcal{B}} 
\newcommand{\sbinom}[2]{\begin{bmatrix}{#1}\\{#2}\end{bmatrix}} 

\begin{document}

\begin{abstract}
It is a classical result in representation theory that the braid group $\Bg$ of a simple Lie algebra $\g$ acts on any integrable representation of $\g$ via triple products of exponentials in its Chevalley generators. In this article, we show that a modification of this construction induces an action of $\Bg$ on the commutative subalgebra $\Yhg[0]\subset \Yhg$ of the Yangian by Hopf algebra automorphisms,
which gives rise to a representation of the Hecke algebra of type $\g$ on a flat deformation of the Cartan subalgebra $\mathfrak{h}[t]\subset \g[t]$. By dualizing, we recover a representation of $\Bg$ constructed in the works of Y.~Tan and V.~Chari, which was used to obtain sufficient conditions for the cyclicity of any tensor product of irreducible representations of $\Yhg$ and the quantum loop algebra $\UqLg$. We apply this dual action to prove that the cyclicity conditions from the work of Tan are identical to those obtained in the recent work of the third author and S. Gautam. Finally, we study the $\UqLg$-counterpart of the braid group action on $\Yhg[0]$, which arises from Lusztig's braid group operators and recovers the aforementioned $\Bg$-action defined by Chari. 
\end{abstract}

\maketitle

{\setlength{\parskip}{0pt}
\setcounter{tocdepth}{1} 
\tableofcontents
}

\section{Introduction}\label{sec:I}

\subsection{Background}\label{ssec:I-back}

Every simple Lie algebra $\g$ over $\C$ has an associated {braid group} $\Bg$, which in the case of $\g = \mathfrak{sl}_{n}$ recovers Artin's braid group $B_n$ on $n$ strands.  The group $\Bg$ surjects onto the Weyl group $\Wg$, and both groups play leading roles throughout representation theory and related fields. As an elementary but important example, it is well-known that   $\Bg$ acts on any integrable representation $(V, \phi)$ of $\g$, with the generators $\{ \braid_i\}_{i \in \mbI}$ of $\Bg$ acting via triple products of exponentials:
\begin{equation}
\label{eq:intro-Bg-triple-exp}
\braid_i \ \mapsto \  \exp\!\big(\phi(e_i)\big) \exp\!\big(\phi(-f_i)\big) \exp\!\big(\phi(e_i)\big) \in \mathrm{Aut}(V).
\end{equation}
Here $e_i, f_i$ are Chevalley generators for $\g$, and $i \in \mbI$ runs over the nodes of its Dynkin diagram. One standard application of this $\Bg$-action is to provide isomorphisms of weight spaces $V_\mu \cong V_{w (\mu)}$ for  elements $w \in \Wg$ of the Weyl group.
Many other important applications of $\Bg$ occur in the theory of quantum groups.  For example, Lusztig \cite{Lusztig} constructed an action of $\Bg$ on the Drinfeld--Jimbo quantum group $U_q(\g)$ and its representations, with these actions being essential in many constructions, such as PBW bases. We refer the reader to the many excellent standard references on quantum groups, such as \cite{Jantzen}*{\S8} or \cite{CPBook}*{\S8.1}, for further details. 

In the present paper, we study certain actions of $\Bg$ on \emph{affine} quantum groups and their subalgebras.  More precisely, we focus on two cases: the Yangian $\Yhg$ and the quantum loop algebra $\UqLg$, which are deformations of the enveloping algebras of the current algebra $\g[t]$ (resp.~the loop algebra $\g[s^{\pm 1}]$).  The algebras $\Yhg$ and $\UqLg$ are both important quantum groups originally arising in the theory of integrable systems, in particular through the groundbreaking work of Drinfeld \cites{Dr,DrNew}.  These algebras have subsequently been studied by many others, both for their interesting representation theory and their ties to geometry and mathematical physics. 

As in the main body of this paper, to describe our results in detail we first turn our attention to $\Yhg$, and defer the corresponding discussion for $\UqLg$ to Section \ref{ssec:I-qloop} at the end of this introduction.  As a first inkling of the appearance of $\Bg$ in the context of the Yangian, we note that there is a natural inclusion $U(\g) \subset \Yhg$, and the induced (adjoint) action of $\g$ on $\Yhg$ is integrable.  Thus we may define an action of $\Bg$ on $\Yhg$ via the operators $\braid_i$ from \eqref{eq:intro-Bg-triple-exp}. This action has appeared in a number of contexts \cites{GNW,Kod19b,WeekesThesis}, and is a primary building block for this paper; our main results relate to a certain modification of this braid group action.

\subsection{The modified braid group action}\label{ssec:I-mbraid}

The Yangian $\Yhg$ admits a large commutative subalgebra $\Yhg[0]$ which deforms the enveloping algebra of the current algebra $\mfh[t]$. It is generated by the coefficients of a family of series $\{\xi_i(u)\}_{i\in \mbI}\subset \Yhg[0][\![u^{-1}]\!]$, and admits a unique Hopf algebra structure for which each $\xi_i(u)$ is grouplike. By virtue of the triangular decomposition for $\Yhg$ (see Sections \ref{ssec:Yhg} and \ref{ssec:Bg-modified}), there is a natural linear projection 
\begin{equation*}
\Pi:\Yhg\to \Yhg[0]
\end{equation*}
and we may thus introduce $\{\mbraid_i\}_{i\in \mbI}\subset \End(\Yhg[0])$ by setting 
\begin{equation*}
\mbraid_i:=\Pi\circ \braid_i\Big|_{\Yhg[0]}: \Yhg[0]\to \Yhg[0].
\end{equation*}
These endomorphisms, which we call the \textit{modified braid group operators} on $\Yhg[0]$, are the main focus of this article. The following theorem summarizes some of their key properties, which collectively provide our first main result.
\begin{theoremintro}\label{T:I} The modified braid group operators $\mbraid_i$ have the following properties:
\begin{enumerate}[font=\upshape]\setlength{\itemsep}{3pt}
\item\label{I.1} They are Hopf algebra automorphisms of $\Yhg[0]$ satisfying the defining 
relations of $\Bg$, with inverses $\mbraid_i^{-1} = \Pi^{op} \circ \braid_i \big|_{\Yhg[0]}$ for $\Pi^{op} : \Yhg\rightarrow \Yhg[0]$ the ``opposite'' projection.
\item\label{I.2} They are uniquely determined by the following formulas, for each $j\in\mbI$:
\begin{equation*}
\mbraid_i(\xi_j(u))=\xi_j(u)\prod_{k=0}^{|a_{ij}|-1} \xi_i\!\left(u-\frac{\hbar d_i}{2}(|a_{ij}|-2k)\right)^{(-1)^{\delta_{ij}}},
\end{equation*}
where $(a_{ij})_{i\in \mbI}$ and  $(d_i)_{i\in \mbI}$ are the Cartan matrix and symmetrizing integers of $\g$, respectively.
\item\label{I.3} They restrict to automorphisms of the subspace $\mathbb{V}$ spanned by all coefficients of the series $\log\xi_j(u)$, for $j\in \mbI$, which satisfy the defining relations of the Hecke algebra of type $\g$.

\item\label{I.4} The diagonal factor $\mcR^0(z)$  of the universal $R$-matrix of $\Yhg$ is a $\Bg$-invariant element of $ \Yhg[0]^{\otimes 2}[\![z^{-1}]\!]$:
\begin{equation*}
(\mbraid_i\otimes \mbraid_i)(\mcR^0(z))=\mcR^0(z).
\end{equation*}
\item\label{I.5} By dualizing the action from \eqref{I.1}, one obtains an action of $\Bg$ on the group $(\C(\!(u^{-1})\!)^\times)^\mbI$ by automorphisms. Moreover, this restricts to an action on $( \C(u)^\times)^\mbI$ introduced by Tan in \cite{tan-braid}*{Prop.~3.1}.
\end{enumerate} 
\end{theoremintro}
The first four parts of this theorem are established in Section \ref{sec:Braid}, in chronological order.
Part \eqref{I.1} is precisely the statement of Theorem \ref{T:mbraid} --- this is the first crucial property of the operators $\mbraid_i$ proven after their introduction in Section \ref{ssec:Bg-modified}. The formulas for $\mbraid_i(\xi_j(u))$ stated in Part \eqref{I.2} are then derived in Section \ref{ssec:Bg-Y0-formulas}; see Corollary \ref{C:mbraid-t-xi}. That the modified braid group operators $\mbraid_i$ give rise to an action of the Hecke algebra of type $\g$ (see Definition \ref{D:Hecke}) on $\mathbb{V}$  is proven in Proposition \ref{P:Hecke}. Finally, the invariance of $\mcR^0(z)$ stated in Part \eqref{I.5} is established in the last subsection of Section \ref{sec:Braid} (Section \ref{ssec:action-R-matrix}); see Proposition \ref{P:tau(R)}.

The final assertion of Theorem \ref{T:I} is deduced in  Section \ref{sec:Weights}. In more detail, we first explain in Section \ref{ssec:wts-dual} how to dualize the representation of $\Bg$ from Part \eqref{I.1} of Theorem \ref{T:I} to obtain an action of $\Bg$ on the group of algebra homomorphisms
\begin{equation*}
\Hom_{Alg}(\Yhg[0],\C)\cong (1+u^{-1}\C[\![u^{-1}]\!])^\mbI.
\end{equation*}
We then use this action and the elementary theory of formal, additive difference equations to construct a $\Bg$-action on the larger space  $\Monic^\mbI$, where $\Monic$ is the group of monic Laurent series in $u^{-1}$; see Proposition \ref{P:Bg-difference}. The $\Bg$-action on $(\C(\!(u^{-1})\!)^\times)^\mbI$ referred to in the statement of \eqref{I.5} is then recovered as the product of $\Monic^\mbI$ and the torus $(\C^\times)^\mbI$, which admits a natural action of $\Wg$; see Remark \ref{R:big-extend}.

Let us now explain the connection between this representation and that constructed by Tan in Proposition 3.1 of \cite{tan-braid}. Therein, an action of $\Bg$ on the group $(\C(u)^\times)^\mbI$ was defined directly, drawing inspiration from the results of Chari \cite{Chari-braid} for quantum loop algebras. That this is a subrepresentation of $(\C(\!(u^{-1})\!)^\times)^\mbI$ with action as described in the previous paragraph is an immediate consequence of the formulas describing the action of the generators $\braid_j\in \Bg$ on $\Monic^\mbI$ obtained in Corollary \ref{C:Monic-braid-compute} using Part \eqref{I.2} of Theorem \ref{T:I}; see Remarks \ref{R:Tan} and \ref{R:big-extend}. In fact, obtaining this interpretation of Tan's action is one of the original inspirations behind this article --- we believe that the theory developed in Sections \ref{sec:Braid} and \ref{sec:Weights} provides a natural theoretical framework for understanding it. In Section \ref{ssec:wts-extemal}, we use this theory to provide a simple proof that Tan's action computes the eigenvalues of the series $\{\xi_i(u)\}_{i\in \mbI}$ on any extremal weight vector in a finite-dimensional highest weight module of $\Yhg$; see Proposition \ref{P:extremal}. As will be explained in Remark \ref{R:Tan-extremal}, this provides a different proof, and generalization, of \cite{tan-braid}*{Prop.~4.5}.

We now give a few auxiliary remarks to complete our discussion of Theorem \ref{T:I} and the results of Sections \ref{sec:Braid} and \ref{sec:Weights}. We first note that the generating series $(A_j(u))_{j\in \mbI}$ for $\Yhg[0]$ introduced by Gerasimov \textit{et al.} in \cite{GKLO} (see Section \ref{ssec:GKLO}) play a crucial role in our results. For instance, the formulas for $\mbraid_i(\xi_j(u))$ given in Part \eqref{I.2} are obtained by first computing the images of each $A_j(u)$ under the standard braid group operators $\braid_i$ on $\Yhg$, and then using the relation between the series $A_j(u)$ and $\xi_k(u)$; see \eqref{GKLO-A}. It is worth pointing out that, when $\g$ is simply laced, these formulas agree with those from \cite{DiKh00}*{\S8}. This suggests that there is an action of $\Bg$ on a completion of the Yangian double $\mathrm{D}\Yhg$ which agrees with that defined by \eqref{I.1} on the subalgebra $\Yhg[0]\subset \mathrm{D}\Yhg$.

Next, we note that Section \ref{ssec:action-R-matrix} also contains an analysis of how the operators $\braid_i$ and $\mbraid_i$ interact with the coefficients of the generating matrices $\tminor_V(u)$ arising in the $R$-matrix presentation of the Yangian \cites{Dr, WRTT}, and their $\Yhg[0]$-counterparts. Following \cite{KTWWY} and \cite{WeekesThesis}, we call these coefficients \textit{lifted minors}. In Corollary \ref{C:minors} we prove that $\Bg$ operates on any lifted minor by permuting its indices. We then use this result  to characterize each of the generating series $A_j(u), B_j(u), C_j(u)$ and $D_j(u)$ introduced in \cite{GKLO} as lifted minors; see Proposition \ref{P:t-vs-A}. In particular, this yields a fairly elementary proof of a result from \cite{Ilin-Rybnikov-19}; see Remark \ref{R:Ilin-Ryb}.

\subsection{Cyclicity criteria and Baxter polynomials}\label{ssec:I-baxter}
One of our main motivations for studying the action of the braid group at the heart of Theorem \ref{T:I} stems from a remarkable phenomenon arising from the finite-dimensional representation theory of the Yangian $\Yhg$: the tensor product $V\otimes W$ of two finite-dimensional irreducible representations will \textit{almost always} be a cyclic module (\textit{i.e.,} a highest weight module), and even an irreducible representation of $\Yhg$. This phenomenon is closely related to a number of interesting applications of Yangians and quantum loop algebras, and has been studied extensively; see \cites{chari-pressley-dorey, chari-pressley-RepsChar, guay-tan, MoBook, NaTa98, NaTa02,  tan-braid} and
\cites{AK-qAffine, Chari-braid, chari-pressley-qaffine, frenkel-mukhin, 
frenkel-reshetikhin-qchar, Hern-simple, Hern-cyclic, kashiwara-level-zero, vv-standard}, for instance. 

A connection between this cyclicity property and the $\Bg$-action from \eqref{I.5} of Theorem \ref{T:I} was established in \cite{tan-braid}, using a construction for the quantum loop algebra $U_q(L\g)$ which appeared earlier in \cite{Chari-braid}. To state this precisely, recall that the finite-dimensional irreducible representations of $\Yhg$ are parametrized by $\mbI$-tuples $\underline{\mathrm{P}}=(P_i(u))_{i\in \mbI}$ of monic polynomials --- called \textit{Drinfeld polynomials}. We will write $L(\underline{\mathrm{P}})$ for the representation associated to the tuple $\underline{\mathrm{P}}$. Next, let 
\begin{equation*}
w_0=s_{j_1}s_{j_2}\cdots s_{j_p}\in \Wg
\end{equation*}
be a reduced expression for the longest element in the Weyl group $\Wg$, and set $\braid_{w_r}=\braid_{j_{r+1}}\cdots \braid_{j_p}\in \Bg$ for each $0<r\leq p$. Then,   by Theorem 4.8 of \cite{tan-braid}, the tensor product $L(\underline{\mathrm{P}})\otimes L(\underline{\mathrm{Q}})$ will be cyclic provided one has 
\begin{equation*}
\mathsf{Z}(Q_{j_r}(u+\hbar d_{j_r}))\subset \C\setminus \mathsf{Z}(\M[\braid]_{w_r}(\underline{\mathrm{P}})_{j_r}) \quad \forall \quad 0<r\leq p,
\end{equation*}
where $\mathsf{Z}(P(u))$ is the set of zeroes of $P(u)$,  $\M[\braid]_{w_r}(\underline{\mathrm{P}})$ denotes the action of $\braid_{w_r}\in \Bg$ on $\underline{\mathrm{P}}$ from \eqref{I.5} of Theorem \ref{T:I}, and  $\M[\braid]_{w_r}(\underline{\mathrm{P}})_{j_r}$ is the $j_r$-th component of $\M[\braid]_{w_r}(\underline{\mathrm{P}})$, which is necessarily a polynomial (see Corollary \ref{C:sigma-poly} or Lemma 4.3 and Proposition 4.5 of \cite{tan-braid}). 

Another concrete sufficient condition for the cyclicity of $L(\underline{\mathrm{P}})\otimes L(\underline{\mathrm{Q}})$ was obtained in the third author's recent work \cite{GWPoles} with S. Gautam, which studied the sets of poles $\{\sigma_i(L(\underline{\mathrm{P}}))\}_{i\in \mbI}$ of the generating series of $\Yhg$, viewed as  $\End(L(\underline{\mathrm{P}}))$-valued rational functions of $u$; see Definition \ref{D:Poles}. In Theorem 7.2 of \cite{GWPoles}, it was shown that the module $L(\underline{\mathrm{P}})\otimes L(\underline{\mathrm{Q}})$ is cyclic provided one has 
\begin{equation*}
\mathsf{Z}(Q_i(u+\hbar d_i))\subset \C\setminus \sigma_i(L(\underline{\mathrm{P}})) \quad \forall \quad i\in \mbI
\end{equation*}
and, by \cite{GWPoles}*{Cor.~7.3}, it is irreducible if this condition  also holds with the roles of $\underline{\mathrm{P}}$ and $\underline{\mathrm{Q}}$ interchanged. Moreover, it was proven in  Theorem 4.4 of \cite{GWPoles} that  $\sigma_i(L(\underline{\mathrm{P}}))$ is exactly the set of roots of a distinguished polynomial $\mcQ_{i,L(\underline{\mathrm{P}})}^\g(u)$ called a specialized \textit{Baxter polynomial} \cite{FrHer-15}, which can be understood heuristically as arising as the eigenvalue of a certain abelian transfer operator on the lowest weight vector of $L(\underline{\mathrm{P}})$. The polynomials  $\mcQ_{i,L(\underline{\mathrm{P}})}^\g(u)$ were then computed in Theorem 5.2 of \cite{GWPoles} in terms of the Drinfeld polynomials $P_j(u)$ and the Cartan data of $\g$; we refer the reader to Sections \ref{ssec:transfer} and \ref{ssec:Baxter-poles} for a more detailed summary of these results.

In Section 7.4 of \cite{GWPoles} it was conjectured that the two cyclicity conditions for $L(\underline{\mathrm{P}})\otimes L(\underline{\mathrm{Q}})$ obtained in \cite{tan-braid} and \cite{GWPoles} are always identical. The following theorem, which is the second main result of our article, proves this conjecture. 

\begin{theoremintro} Let $\underline{\mathrm{P}}=(P_j(u))_{j\in \mbI}$ be any $\mbI$-tuple of monic polynomials. Then, for each $i\in \mbI$, the specialized Baxter polynomial $\mcQ_{i,L(\underline{\mathrm{P}})}^\mfg(u)$ admits the factorization 
\begin{equation*}
\mcQ_{i,L(\underline{\mathrm{P}})}^\mfg(u)=\prod_{r:j_r=i} \M[\braid]_{w_r}(\underline{\mathrm{P}})_i.
\end{equation*}
Consequently, the sufficient conditions for the cyclicity of any tensor product $L(\underline{\mathrm{P}})\otimes L(\underline{\mathrm{Q}})$ obtained in \cite{tan-braid} and \cite{GWPoles} are identical. 
\end{theoremintro}
This theorem is a combination of Theorem \ref{T:Baxter-braid}, which establishes the stated formula $ \mcQ_{i,L(\underline{\mathrm{P}})}^\mfg(u)$, and Corollary \ref{C:Cyclic}, which spells out explicitly the second assertion of the above theorem. In fact, Theorem \ref{T:Baxter-braid} provides a more general formula which factorizes the specialized Baxter polynomial associated to \textit{any} extremal weight space of $L(\underline{\mathrm{P}})$ as a product of polynomials arising from the $\Bg$-orbit of $\underline{\mathrm{P}}$.

\subsection{The quantum loop algebra $\UqLg$}\label{ssec:I-qloop}

We now turn to the parallel situation of the quantum loop algebra $\UqLg$, which is the topic of Section \ref{sec:qLoop}.  In this setting there is an action of the braid group $\Bg$ on $\UqLg$ via Lusztig's operators \cite{Lusztig}.  In fact, the algebra $\UqLg$ arises as a subquotient of the quantum affine algebra $U_q(\widehat{\g})$. This larger algebra $U_q(\widehat{\g})$ carries an action of the affine braid group $\mathscr{B}_{\widehat{\g}}$, which was utilized heavily by Beck \cites{Beck1, Beck2} in his works on alternate presentations of $U_q(\widehat{\g})$ and on its Poincar\'{e}--Birkhoff--Witt theorem.  It is not hard to see that the action of the subgroup $\Bg \subset \mathscr{B}_{\widehat{\g}}$ passes to the subquotient $\UqLg$ (see Remark \ref{rmk:Beck-UqLg} for more details), and this action of $\Bg$ is our starting point.

Similarly to the case of the Yangian, we use the action of $\Bg$ to define \emph{modified braid group operators} $\mbraid_i$ on the commutative subalgebra $\UqLg[0] \subset \UqLg$, see Definition \ref{T:mbraid2}.  We prove many analogous properties to those stated for the Yangian in Theorem \ref{T:I} above, which are summarized in Theorem \ref{thm:main-q-thm} below.  In particular, we show that these operators $\{\mbraid_i\}_{i\in\mbI}$ define an action of $\Bg$ on $\UqLg[0]$ over $\C(q)$, and that after dualizing to an action on
$$
\operatorname{Hom}_{Alg}( \UqLg[0], \C(q) )
$$
we recover a braid group action studied by Chari \cite{Chari-braid}.  We also obtain actions of the Hecke algebra of type $\g$ on certain subspaces of  $\UqLg[0]$; see Remark \ref{rmk:q-Hecke}.

In addition, we show that our modified braid group action on $\UqLg[0]$ is compatible with the restriction $\Phi$ of the Gautam--Toledano Laredo homomorphism from \cite{GTL1}*{Thm.~1.4} to a certain $\C[q,q^{-1}]$-form $\mathscr{U}_q^0(L\g) $ of $\UqLg[0]$. More precisely, $\Phi$ may be viewed as an embedding
$$
\Phi: \mathscr{U}_q^0(L\g) \ \into \ \Yhg[0][\![v]\!]
$$
and the compatibility asserts that it intertwines our modified braid group actions on both sides; we refer the reader to Proposition \ref{P:GTL} for the precise statement. 
In fact, our proof proceeds the other way around: we first show by direct calculation (Proposition \ref{P:GTL} and Corollary \ref{cor: braid group via GTL}) that the modified braid group action on $\Yhg[0]$ induces a braid group action on $\UqLg[0]$ by operators $\mbraid_i^\Phi$.  We then use an argument exploiting Drinfeld polynomials (Lemma \ref{lem: technical}) to show that $\mbraid_i = \mbraid_i^\Phi$.  

Finally, we note that Frenkel and Hernandez \cite{FrHer-22} recently defined an action of the Weyl group $\Wg$ on a completion of the space of $q$-characters for $\UqLg$.  In Proposition 6.10 therein, it is shown that this action, after an appropriate truncation, recovers Chari's \cite{Chari-braid} braid group action mentioned above. Both the original $\Bg$-action of \cite{Chari-braid} and this action of the Weyl group $\Wg$ play important roles in the sequel \cite{FrHer-23} to \cite{FrHer-22}, where the $q$-analogues of some of the constructions of Section \ref{sec:Weights} also feature. In particular,  the $q$-difference equation \cite{FrHer-23}*{(4.6)} satisfied by the elements $\Psi_{w(i),a}$ from Definition 3.5 therein is the counterpart of that used to define the action of $\Bg$ on $\Monic^\mbI$ in Proposition \ref{P:Bg-difference} for the special case where $\underline{\mu}$ is the highest weight of the $i$-th fundamental representation $L_{\varpi_i}(a)$ of $\Yhg$; see Section \ref{ssec:Rep}. 
Moreover, the subtle differences between the formulas of Corollaries \ref{C:T_1-rep} and \ref{C:Monic-braid-compute} encoding the action of $\Bg$ on $(1+u^{-1}\C[\![u^{-1}]\!])^\mbI$ and $\Monic^\mbI$, respectively, are reflected in the definition of the elements $\Psi_{w(i),a}$ and  further commented on in \cite{FrHer-23}*{Rem.~3.7} in connection with Langlands duality. We refer the reader to the discussions preceding and following Corollary \ref{C:Monic-braid-compute} for a few additional details in this direction.

\subsection{Acknowledgments}\label{ssec:Acknowledge}

The authors gratefully acknowledge the support of the Natural Sciences and Engineering Research Council
Canada, provided via the CGS M program (N.F.) and the Discovery Grants program (A.W.: Grant RGPIN-2022-03135 and DGECR-2022-00437, C.W: Grant RGPIN-2022-03298 and DGECR-2022-00440).
\section{Yangians}\label{sec:Prelim}

\subsection{The Lie algebra $\g$}\label{ssec:g}
Let $\g$ be a simple Lie algebra over $\C$, with Cartan matrix $(a_{ij})_{i,j \in \mbI}$.  Fix minimal symmetrizing integers $d_i \in \{1,2,3\}$ so that $d_i a_{ij} = d_j a_{ji}$.  Fix Chevalley generators $\{e_i, h_i, f_i\}_{i \in \mbI}$ for $\g$, with corresponding triangular decomposition $\g = \mathfrak{n}^- \oplus \mathfrak{h} \oplus \mathfrak{n}^+$. 

Let $\{\alpha_i\}_{i \in \mbI}$ be the simple roots for $\g$, and let $(\cdot,\cdot)$ be the non-degenerate symmetric bilinear form on $\mathfrak{h}^\ast$ defined by $(\alpha_i, \alpha_j) = d_i a_{ij}$.  Let $Q = \bigoplus_{i \in \mbI} \Z \alpha_i$ denote the root lattice of $\g$ and $Q^+ = \bigoplus_{i \in \mbI} \Z_{\geq 0} \alpha_i$ the positive cone in $Q$. Define the fundamental weights $\varpi_i \in \mfh^\ast$ by $\varpi_i(h_j) = \delta_{ij}$ for $i,j \in \mbI$, and let $\Lambda = \bigoplus_{i \in \mbI} \Z \varpi_i$ denote the weight lattice of $\g$.   Let $\Root \subset Q$ denote the set of roots of $\g$ with positive/negative roots $\Root^\pm$. We denote the root space corresponding to $\alpha\in \Root$ by $\g_\alpha\subset \g$ and, similarly,  we write $U(\g)_\beta\subset U(\g)$ for the homogeneous component of $U(\g)$ associated to any $\beta \in Q$ with respect to the root grading. 


\subsection{The Yangian $\Yhg$}\label{ssec:Yhg}

We now turn to recalling the definition of the Yangian of $\g$, together with some of its basic properties. 
\begin{definition}
Let $\hbar \in \C^\times$ be fixed.  The \textbf{Yangian} $\Yhg$ is the unital associative algebra over $\C$ generated by elements $\{\xi_{i,r}, x_{i,r}^\pm \}_{i \in \mbI, r \in \Z_{\geq 0}}$, subject to the following relations: 
$$
[\xi_{i,r} , \xi_{j,s}] = 0,
$$
$$
[\xi_{i,0}, x_{j,s}^\pm] = \pm d_i a_{ij} x_{j,s}^\pm,
$$
$$
[\xi_{i,r+1}, x_{j,s}^\pm] - [\xi_{i,r}, x_{j,s+1}^\pm] = \pm \hbar \frac{d_i a_{ij}}{2} ( \xi_{i,r} x_{j,s}^\pm + x_{j,s}^\pm \xi_{i,r}),
$$
$$
[x_{i,r+1}^\pm, x_{j,s}^\pm] - [x_{i,r}^\pm, x_{j,s+1}^\pm] = \pm \hbar \frac{d_i a_{ij}}{2} ( x_{i,r}^\pm x_{j,s}^\pm + x_{j,s}^\pm x_{i,r}^\pm),
$$
$$
[x_{i,r}^+, x_{j,s}^-] = \delta_{ij} \xi_{i, r+s},
$$
$$
\sum_{\pi \in S_m} [x_{i,r_{\pi(1)}}^\pm, [ x_{i,r_{\pi(2)}}^\pm, [ \cdots [x_{i,r_{\pi(m)}}^\pm, x_{j,s}^\pm]\cdots]]] = 0 \quad \text{if } i\neq j,
$$
where in the final relation $m = 1 - a_{ij}$ and $S_m$ denotes the symmetric group of degree $m$. 
\end{definition}
%
%
The Yangian admits the structure of a filtered algebra with filtration $\mcF_\bullet\Yhg$ defined by assigning  degree $r$ to $\xi_{i,r}$ and $x_{i,r}^\pm$ for each $i\in \mbI$ and $r\in \Z_{\geq 0}$. The \textit{Poincar\'{e}--Birkhoff--Witt Theorem} for $\Yhg$ asserts that the assignment 
\begin{equation*}
e_i t^r \mapsto d_i^{-1/2} \overline{x_{i,r}^+}, \quad f_i t^r \mapsto  d_i^{-1/2} \overline{x_{i,r}^-},\quad h_it^r \mapsto d_i^{-1} \overline{\xi_{i,r}}
\end{equation*}
extends uniquely to an isomorphism of graded algebras
\begin{equation}\label{Yhg-PBW}
U(\mfg[t])\iso \mathrm{gr}(\Yhg)=\bigoplus_{n\geq 0}\mcF_n\Yhg/\mcF_{n-1}\Yhg,
\end{equation} 
where, for each $y\in \{\xi_i,x_i^\pm\}$, $\overline{y}_r$ is the image of $y_r$ in $\mcF_r\Yhg/\mcF_{r-1}\Yhg$.
A particular consequence of this theorem is that there is an injective homomorphism of algebras $U(\g) \hookrightarrow \Yhg$ defined on Chevalley generators by
\begin{equation*}
e_i \mapsto d_i^{-1/2} x_{i,0}^+, \qquad h_i \mapsto d_i^{-1} \xi_{i,0}, \qquad  f_i \mapsto d_i^{-1/2} x_{i,0}^-.
\end{equation*}
We will henceforth view $U(\mfg)\subset \Yhg$, with the above identification understood. Note that the adjoint action of $\mfh$ on $\Yhg$ gives rise to a weight space decomposition $\Yhg=\bigoplus_{\beta\in Q}\Yhg_\beta$, where 
\begin{equation*}
\Yhg_\beta=\{y\in \Yhg:[h,y]=\beta(h)y \quad \forall\; h\in \mfh\} \quad \forall\quad \beta\in Q.
\end{equation*}
This decomposition equips $\Yhg$ with the structure of a $Q$-graded algebra. 

Consider now the subalgebras $\Yhg[0]$ and $\Yhg[\pm]$ generated by the elements $\{\xi_{i,r}\}_{i\in \mbI,r\in \Z_{\geq 0}}$ and  $\{x_{i,r}^\pm\}_{i\in \mbI,r\in \Z_{\geq 0}}$, respectively. Note that these are filtered, $Q$-graded subalgebras of $\Yhg$. Moreover, the product on $\Yhg$ induces an isomorphism of vector spaces
\begin{equation*}
\Yhg[-] \otimes \Yhg[0] \otimes \Yhg[+]\iso \Yhg,
\end{equation*}
while the graded algebra isomorphism \eqref{Yhg-PBW} restricts to yield isomorphisms
\begin{equation*}
U(\mfh[t])\iso \mathrm{gr}(\Yhg[0]) \quad \text{ and }\quad U(\mfn^\pm[t])\iso\mathrm{gr}(\Yhg[\pm]).
\end{equation*} 
%

\subsection{Alternate generators}\label{ssec:GKLO}

 Another important set of generators for $\Yhg$ was introduced in the work \cite{GKLO} of Gerasimov  \textit{et al.} To recall these, let us first introduce $\xi_i(u)\in \Yhg[0][\![u^{-1}]\!]$ and $x_i^\pm(u)\in \Yhg[\pm][\![u^{-1}]\!]$, for each $i\in \mbI$, by setting
\begin{gather*}
\xi_i(u):=1+\hbar\sum_{r\geq 0} \xi_{i,r}u^{-r-1} \quad \text{ and }\quad x_i^\pm(u):=\hbar\sum_{r\geq 0} x_{i,r}^\pm u^{-r-1}.
\end{gather*}
Then, by  Lemma 2.1 of \cite{GKLO}, there  is a unique tuple of formal series $(A_j(u))_{j\in \mbI}\in (1+u^{-1}\Yhg[0][\![u^{-1}]\!])^\mbI$ satisfying 
\begin{equation}\label{GKLO-A}
\xi_i(u) =\frac{\prod_{j\neq i}\prod_{r=1}^{-a_{ji}}A_j(u-\tfrac{\hbar d_j}{2}(2r+a_{ji}))}{A_i(u) A_i(u-\hbar d_i)}\quad \forall \quad i\in \mbI. 
\end{equation}
The coefficients $A_i(u) = 1 + \hbar \sum_{r \geq 0} A_{i,r} u^{-r-1}$ of these series generate the subalgebra $\Yhg[0]\subset \Yhg$. 
\begin{remark}\label{R:logs}
For each $i\in \mbI$, let $t_i(u)=\hbar\sum_{r\geq 0}t_{i,r}u^{-r-1}$ be the formal series logarithm of $\xi_i(u)$: 
\begin{equation*}
t_i(u):=\log\xi_i(u)=\sum_{n>0} \frac{(-1)^{n+1}}{n} (\xi_i(u)-1)^n.
\end{equation*}
Similarly, set $a_i(u):=\log A_i(u)$ for each $i\in \mbI$. 
Let $\qshift$ be the shift operator on $\Yhg[0][\![u^{-1}]\!]$ defined by $\qshift(f(u))=f(u+\hbar/2)$. Equivalently, one has $\qshift=e^{\frac{\hbar}{2}\partial_u}$. Then the defining relation \eqref{GKLO-A} for the tuple $(A_j(u))_{j\in \mbI}$ is equivalent to
\begin{equation}\label{t-vs-A}
t_i(u)=-\sum_{j\in \mbI} \qshift^{-d_j}[a_{ji}]_{\qshift^{d_j}}(a_j(u)) \quad \forall\quad i\in \mbI,
\end{equation}
where we follow the standard notation for $z$-numbers: if $m\in \Z$, then 
\begin{equation}\label{z-number}
[m]_z=\frac{z^m-z^{-m}}{z-z^{-1}}\in \Z[z,z^{-1}]\subset \Q(z). 
\end{equation}
\end{remark} 

Now let us introduce series $B_i(u)$, $C_i(u)$ and $D_i(u)$, for each $i\in \mbI$, by setting 
\begin{gather*}
B_i(u)=d_i^{1/2} A_i(u)x_i^+(u), \quad C_i(u)=d_i^{1/2} x_i^-(u) A_i(u),\\  D_i(u)=A_i(u)\xi_i(u)+C_i(u)A_i(u)^{-1} B_i(u).
\end{gather*}
Then the coefficients of $A_i(u), B_i(u)$ and $C_i(u)$ generate $\Yhg$ as an algebra. Some of the commutation relations satisfied by these series are spelled out in the following proposition, which is a  restatement of \cite{GKLO}*{Prop.~2.1}.
\begin{proposition}\label{prop: GKLO relations} The series $A_i(u),B_i(u),C_i(u)$ and $D_i(u)$ satisfy the following commutation relations:
\begin{enumerate}

\item For each $i,j \in \mbI$, we have 
\begin{gather*}
[A_i(u),A_j(v)]=0,\\
[B_i(u),B_i(v)]=[C_i(u),C_i(v)]=0,
\end{gather*}

\item For each $i,j\in \mbI$ with $i\neq j$, we have 
\begin{gather*}
[A_i(u),B_j(v)]=[A_i(u),C_j(v)]=[B_i(u),C_j(v)]=0.
\end{gather*}

\item For each $i\in \mbI$, we have 
\begin{equation*}
\begin{aligned}
(u-v)[A_i(u),B_i(v)]&=d_i\hbar\left( B_i(u)A_i(v)-B_i(v)A_i(u)\right),\\
(u-v)[A_i(u),C_i(v)]&=d_i\hbar\left( A_i(u)C_i(v)-A_i(v)C_i(u)\right),\\
(u-v)[B_i(u),C_i(v)]&= d_i \hbar \left( A_i(u)D_i(v)-A_i(v) D_i(u)\right),\\
(u-v)[C_i(u),D_i(v)]&= d_i \hbar \left( D_i(u)C_i(v)-D_i(v)C_i(u)\right),\\
(u-v)[A_i(u),D_i(v)]&=d_i \hbar \left( B_i(u)C_i(v)-B_i(v)C_i(u)\right).
\end{aligned}
\end{equation*}
\end{enumerate}
\end{proposition}
We note that this proposition does \emph{not} give a presentation of $\Yhg$, \textit{i.e.}~these are not a complete set of relations.  However, in the parallel quantum loop algebra setting a complete set of relations is given in \cite{FiTs19}*{Thm.~6.6}.  We will not need this presentation for our purposes.
Indeed, we shall be primarily interested in the following weaker set of relations, which follow readily from the above proposition and encode information about the adjoint action of $\mfg$ on $\Yhg$. 
\begin{corollary}\label{C:[g,GKLO]}
The series $A_i(u),B_i(u),C_i(u)$ and $D_i(u)$ satisfy the following commutation relations:
\begin{enumerate}
\item\label{[g,GKLO]:1} For each $i,j\in \mbI$ with $i\neq j$, we have 
\begin{equation*}
[e_i,A_j(u)]=[f_i,A_j(u)]=0
\end{equation*}
\item\label{[g,GKLO]:2} For each $i\in \mbI$, one has 
\begin{gather*}
[e_i,A_i(u)]=B_i(u), \quad  [e_i,B_i(u)]=0,\\
[e_i,C_i(u)]=D_i(u)-A_i(u),\quad [e_i,D_i(u)]=-B_i(v).
\end{gather*}
\item\label{[g,GKLO]:3} For each $i\in \mbI$, one has 
\begin{gather*}
[f_i,A_i(u)]=-C_i(u), \quad  [f_i,B_i(u)]=A_i(u)-D_i(u),\\
[f_i,C_i(u)]=0, \quad [f_i,D_i(u)]=C_i(u).
\end{gather*}
\end{enumerate}
\end{corollary}

\subsection{Hopf structure and universal R-matrix}\label{sec:Hopf}

The Yangian admits a filtered Hopf algebra  structure which deforms the standard graded Hopf algebra structure on the enveloping algebra $U(\mfg[t])$ (that is, \eqref{Yhg-PBW} becomes an isomorphism of graded Hopf algebras).  The counit $\veps$, coproduct $\Delta$ and antipode $S$ on $\Yhg$ are uniquely determined by the requirement that the embedding $U(\mfg)\subset \Yhg$ from Section \ref{ssec:Yhg} is a homomorphism of Hopf algebras and that, for each $i\in \mbI$, one has 
\begin{gather*}
\veps(t_{i,1})=0, \quad 
\Delta(t_{i,1})=t_{i,1}\otimes 1 + 1\otimes t_{i,1} -\hbar \sum_{\alpha\in \Root^+} (\alpha_i,\alpha) x_\alpha^-\otimes x_\alpha^+,\\
 S(t_{i,1})=-t_{i,1}-\hbar\sum_{\alpha\in \Root^+}(\alpha_i,\alpha) x_\alpha^-x_\alpha^+,
\end{gather*}
where $x_\alpha^\pm\in \mfg_{\pm \alpha}\subset \Yhg$ are any root vectors satisfying $(x_\alpha^+,x_\alpha^-)=1$ for all $\alpha\in \Root^+$, and  $\{t_{i,1}\}_{i\in \mbI}\subset \Yhg[0]$ are as in Remark \ref{R:logs}. Explicitly, one has 
\begin{equation*}
t_{i,1}=\xi_{i,1}-\frac{\hbar}{2}\xi_{i,0}^2 \quad \forall \quad i \in \mbI.
\end{equation*}
%
%

For each $a\in \C$, there is a Hopf algebra automorphism $\tau_a$ of $\Yhg$, called a shift automorphism, uniquely determined by the formulas
\begin{equation*}
\tau_a(\xi_i(u))=\xi_i(u-a) \quad \text{ and }\quad \tau_a(x_i^\pm(u))=x_i^\pm(u-a)\quad \forall \; i\in \mbI.
\end{equation*}
Replacing $a$ by a formal variable $z$  yields instead an algebra embedding 
\begin{equation*}
\tau_z:\Yhg\to \Yhg[][z].
\end{equation*}
This homomorphism plays a crucial role in defining the universal $R$-matrix of the Yangian, which was first constructed by Drinfeld in \cite{Dr}*{Thm.~3}. The following theorem is a restatement of Drinfeld's result.
\begin{theorem}\label{T:Uni-R}
There is a unique element $\mcR(z)\in
1+z^{-1}\Yhg^{\otimes 2}[\![z^{-1}]\!]
$, called the universal $R$-matrix of $\Yhg$, 
satisfying 
\begin{equation}\label{R-inter}
(\tau_z\otimes \mathrm{Id})\Delta^{\mathrm{op}}(x)= \mcR(z) \cdot (\tau_z\otimes \mathrm{Id} )(\Delta(x)) \cdot \mcR(z)^{-1} \quad \forall\; x\in \Yhg
\end{equation}
in $\Yhg^{\otimes 2}(\!(z^{-1})\!)$, in addition to the following identities in $\Yhg^{\otimes 3}[\![z^{-1}]\!]$:
\begin{align*}
(\Delta\otimes \mathrm{Id}) (\mcR(z))&= \mcR_{13}(z)\mcR_{23}(z),\\
(\mathrm{Id}\otimes \Delta) (\mcR(z))&= \mcR_{13}(z)\mcR_{12}(z).
\end{align*}
Moreover, $\mcR(z)$ satisfies $\mcR(z)^{-1}=\mcR_{21}(-z)$ in addition to
\begin{equation*}
(\tau_a\otimes \tau_b)\mcR(z)=\mcR(z+a-b) \quad \forall \quad a,b\in \C. 
\end{equation*}
\end{theorem}
\begin{remark}
Here we have followed the notation and conventions from \cite{GTLW19}*{Thm.~7.4}, where a proof of Drinfeld's theorem was recently given. The element $\mathscr{R}(z)$ from \cite{Dr}*{Thm.~3} is related to $\mcR(z)$ by 
\begin{equation}\label{Drinfeld-uniR}
\mathscr{R}(z)=\mcR(-z)^{-1}=\mcR_{21}(z).
\end{equation}
\end{remark}
The proof of the above theorem given in \cite{GTLW19} reconstructed the universal $R$-matrix $\mcR(z)$ from the components in its Gauss decomposition 
\begin{equation}\label{R:Gauss}
\mcR(z)=\mcR^+(z)\mcR^0(z)\mcR^-(z),
\end{equation}
where $\mcR^0(z)\in 1+u^{-1}\Yhg[0]^{\otimes 2}[\![u^{-1}]\!]$,  $\mcR^+(z)=\mcR_{21}^-(-z)^{-1}$ and $\mcR^-(z)$ is an element of $(\Yhg[-]\otimes \Yhg[+])[\![z^{-1}]\!]$ of the form 
\begin{equation*}
\mcR^-(z)=\sum_{\beta\in Q^+}\mcR^-_\beta(z)\quad \text{ with }\quad \mcR^-_\beta(z)\in (\Yhg[-]_{-\beta}\otimes \Yhg[+]_{\beta})[\![z^{-1}]\!]
\end{equation*}
and $\mcR^-_0(z)=1$. The components $\mcR^-_\beta(z)$ for $\beta\neq 0$ were constructed recursively in the height of $\beta$ in the proof of Theorem 4.1 in  \cite{GTLW19}; see Section 4.2 and (4.5) therein.

The diagonal factor $\mcR^0(z)$ was defined in Section 6 of \cite{GTLW19} using the results of \cite{GTL3}. By definition, it is the unique series in $1+z^{-1}\Yhg[0]^{\otimes 2}[\![z^{-1}]\!]$  satisfying the formal difference equation 
\begin{equation}\label{diff:R^0}
(\qshift^{4\kappa}-1)\log(\mcR^0(z))=\mcL(z),
\end{equation}
where $\qshift$ is the shift operator $e^{\frac{\hbar}{2}\partial_z}$ on $\Yhg[0]^{\otimes 2}[\![z^{-1}]\!]$ (see Remark \ref{R:logs}) and $\mcL(z)\in z^{-2}\Yhg[0]^{\otimes 2}[\![z^{-1}]\!]$ is the element
\begin{equation*}
\mcL(z)=\qshift^{2\kappa}\sum_{i,j\in \mbI} c_{ij}(\qshift)  \bt_i(\partial_z) \otimes \bt_j(-\partial_z) \cdot  (-z^{-2}).
\end{equation*}
Here $\kappa\in \frac{1}{2}\Z$, $c_{ij}(z)\in \Z[z,z^{-1}]$ and $\bt_i(u)\in \Yhg[0][\![u]\!]$ are defined as follows:
\begin{enumerate}\setlength{\itemsep}{5pt}
\item $\kappa=(1/4)c_\g$, where $c_\mfg$ is the eigenvalue of the quadratic Casimir element $C\in S^2(\mfg)\subset U(\mfg)$ on the adjoint representation of $\g$.
\item  For each $i,j\in \mbI$, $c_{ij}(z)=\left[2\kappa/d_j\right]_{z^{d_j}} v_{ij}(z)$, where $v_{ij}(z)$ is the $(i,j)$-th entry of the matrix 
\begin{equation*}
\mathbf{E}(z)=([a_{ij}]_{z^{d_i}})_{i,j\in \mbI}^{-1}
\end{equation*}
It is known that $c_{ij}(z)$ is an element of $\Z[z,z^{-1}]$ satisfying $c_{ij}(z^{-1})=c_{ij}(z)$; see (5.1) and Appendix A of \cite{GTL3}, for instance.

\item For each $i\in \mbI$, $\bt_i(u)$ is the formal Borel transform of $t_i(u)$. That is, 
\begin{equation*}
\bt_i(u)=\bt(t_i(u))=\hbar \sum_{r\geq 0}t_{i,r}\frac{u^r}{r!}.
\end{equation*}
\end{enumerate}

To conclude this subsection, we recall that the \textit{deformed Drinfeld coproduct} $\Delta_z^D$ on the Yangian $\Yhg$ (see \cite{GTLW19}*{\S3.3--3.4}) can be recovered as the composite
\begin{equation*}
\Delta_z^D=\mathrm{Ad}(\mcR^-(z))\circ (\tau_z\otimes \id)\circ \Delta : \Yhg \to \Yhg^{\otimes 2}(\!(z^{-1})\!).
\end{equation*}
Its restriction to $\Yhg[0]$ has image in $\Yhg[0]^{\otimes 2}[z]$, and is given explicitly by 
\begin{equation*}
\Delta_z^D(\xi_i(u))=\xi_i(u-z)\otimes \xi_i(u) \quad \forall \quad i\in \mbI. 
\end{equation*}
We shall let $\Delta^0$ denote the evaluation of this homomorphism at $z=0$:  
\begin{equation}\label{def:Delta^0}
\Delta^0:=\mathrm{ev}_{z=0}\circ \Delta_z^D|_{\Yhg[0]}: \Yhg[0]\to \Yhg[0]\otimes \Yhg[0],
\end{equation}
where $\mathrm{ev}_{z=0}:\Yhg[0]^{\otimes 2}[z]\to \Yhg[0]^{\otimes 2}$ is given by $z\mapsto 0$. This defines a Hopf algebra structure on $\Yhg[0]$, provided it is equipped with counit $\veps^0$ and antipode $S^0$ uniquely determined by 
\begin{equation*}
\veps^0(\xi_i(u))=1 \quad \text{ and }\quad S^0(\xi_i(u))=\xi_i(u)^{-1} \quad \forall \quad i \in \mbI.
\end{equation*}
We shall call this the \textit{Drinfeld Hopf algebra structure} on $\Yhg[0]$. 

\subsection{Representation theory}\label{ssec:Rep}
We now turn to recalling some basic facts about finite-dimensional representations of $\Yhg$.

A $\Yhg$-module $V$ is called a \textit{highest weight module} of highest weight $\underline{\lambda}=(\lambda_i(u))_{i\in \mbI}\in (1+u^{-1}\C[\![u^{-1}]\!])^\mbI$ if it is generated by a nonzero vector $v\in V$ satisfying  
\begin{equation*}
x_i^+(u)v=0 \quad \text{ and }\quad \xi_i(u)v=\lambda_i(u)v \quad \forall \quad i\in \mbI.
\end{equation*}
The vector $v$ is then unique up to scalar multiplication and called the highest weight vector of $V$. 
There is a unique, up to isomorphism, highest weight module associated to any element $\underline{\lambda}\in (1+u^{-1}\C[\![u^{-1}]\!])^\mbI$. 

In \cite{DrNew}, Drinfeld used the language of highest weight modules to classify the finite-dimensional irreducible representations of $\Yhg$:
\begin{theorem}[\cite{DrNew}]
Let $V$ be a finite-dimensional irreducible representation of $\Yhg$. Then $V$ is a highest weight module, and there is a unique $\mbI$-tuple of monic polynomials $\underline{\mathrm{P}}=(P_i(u))_{i\in \mbI}\in \C[u]^\mbI$ satisfying
\begin{equation*}
\xi_i(u)v =\frac{P_i(u+\hbar d_i)}{P_i(u)}v \quad \forall \quad i\in \mbI,
\end{equation*}
where $v\in V$ is any highest weight vector. Moreover, every $\mbI$-tuple of monic polynomials $\underline{\mathrm{P}}$ arises in this way, and uniquely determines the underlying representation  up to isomorphism.
\end{theorem}
The monic polynomials $P_i(u)$ appearing in the above theorem are referred to as the \textit{Drinfeld polynomials} associated to $V$. Henceforth, we will write $L(\underline{\mathrm{P}})$ for the unique, up to isomorphism, finite-dimensional irreducible representation of $\Yhg$ with tuple of Drinfeld polynomials $\underline{\mathrm{P}}=(P_i(u))_{i\in \mbI}$. 
Given $j\in \mbI$ and $a\in \C$, the \textit{$j$-th fundamental representation} $L_{\varpi_j}(a)$ of $\Yhg$ is the module $L(\underline{\mathrm{P}})$ where 
\begin{equation*}
P_i(u)=(u-a)^{\delta_{ij}} \quad \forall \quad i\in \mbI.
\end{equation*}
In the special case where $a=0$, we simply write $L_{\varpi_j}$ for $L_{\varpi_j}(0)$.

Since $U(\mfg)\subset \Yhg$, any finite-dimensional $\Yhg$-module $V$ admits a $\mfg$-weight space decomposition 
\begin{equation*}
V=\bigoplus_{\mu\in \Lambda}V_\mu, \quad \text{ where }\quad V_\mu=\{v\in V: h v=\mu(h) v\quad \forall\; h\in \mfh\}.
\end{equation*}
If $V$ is a highest weight module with the highest weight $\underline{\lambda}=(\lambda_i(u))_{i\in \mbI}$, then the $\g$-weight $\lambda\in \Lambda$ of any highest weight vector is given by the formula $\lambda=\sum_{i\in \mbI} d_i^{-1}\lambda_{i,0}\varpi_i$, where $\lambda_i(u)=1+\hbar \sum_{r\geq 0}\lambda_{i,r} u^{-r-1}$ for each $i\in \mbI$, and the weight space $V_\lambda$ is one-dimensional.  In particular, if $V\cong L(\underline{\mathrm{P}})$, then 
\begin{equation*}
\lambda=\sum_{i\in \mbI} \deg(P_i(u))\varpi_i \in \Lambda. 
\end{equation*}
Observe that this implies that the $\g$-weight of any highest weight vector in $L_{\varpi_j}(a)$ is indeed equal to the $j$-th fundamental weight $\lambda=\varpi_j$ of $\mfg$. 

The generating series $\xi_i(u)$ and $x_i^\pm(u)$ introduced in Section \ref{ssec:GKLO} operate on any finite-dimensional representation $V$ of $\Yhg$ as the expansions at infinity of $\End(V)$-valued rational functions of $u$; see \cite{GTL2}*{Prop.~3.6}.  The joint sets of poles of these operators encode important information about the structure of $V$ and merit their own definition:
\begin{definition}\label{D:Poles}
Let $V$ be a finite-dimensional representation of $\Yhg$.
Then, for each $i\in \mbI$, the \textbf{$i$-th set of poles} of $V$ is the subset
\begin{equation*}\label{def:Poles}
\Poles_i(V):=\{\text{Poles of }\; \xi_i(u)|_V,x_i^\pm(u)|_V\in \End(V)(u)\}\subset \C.
\end{equation*}
\end{definition}
These sets were the main object of study in \cite{GWPoles}. They are determined by their values on the composition factors of $V$, and the sets $\Poles_i(L(\underline{\mathrm{P}}))$ were computed explicitly in terms of the the roots of the Drinfeld polynomials $P_i(u)$ and the inverse, normalized $q$-Cartan matrix $\mathbf{E}(z)$ of $\mfg$ in Theorem 5.2 of \cite{GWPoles}. These results, and their connection with the so-called Baxter polynomials of $L(\underline{\mathrm{P}})$, will be recalled in more detail in Section \ref{ssec:Baxter-poles}.

{ 
\newcommand{\ad}{\operatorname{ad}} 
\newcommand{\db}[1]{[\![#1]\!]} 

\section{Braid group actions}\label{sec:Braid}

In this section we will define an action of the braid group $\Bg$ corresponding to $\g$ on the algebra $\Yhg[0]$ by Hopf algebra automorphisms; see Sections \ref{ssec:Bg} and \ref{ssec:Bg-modified}.  We will then establish various basic properties of this action, including explicit formulas for the action on generators (Section  \ref{ssec:Bg-Y0-formulas}), a relation to Hecke algebras (Section \ref{ssec:Hecke}), and finally its interactions with the universal $R$-matrix (Section \ref{ssec:action-R-matrix}).

\subsection{The braid group}\label{ssec:Bg}
Denote the Weyl group of $\g$ by $\Wg$, generated by the simple reflections $s_i$ for $i\in\mbI$.  Recall that the \textbf{braid group} $\Bg$ associated to $\g$ is the group with generators $\braid_i$ for $i\in\mbI$, and the following defining relations: for all $i, j\in \mbI$ with $i\neq j$, 
\begin{equation}
\label{eq: braid rels}
\underbrace{\braid_i \braid_j \braid_i \cdots}_{m_{ij}\text{ factors}}  = \underbrace{\braid_j \braid_i \braid_j\cdots}_{m_{ij}\text{ factors}}
\end{equation}
Here $m_{ij} = m_{ji}$ is defined according to

\begin{center}
\begin{tabular}{c|cccc}
$a_{ij} a_{ji}$ & 0 & 1 & 2 & 3 \\
\hline
$m_{ij}$ & 2 & 3 & 4 & 6
\end{tabular}
\end{center}

There is a surjective map $\Bg \rightarrow \Wg$ onto the Weyl group defined by $\braid_i \mapsto s_i$, whose kernel is generated by the elements $\braid_i^2$. There is also a section $w\mapsto \braid_w$ of this surjection, defined by taking  any reduced expression $w = s_{i_1} \cdots s_{i_\ell}$ and setting
\begin{equation*}
\braid_w = \braid_{i_1} \cdots \braid_{i_\ell}.
\end{equation*}
This is independent of the choice of reduced expression, and these elements satisfy $\braid_{vw} = \braid_v \braid_w$ whenever the lengths $\ell(vw) = \ell(v) + \ell(w)$ add; see \cite{Lusztig}*{\S 2.1.2}, for instance.

Let $(V, \phi)$ be an integrable representation of $\g$.  Recall that this means that $V = \bigoplus_{\mu \in \Lambda} V_\mu$ breaks into weight spaces labelled by integral weights, and that $e_i, f_i$ act locally nilpotently on $V$.  Define operators $\braid_i^V$ on $V$ by:
\begin{equation*}
\braid_i^V \ = \ \exp\big(\phi(e_i)\big)\exp\big(\phi(-f_i)\big)\exp\big(\phi(e_i)\big).
\end{equation*}
These operators are functorial in $V$, in the natural sense.  Note that since $\phi(e_i)$ and $\phi(-f_i)$ are locally nilpotent endomorphisms of $V$, each $\braid^V_i$ is an automorphism of $V$.

The following proposition summarizes several well-known, notable properties of the operators $\braid_i^V$.  We refer the reader to \cite{Kumar}*{\S 1.3} or \cite{KacBook90}*{\S 3}, for example, for further details. 
\begin{proposition}\label{P:tau-wt-space}
\mbox{}
\begin{enumerate}[label=(\alph*)]\setlength{\itemsep}{3pt}
\item\label{braid:a} The map $\braid_i \mapsto \braid_i^V$ defines an action of $\Bg$ on $V$. In other words, the operators $\braid_i^V$ satisfy the braid relations (\ref{eq: braid rels}).

\item\label{braid:b} For any weight $\mu \in \Lambda$, $\braid_i^V$ intertwines the $\mu$ and $s_i(\mu)$ weight spaces of $V$:
\begin{equation*}
\braid^V_i(V_\mu) = V_{s_i(\mu)}.
\end{equation*}
Moreover, for any $v \in V_\mu$ we have $(\braid_i^V)^2(v) = (-1)^{\langle \mu, \alpha_i^\vee\rangle} v$.

\item\label{braid:c} If $V,W$ are both integrable representations of $\g$, then
\begin{equation*}
\braid_i^{V \otimes W} = \braid_i^V \otimes \braid_i^W.
\end{equation*}

\item\label{braid:d} If there is an algebra map $U(\g) \rightarrow A$, and the corresponding adjoint action of $\g$ on $A$ is integrable, then the maps $\braid_i^A$ are algebra automorphisms.

\item\label{braid:e} Let $A$ be as in part (d) above, and $V$ a module for $A$. If the $\g$ actions on $A$ and $V$ are both integrable, then
\begin{equation*}
\braid_i^V( a \cdot v) = \braid_i^A(a) \cdot \braid_i^V(v)
\end{equation*}
for any $a \in A$ and $v \in V$.
\end{enumerate}
\end{proposition}

\subsection{Modified braid group operators}\label{ssec:Bg-modified}
Recall from Section \ref{ssec:Yhg} that there is an embedding $U(\g) \hookrightarrow \Yhg$, and in particular an adjoint action of $\g$ on $\Yhg$. It is well-known that this action defines an integrable representation of $\g$:
\begin{lemma}
    The action of $\g $ on $\Yhg$ is integrable.
\end{lemma}
Consequently, we can consider the algebra automorphisms $\braid_i^{\Yhg}$ as in the previous section for $i \in \mbI$. For simplicity, we denote these elements by $\braid_i$. By Proposition \ref{P:tau-wt-space}, they define an action of the braid group $\Bg$ on $\Yhg$.

Consider the subsets $N^\pm = \{ x_{i,r}^\pm : i \in \mbI, r \in \Z_{\geq 0} \}$, and the corresponding left ideals $\Yhg N^\pm$ and right ideals $N^\pm \Yhg$.  The PBW Theorem implies that
\begin{equation*}
    \Yhg \ = \ \Yhg[0] \oplus( N^- \Yhg + \Yhg N^+).
\end{equation*}
Consider the corresponding projection $\Pi$ onto the direct summand $\Yhg[0]$:
\begin{equation}
\label{eq: Pi for Yangian}
    \Pi: \Yhg \longrightarrow \Yhg[0].
\end{equation}
If $\beta \in Q$ is any non-zero weight, then the corresponding weight space $\Yhg_\beta$ is in the kernel of $\Pi$.  Meanwhile, the restriction of $\Pi$ to the zero weight space $\Yhg_0$ coincides with the projection along the direct sum
\begin{equation*}
    \Yhg_0 \ = \ \Yhg[0] \oplus (\Yhg_0\cap N^- \Yhg \cap \Yhg N^+).
\end{equation*}
This restriction is an algebra homomorphism, since  $\Yhg_0\cap N^- \Yhg \cap \Yhg N^+$ is an ideal in the ring $\Yhg_0$.
\begin{definition}\label{D:mbraid}
    The \textbf{modified braid group operators} are the maps $$\mbraid_i : \Yhg[0] \longrightarrow \Yhg[0] \quad \text{ for } i \in \mbI$$ defined by the composition $\mbraid_i = \Pi \circ \braid_i\Big|_{\Yhg[0]}$. 
\end{definition}
\begin{remark}
The operators $\braid_i\big|_{\Yhg_0}$ have order two by Proposition \ref{P:tau-wt-space}(b), and in particular define an action of the Weyl group $\Wg$ on the zero weight space $\Yhg_0$.  However, the modified operators $\mbraid_i$ have \emph{infinite} order, as may be seen from the explicit formulas in Section \ref{ssec:Bg-Y0-formulas} below.  
\end{remark}
We will also consider the ``opposite'' linear projection defined similarly to (\ref{eq: Pi for Yangian}), but where we exchange the roles of $N^\pm$:
\begin{equation*}
\Pi^{op}: \Yhg = \Yhg[0]  \oplus ( N^+ \Yhg + \Yhg N^-) \longrightarrow \Yhg[0].
\end{equation*}
Our goal in the rest of this section will be to  prove the following theorem:

\begin{theorem}\label{T:mbraid}
    The operators $\mbraid_i$ define an action of the braid group $\Bg$ on $\Yhg[0]$ by Hopf algebra automorphisms.  The inverse of $\mbraid_i$ is given by the opposite modified braid group operator $ \mbraid_i^{-1} = \Pi^{op} \circ \braid_i \Big|_{\Yhg[0]}$.
\end{theorem}
Here it is understood that $\Yhg[0]$ is equipped with its Drinfeld Hopf algebra structure, as defined in Section \ref{sec:Hopf} just below equation \eqref{def:Delta^0}.
We will prove this theorem in a sequence of results. First we will prove in Lemma \ref{L:mbraid-reduced} that, for each $w\in \Wg$, the modified analogue $\mbraid_w$ of the operator $\braid_w$ (see \eqref{mbraid_w}) is independent of the choice of reduced expression for $w$. This will allow us to conclude in Corollary \ref{cor: braid rels} that the operators $\mbraid_i$ indeed satisfy the defining braid relations of $\Bg$.  Afterwards in Lemma \ref{lemma: braid aut Yangian}, we will show that each $\mbraid_i$ is a Hopf algebra automorphism of $\Yhg[0]$, and prove that its inverse is as claimed.  This will complete the proof of the theorem. 

Recall that every $w\in \Wg$ has a lift $\braid_w \in \Bg$, and therefore a corresponding automorphism of $\Yhg$ which we also denote by $\braid_w$.
\begin{lemma}\label{L:mbraid-reduced}
    Let $w = s_{i_1}\cdots s_{i_\ell}$ be a reduced expression.  Then
    \begin{equation*}
    \Pi \circ \braid_w\Big|_{\Yhg[0]} = \mbraid_{i_1} \cdots \mbraid_{i_\ell}.
    \end{equation*}
    In particular, the right-hand side of this equation is independent of the choice of reduced expression.
\end{lemma}
For each $w\in\Wg$, we denote the element defined by the lemma by 
\begin{equation}\label{mbraid_w}
    \mbraid_w := \Pi \circ \braid_w\Big|_{\Yhg[0]} =  \mbraid_{i_1} \cdots \mbraid_{i_\ell}.
\end{equation}
The following proof is essentially taken from \cite{WeekesThesis}*{Thm.~5.3.19} in the second author's PhD thesis, though the proof there contains a mistake: it only includes $n=1$!
\begin{proof}
    We begin by noting that for any $a \in \Yhg[0]$ we have
    \begin{equation}
    \label{eq: braid estimate}
        \braid_i(a) = \mbraid_i(a) + \sum_{n \geq 1} \Yhg_{-n\alpha_i} \Yhg_{n \alpha_i}.
    \end{equation}
    This follows from the construction of $\braid_i$, via the adjoint action of the elements $x_{i,0}^\pm$. Next, recall that for $\beta \in Q$ we have $\braid_i(\Yhg_\beta) = \Yhg_{s_i \beta}$. Using this and (\ref{eq: braid estimate}) it follows by induction on $\ell$ that
    \begin{equation*}
    \braid_{i_1} \cdots \braid_{i_\ell}(a) = \mbraid_{i_1}\cdots \mbraid_{i_\ell}(a) + \sum_{b=1}^\ell \sum_{n \geq 1} \Yhg_{-n s_{i_1}\cdots s_{i_{b-1}}\alpha_{i_b}} \Yhg_{ns_{i_1}\cdots s_{i_{b-1}} \alpha_{i_b}}.
    \end{equation*}
    If  $w = s_{i_1} \cdots s_{i_\ell}$ is a reduced expression, then the elements $s_{i_1} \cdots s_{i_{b-1}} \alpha_{i_b}$ are positive roots: they are precisely the elements  of the inversion set of $w$  \cite{Kumar}*{Lem.~1.3.14}.  But for any positive element $\beta \in Q^+$, the product $\Yhg_{-\beta} \Yhg_\beta$  is  killed by $\Pi$. This proves the claim.
\end{proof}

Since both sides of any braid relation (\ref{eq: braid rels}) are reduced expressions, we conclude from the previous lemma that:
\begin{corollary}
\label{cor: braid rels}
    The operators $\mbraid_i$ satisfy the braid relations (\ref{eq: braid rels}).
\end{corollary}

The next lemma will complete the proof of Theorem \ref{T:mbraid}:
\begin{lemma}
\label{lemma: braid aut Yangian}
     Each $\mbraid_i$ is a Hopf algebra automorphism of $\Yhg[0]$, with inverse given by $\mbraid^{-1}_i = \Pi^{op} \circ \braid_i \Big|_{\Yhg[0]}$.
\end{lemma}
\begin{proof}
Since $\Yhg_0$ is preserved by $\braid_i$, $\Yhg[0]$ is a subset of $\Yhg_0$,  and $\Pi|_{\Yhg_0}$ is an algebra homomorphism, we see that $\mbraid_i$ is an algebra homomorphism.  The fact that $\mbraid_i$ is a coalgebra homomorphism commuting with the antipode $S^0$ of $\Yhg[0]$ follows immediately from the formulas of Proposition \ref{P:tau-dot(A)} or Corollary \ref{C:mbraid-t-xi} established in the next section, using that $\xi_i(u)$ and $A_i(u)$ are grouplike series; see above  \eqref{def:Delta^0}. 

Next, we temporarily adopt the notation $\mbraid^{op}_i = \Pi^{op}\circ \braid_i\big|_{\Yhg[0]}$.   We will show that $\mbraid^{op}_i$ is the inverse of $\mbraid_i$, which will complete the proof of the lemma. First, we apply $\braid_i$ to both sides of equation (\ref{eq: braid estimate}), using the fact that $\braid_i( \Yhg_{\pm n \alpha_i}) = \Yhg_{\mp n \alpha_i}$:
$$
\braid^2_i(a) = \braid_i( \mbraid_i(a)) +  \sum_{n \geq 1} \Yhg_{n\alpha_i} \Yhg_{-n \alpha_i}
$$
Since $\mbraid_i(a) \in \Yhg[0]$, we also have the following ``opposite'' analogue of (\ref{eq: braid estimate}):
$$
\braid_i( \mbraid_i(a)) = \mbraid^{op}_i ( \mbraid_i(a)) + \sum_{n \geq 1} \Yhg_{n\alpha_i} \Yhg_{-n \alpha_i}
$$
Combined with the fact that $\braid_i^2(a) = a$ by Proposition \ref{P:tau-wt-space}(b), we thus obtain:
$$
a = \braid^2_i(a)  = \mbraid^{op}_i ( \mbraid_i(a)) + \sum_{n \geq 1} \Yhg_{n\alpha_i} \Yhg_{-n \alpha_i}
$$
Now, observe that $a$ and $\mbraid^{op}_i(\mbraid_i(a))$ are both elements of $\Yhg[0]$, and we have just shown that their difference lies in $\sum_{n \geq 1} \Yhg_{n\alpha_i} \Yhg_{-n \alpha_i}$.  By the PBW theorem, we conclude that $a =\mbraid^{op}_i(\mbraid_i(a))$, and therefore $\mbraid^{op}_i \circ \mbraid_i$ is the identity operator on $\Yhg[0]$.  A symmetric argument shows that $\mbraid_i \circ \mbraid^{op}_i$ is also the identity, and thus $\mbraid_i$ is invertible with inverse $\mbraid^{op}_i$,  as claimed. \qedhere

\end{proof}

\begin{remark}
One can show that $\mbraid_i = \Pi \circ \braid_i|_{\Yhg[0]}$ is a filtered map of degree 0, and that the associated graded map $\mathrm{gr}(\mbraid_i)$ acts on $\mathrm{gr}\Yhg[0]\cong U(\mfh[t])$ via the simple reflection $s_i \in \Wg$.  In other words, we have:
$$
\mathrm{gr}(\mbraid_i): h t^k \mapsto s_i(h)t^k
$$
for any  $h \in \mfh$ and integer $ k \geq 0$.  Since $\mathrm{gr}(\mbraid_i)$ is invertible, this provides another proof of the fact that $\mbraid_i$ is invertible.
\end{remark}

\subsection{Computing $\mbraid_i$ on generators}\label{ssec:Bg-Y0-formulas}

To compute the action of the modified braid group operators $\mbraid_i$ on $\Yhg[0]$ explicitly, we will use the generating series $(A_j(u))_{j\in \mbI}$ of $\Yhg[0]$ from \cite{GKLO}, whose definition is recalled in Section \ref{ssec:GKLO}. The first assertion of the following proposition was established for simply-laced types in  \cite{WeekesThesis}*{Lem.~5.3.16} using a different argument.
\begin{proposition} \label{P:tau-dot(A)}
For each $i\in \mbI$, we have  
\begin{equation*}
\braid_i(A_j(u))
=
\begin{cases} 
D_i(u) \quad \text{ if } \; j=i,\\
A_j(u) \quad \text{ if }\; j\neq i.
\end{cases}
\end{equation*}
Consequently, the action of $\mbraid_i$ on $\Yhg[0]$ is uniquely determined by the formula 
\begin{equation*}
\mbraid_i(A_j(u))=A_j(u)\xi_i(u)^{\delta_{ij}} \quad \forall \quad j\in \mbI. 
\end{equation*}
\end{proposition}
\begin{proof}
By Part \eqref{[g,GKLO]:1} of Corollary \ref{C:[g,GKLO]}, $\braid_j(A_i(u))=A_i(u)$ if $j\neq i$. Moreover, using the 
relations of Parts \eqref{[g,GKLO]:2} and \eqref{[g,GKLO]:3} of that corollary, we obtain 
\begin{align*}
\braid_i(A_i(u))
&=\exp(\mathrm{ad}(e_i))\exp(-\mathrm{ad}(f_i))\exp(\mathrm{ad}(e_i))(A_i(u))\\
&=\exp(\mathrm{ad}(e_i))\exp(-\mathrm{ad}(f_i))(A_i(u)+B_i(u))\\
&
=\exp(\mathrm{ad}(e_i))(A_i(u)+B_i(u)-(A_i(u)-C_i(u)-D_i(u))-C_i(u))\\
&
=\exp(\mathrm{ad}(e_i))(B_i(u)+D_i(u))\\
&
=B_i(u)+D_i(u)-B_i(u)=D_i(u),
\end{align*}
which completes the proof of the first statement of the proposition. The stated formula for $\mbraid_i(A_j(u))$ now follows from the definition of $\mbraid_i$ (namely, $\mbraid_i=\Pi\circ \braid_i|_{\Yhg[0]}$) and that 
\begin{equation*}
\Pi(D_i(u)) = \Pi(A_i(u)\xi_i(u) + C_i(u)A_i^{-1}(u)B_i(u))=A_i(u)\xi_i(u). \qedhere
\end{equation*}
\end{proof}

Next, recall from Remark \ref{R:logs} that, for each $i\in \mbI$, $t_i(u)\in u^{-1}\Yhg[0][\![u^{-1}]\!]$ is the formal series logarithm of $\xi_i(u)$: $t_i(u)=\log\xi_i(u)$. We further recall that $\qshift=e^{\frac{\hbar}{2}\partial_u}$, viewed as an operator on $\Yhg[0][\![u^{-1}]\!]$. 
\begin{corollary}\label{C:mbraid-t-xi}
Let $i,j\in \mbI$. Then the action of $\mbraid_i$ on $t_j(u)$ and $\xi_j(u)$ is given by 
\begin{align*}
\mbraid_i(t_j(u))&=t_j(u)-\qshift^{-d_i}[a_{ij}]_{\qshift^{d_i}}(t_i(u)),
\\
\mbraid_i(\xi_j(u))
&=\xi_j(u)\prod_{k=0}^{|a_{ij}|-1} \xi_i\!\left(u-\frac{\hbar d_i}{2}(|a_{ij}|-2k)\right)^{(-1)^{\delta_{ij}}}.
\end{align*}
Moreover, the inverse of $\mbraid_i$ is determined by the three equivalent formulas
\begin{align*}
\mbraid^{-1}_i( A_j(u)) & = A_j(u) \xi_i(u+\hbar d_i)^{\delta_{ij}}, \\
\mbraid^{-1}_i( t_j(u) )& = t_j(u) - \qshift^{d_i} [a_{ij}]_{\qshift^{d_i}}( t_i(u)), \\
\mbraid^{-1}_i( \xi_j(u) ) & = \xi_j(u) \prod_{k=1}^{|a_{ij}|} \xi_i\left( u -\frac{\hbar d_i}{2}(|a_{ij}|-2k) \right)^{(-1)^{\delta_{ij}}}.
\end{align*}
\end{corollary}
\begin{proof}
 By Proposition \ref{P:tau-dot(A)}, one has $\mbraid_i(a_j(u))=a_j(u)+\delta_{ij}t_i(u)$ for each $i,j \in \mbI$, where we recall from Remark \ref{R:logs} that $a_j(u)=\log A_j(u)$. The stated formula for $\mbraid_i(t_j(u))$ is a consequence of this identity and the relation \eqref{t-vs-A}. One then obtains the given formula for $\mbraid_i(\xi_j(u))$ by exponentiating and using that 
 \begin{equation}\label{qshift-poly}
-\qshift^{-d_i}[a_{ij}]_{\qshift^{d_i}}=(-1)^{\delta_{ij}}\sum_{k=0}^{|a_{ij}|-1}\qshift^{-d_i(|a_{ij}|-2k)}. 
 \end{equation}
 
We next solve for $\mbraid_i^{-1}$. Start with the formula above for $\mbraid_i( t_i(u))$:
 $$
 \mbraid_i( t_i(u)) = t_i(u) - \qshift^{-d_i}[2]_{\qshift^{d_i}} (t_i(u)) = - \qshift^{-2d_i}( t_i(u)).
 $$
 Note that $\mbraid_i^{\pm 1}$ commute with $\qshift$.  Thus, applying $\mbraid_i^{-1}$ to both sides gives
 $$\mbraid_i^{-1}( t_i(u)) = - \qshift^{2d_i} ( t_i(u)) = t_i(u) - \qshift^{d_i} [2]_{\qshift^{d_i}}( t_i(u)).$$
Next, applying $\mbraid_i^{-1}$ to both sides of our formula for $\mbraid_i( t_j(u))$, we obtain
\begin{align*}
t_j(u) &= \mbraid_i^{-1}( t_j(u)) - \qshift^{-d_i} [a_{ij}]_{\qshift^{d_i}} ( \mbraid_i^{-1}( t_i(u))) \\
& = \mbraid_i^{-1} (t_j(u)) + \qshift^{d_i} [a_{ij}]_{\qshift^{d_i}}( t_i(u)).
\end{align*}
Rearranging, we obtain the claimed formula for $\mbraid_i^{-1}(t_j(u))$.  Once again, the formulas for $\mbraid_i^{-1}( \xi_j(u))$ and $\mbraid_i^{-1}( A_j(u))$ follow by exponentiating.
\end{proof}
\begin{remark}
Using the relations in Proposition \ref{prop: GKLO relations}, one can show that 
$$
D_i(u) = A_i(u) \xi_i(u+\hbar d_i) + B_i(u) A_i(u)^{-1} C_i(u)
$$
This gives an alternative way to calculate $\mbraid_i^{-1}( A_j(u))$ via the technique in the proof of Proposition \ref{P:tau-dot(A)}, since $\mbraid_i^{-1} = \Pi^{op} \circ \braid_i \big|_{\Yhg[0]}$ and
$$
\Pi^{op} (D_i(u)) = \Pi^{op}\big( A_i(u) \xi_i(u+\hbar d_i) + B_i(u) A_i(u)^{-1} C_i(u) \big) = A_i(u) \xi_i(u+\hbar d_i)
$$
\end{remark}

\subsection{Action of the Hecke algebra of type $\g$}\label{ssec:Hecke}
For the remainder of Section \ref{sec:Braid}, we shift our focus towards establishing auxiliary properties of the braid group operators $\braid_i$ and $\mbraid_i$ which, in particular, yield proofs of Parts \eqref{I.3} and \eqref{I.4} of Theorem \ref{T:I}.  
In this section, we show that the $\Bg$-action on $\Yhg[0]$ from Theorem \ref{T:mbraid} gives rise to an action of the Hecke algebra associated to $\g$, with parameters encoded by the symmetrizing integers  $(d_j)_{j\in \mbI}$, on a deformation of the space $\mfh[t]$. We begin by recalling the definition and basic properties of the relevant Hecke algebra, following \cite{Lusztig-Hecke}.
\begin{definition}\label{D:Hecke}
The \textbf{Iwahori--Hecke algebra} $\mathscr{H}_z(\mfg)$  is the unital associative $\C[z,z^{-1}]$-algebra generated by elements $\{T_i\}_{i\in \mbI}$, subject to the relations
\begin{equation*}
\underbrace{T_i T_j T_i \cdots}_{m_{ij}\text{ factors}} = \underbrace{T_j T_iT_j \cdots}_{m_{ij} \text{ factors}} \qquad \text{ and }\qquad (T_i-z_i)(T_i+z_i^{-1})=0,
\end{equation*}
for all $i,j\in \mbI$ with $j\neq i$, where $z_i=z^{d_i}$ and $m_{ij}$ is as in \eqref{eq: braid rels}.
\end{definition}
The second set of relations appearing in the definition of $\mathscr{H}_z(\mfg)$ are called the \textit{Hecke relations}; note that they immediately imply that each generator $T_i$ is invertible. Thus, there is a surjective $\C[z,z^{-1}]$-algebra homomorphism 
\begin{equation*}
\pi_\mathscr{H}:\C[\Bg]\otimes \C[z,z^{-1}]\onto \mathscr{H}_z(\mfg), \quad \braid_i\mapsto T_i \quad \forall \quad i\in \mbI, 
\end{equation*}
with kernel generated by the Hecke relations as an ideal. 

The Hecke algebra $\mathscr{H}_z(\mfg)$ is a flat deformation of the group algebra $\C[\Wg]$ over the ring $\C[z,z^{-1}]$: the quotient of $\mathscr{H}_z(\mfg)$ by the ideal $(z-1)\mathscr{H}_z(\mfg)$ coincides with the group algebra $\C[\Wg]$. The flatness is due to the well-known fact that $\mathscr{H}_z(\mfg)$ is a free module over $\C[z,z^{-1}]$ with basis given by $\{T_w\}_{  w\in \Wg}$, where $T_w:=\pi_\mathscr{H}(\braid_w)$. We refer the reader to \cite{Lusztig-Hecke}*{\S3}, for instance, for further details. 

Let us now introduce the vector space $\mbbV=\bigcup_{r\geq 0}\mbbV_r \subset \Yhg[0]$, where 
\begin{equation*}
\mbbV_r= \bigoplus_{k=0}^r\mathrm{span}\{t_{j,k}:j\in \mbI\} \quad \forall \quad r\geq 0.
\end{equation*}
Let $\msz:=\tau_{-\frac{\hbar}{2}}|_{\mbbV}\in \End(\mbbV)$, where we recall that, for each $a\in \C$, $\tau_a$ is the shift automorphism of $\Yhg$ determined by $\tau_a(y(u))=y(u-a)$ for $y\in \{x_i^\pm,\xi_i\}$. Note that $\msz$ is invertible with inverse $\tau_{\frac{\hbar}{2}}|_{\mbbV}$. In what follows, we shall view $\mbbV$ as a $\C[z,z^{-1}]$-module via the inclusion  $\C[z,z^{-1}]\into \mathrm{GL}(\mbbV)$ sending $z$ to $\msz$. 
\begin{proposition}\label{P:Hecke}
The assignment $T_i\mapsto \msz^{d_i}\circ \mbraid_i\Big|_{\mbbV}$ for each $i\in \mbI$ defines a $\mathscr{H}_z(\g)$-module structure on $\mbbV$ for which $\mbbV_r$ is a submodule for any $r\geq 0$. 
\end{proposition}
\begin{proof}
As  $\qshift^d(t_i(u))=t_i(u+d\hbar/2)=\msz^{d}(t_i(u))$ for any $d\in \Z$ and $i\in \mbI$, we can rewrite the first formula of Corollary \ref{C:mbraid-t-xi} as
\begin{equation}\label{mbraid-tjk}
\mbraid_i(t_{j,k})=t_{j,k}-\msz^{-d_i}[a_{ij}]_{\msz^{d_i}}(t_{i,k}) \quad \forall \quad i,j\in \mbI, \; k\geq 0,
\end{equation}
where we recall that $\msz^{-d_i}[a_{ij}]_{\msz^{d_i}}$ is a Laurent polynomial in $\msz$; see \eqref{qshift-poly}. It follows that the operators $\mbraid_i$ restrict to linear automorphisms of  $\mbbV$ which satisfy the defining braid relations of $\Bg$. 

The formulas of Corollary \ref{C:mbraid-t-xi} also imply that $\mbraid_i\circ \tau_a=\tau_a\circ \mbraid_i$ for each $i\in \mbI$ and $a\in \C$. Thus, $\mbraid_i$ and $\msz$ commute, and the assignment $\braid_i\mapsto \msz^{d_i}\circ \mbraid_i|_{\mbbV}$ defines an algebra homomorphism
\begin{equation*}
\C[\Bg]\otimes \C[z,z^{-1}]\to \End_{\C[z,z^{-1}]}(\mbbV).
\end{equation*}

Hence, to see that that $\mbbV$ is a $\mathscr{H}_z(\g)$-module, we are left to show that the operators $\msz^{d_i}\circ \mbraid_i|_{\mbbV}$ satisfy the Hecke relations. Equivalently, we must prove that the following identity holds on $\mbbV$,
for each $i\in \mbI$:
\begin{equation}\label{Hecke-rels}
\mbraid_i^2|_{\mbbV}+(\msz^{-2d_i}-\mathrm{Id}_\mbbV)\circ \mbraid_i|_{\mbbV}=\msz^{-2d_i}\circ \mathrm{Id}_\mbbV.
\end{equation}
To see this, first note that we have
\begin{gather*}
\begin{aligned}
\mbraid_i^2(t_{j}(u))
&=\mbraid_i(t_{j}(u))-\msz^{-d_i}[a_{ij}]_{\msz^{d_i}}(\mbraid_i(t_{i}(u)))\\
&
=t_{j}(u)+(\msz^{-d_i}[2]_{\msz^{d_i}}\msz^{-d_i}[a_{ij}]_{\msz^{d_i}}-2\msz^{-d_i}[a_{ij}]_{\msz^{d_i}})(t_{i}(u)),
\end{aligned}
\\[5pt]
\begin{aligned}
(\msz^{-2d_i}-\mathrm{Id}_\mbbV)(\mbraid_i(t_j(u)))
&
=(\msz^{-2d_i}-\mathrm{Id}_\mbbV)(t_{j}(u)-\msz^{-d_i}[a_{ij}]_{\msz^{d_i}}(t_{i}(u)))\\
&
=(\msz^{-2d_i}-\mathrm{Id}_\mbbV)(t_{j}(u))-(\msz^{-2d_i}-1)\msz^{-d_i}[a_{ij}]_{\msz^{d_i}}(t_{i}(u))).
\end{aligned}
\end{gather*}
Adding together these two expressions yields 
\begin{align*}
(\mbraid_i^2|_{\mbbV}+&(\msz^{-2d_i}-\mathrm{Id}_\mbbV)\circ \mbraid_i|_{\mbbV})(t_j(u))\\
&=\msz^{-2d_i}(t_{j}(u))+(\msz^{-d_i}[2]_{\msz^{d_i}}-\msz^{-2d_i}-1)\msz^{-d_i}[a_{ij}]_{\msz^{d_i}}(t_{i}(u))\\
&=\msz^{-2d_i}(t_{j}(u)),
\end{align*}
where in the last equality we have used that $[2]_{\msz^{d_i}}=\msz^{d_i}+\msz^{-d_i}$. This completes the proof of \eqref{Hecke-rels}, and thus establishes the first assertion of the proposition. That $\mbbV_r$ is a $\mathscr{H}_z(\g)$-submodule of $\mbbV$ for any $r\geq 0$ now follows immediately from the fact that it is preserved by $\msz$ and $\mbraid_i$ for any $i\in \mbI$; see \eqref{mbraid-tjk}. \qedhere
\end{proof}

\subsection{The universal $R$-matrix and lifted minors}\label{ssec:action-R-matrix}
We now turn to studying how the operators $\braid_i$ and $\mbraid_i$ interact with the universal $R$-matrix of the Yangian as well as the so-called lifted minors (or matrix coefficients) it defines on any finite-dimensional representation. 
 
Since the shift homomorphism $\tau_z:\Yhg\to \Yhg[][z]$ restricts to the identity map on $U(\g)$, the  intertwiner equation \eqref{R-inter} implies that $\mcR(z)$ is a $\g$-invariant element of $\Yhg^{\otimes 2}[\![z^{-1}]\!]$. That is, one has
\begin{equation*}
[\Delta(x), \mcR(z)]
    = 0 \quad \forall \quad x\in U(\g).
\end{equation*} 
It follows that $\mcR(z)$ is fixed by $\exp(\ad(x))\otimes \exp(\ad(x))$ for any $x\in \g$, and thus by $\braid_i\otimes\braid_i$ for each $i\in \mbI$. This observation proves the first part of the below proposition. 
\begin{proposition}\label{P:tau(R)}
For each $i\in \mbI$, we have 
\begin{equation*}
(\braid_i\otimes\braid_i)(\mcR(z))=\mcR(z) \quad \text{ and }\quad (\mbraid_i\otimes\mbraid_i)(\mcR^0(z))=\mcR^0(z).
\end{equation*}
\end{proposition}
\begin{proof}
 It remains to prove that $\mcR^0(z)$ is fixed by $\mbraid_i\otimes \mbraid_i$ for each $i\in \mbI$. To this end, observe that the restriction of  $\Pi\otimes \Pi$ to the weight zero component $(\Yhg\otimes \Yhg)_0$ is an algebra homomorphism
 \begin{equation}\label{Pi(x)Pi}
 (\Pi\otimes \Pi)\Big|_{(\Yhg\otimes \Yhg)_0}:(\Yhg\otimes \Yhg)_0\to \Yhg[0]\otimes \Yhg[0].
 \end{equation}
Indeed, this follows from the observation above Definition \ref{D:mbraid} applied to the Yangian $Y_\hbar(\mfg\oplus\mfg)\cong \Yhg\otimes \Yhg$ associated to the semisimple Lie algebra $\mfg\oplus \mfg$.

Since the universal $R$-matrix $\mcR(z)$ and its components $\mcR^\pm(z)$ and $\mcR^0(z)$ have weight zero with $\Pi^{\otimes 2}(\mcR^\pm(z))=1$ and $\Pi^{\otimes 2}(\mcR^0(z))=\mcR^0(z)$, the above observation implies that 
\begin{align*}
(\Pi\otimes \Pi)(\mcR(z))
&=\Pi^{\otimes 2}(\mcR^+(z))\cdot \Pi^{\otimes 2}(\mcR^0(z))\cdot \Pi^{\otimes 2}(\mcR^-(z))=\mcR^0(z). 
\end{align*}
 Therefore, applying $\Pi\otimes \Pi$ to both sides of the identity $(\braid_i\otimes \braid_i)(\mcR(z))=\mcR(z)$ yields 
 \begin{equation*}
 (\Pi \circ \braid_i)^{\otimes 2} (\mcR(z))=\mcR^0(z).
 \end{equation*}
 We claim that the left-hand side of this relation  coincides with $(\mbraid_i\otimes\mbraid_i)(\mcR^0(z))$. To see this, note that the observation \eqref{Pi(x)Pi} implies that 
 \begin{equation*}
 (\Pi \circ \braid_i)^{\otimes 2} (\mcR(z))=(\Pi \circ \braid_i)^{\otimes 2} (\mcR^+(z))\cdot \mbraid_i^{\otimes 2} (\mcR^0(z))\cdot (\Pi \circ \braid_i)^{\otimes 2} (\mcR^-(z)).
 \end{equation*}
 Thus, we are left to see that $(\Pi \circ \braid_i)^{\otimes 2} (\mcR^\pm(z))=1$. This follows from the fact that $\Yhg_{\alpha}$ is in the kernel of $\Pi$ for any nonzero $\alpha\in Q$, and that
 \begin{equation*}
 \braid_i^{\otimes 2}(\mcR^\pm(z)-1)=\sum_{\beta>0}\braid_i^{\otimes 2}(\mcR^\pm_\beta(z))\subset \Yhg^{\otimes 2}[\![z^{-1}]\!],
 \end{equation*}
where the sum is taken over all nonzero $\beta\in Q^+$ and the coefficients of $ \braid_i^{\otimes 2}(\mcR^\pm_\beta(z))$ belong to $\Yhg_{\pm s_i(\beta)}\otimes \Yhg_{\mp s_i(\beta)}\subset \mathrm{Ker}(\Pi\otimes \Pi)$ for each $\beta$. 
\end{proof}

Given a finite-dimensional representation $V$ of $\Yhg$ with action homomorphism $\pi_V:\Yhg\to \End(V)$, we define $\mathscr{T}_V(u)$ and $\mathscr{T}_V^0(u)$ in $\End(V)\otimes \Yhg[][\![u^{-1}]\!]$ by 
\begin{equation*}
\mathscr{T}_V(u):=(\pi_V\otimes \mathrm{Id})(\mathscr{R}(-u)) \quad \text{ and }\quad 
\mathscr{T}_V^0(u):=(\pi_V\otimes \mathrm{Id})(\mathscr{R}^0(-u))
\end{equation*}
where we recall that $\mathscr{R}(u)=\mcR_{21}(u)$ (\textit{i.e.,} it is the universal $R$-matrix of $\Yhg$ as defined in \cite{Dr}*{Thm.~3}), and we have set $\mathscr{R}^0(u)=\mcR_{21}^0(u)$. 

\begin{remark}
Provided $V$ has a non-trivial composition factor, the coefficients of $\mathscr{T}_V(u)$ generate $\Yhg$ as an algebra, and $\Yhg$ can be rebuilt from $\mathscr{T}_V(u)$ via generators and relations using the so-called $R$-matrix formalism for constructing quantum groups; see \cite{Dr}*{Thm.~6} and \cite{WRTT}*{Thm.~6.2}.
\end{remark}
To each pair $(v,f)\in V\times V^\ast$, we may associate coefficients $\tminor_{f,v}(u)\in\Yhg[][\![u^{-1}]\!]$ of $\mathscr{T}_V(u)$ and $\tminor_{f,v}^0(u)\in \Yhg[0][\![u^{-1}]\!]$ of $\mathscr{T}_{V}^0(u)$ by setting
\begin{equation*}
\tminor_{f,v}(u):=(f\otimes \mathrm{Id})(\mathscr{T}_{V}(u)v) \quad \text{ and }\quad \tminor_{f,v}^0(u):=(f\otimes \mathrm{Id})(\mathscr{T}_{V}^0(u)v).
\end{equation*}
Following \cite{KTWWY}*{\S5.3} and \cite{WeekesThesis}*{\S4.3.2}, we call these elements \textit{lifted minors}. 
\begin{corollary}\label{C:minors}
Let $v\in V$ and $f\in V^\ast$. Then for each $i\in \mbI$, one has
\begin{equation*}
\braid_i(\tminor_{f,v}(u))=\tminor_{\braid_i(f),\braid_i(v)}(u).
\end{equation*}
Moreover, the adjoint action of any $x\in \g$ on $\tminor_{f,v}(u)$ is given by
\begin{equation*}
[x,\tminor_{f,v}(u)]=\tminor_{x\cdot f, v}(u)+\tminor_{f,x\cdot v}(u).
\end{equation*}
\end{corollary}
\begin{proof}
Since $\braid_i(f)=\braid_i^{V^\ast}(f)=f\circ (\braid_i^V)^{-1}$, the first assertion of the corollary is a consequence of the definition of $\tminor_{f,v}(u)$ and the two relations
\begin{gather*}
(\braid_i^V)^{-1} \left( \pi_V(x)\braid_i^V(v)\right)=\pi_V(\braid_i^{-1}(x))v,\\
(\braid_i^{-1}\otimes \mathrm{Id})(\mathscr{R}(u))=(\mathrm{Id}\otimes \braid_i)(\mathscr{R}(u)),
\end{gather*}
which are due to the last part of Proposition \ref{P:tau-wt-space} and the first equality of Proposition \ref{P:tau(R)}, 
respectively. 
Similarly, the relation $[x,\tminor_{f,v}(u)]=\tminor_{x\cdot f, v}(u)+\tminor_{f,x\cdot v}(u)$ is a consequence of the definition of $\tminor_{f,v}(u)$ and the $\g$-invariance of  $\mathscr{R}(u)$. \qedhere
\end{proof}
\begin{remark}\label{R:tminor-0}
Let $V$ be a finite-dimensional representation of $\Yhg$, let $v\in V$ be a highest weight vector, and let $f\in V^\ast$. Then, for each $w\in \Wg$, we have 
\begin{equation*}
\mbraid_w^{-1}\left(\tminor_{f,v}^0(u)\right)=\tminor_{\braid_w^{-1}(f),\braid_w^{-1}(v)}^0(u).
\end{equation*} 
This is proven similarly to the first assertion of the previous corollary, using the second identity of Proposition \ref{P:tau(R)}.
\end{remark}

One application of Corollary \ref{C:minors} is that it can be used to identify each of the generating series $A_j(u),B_j(u),C_j(u)$ and $D_j(u)$ of $\Yhg$ (see Section \ref{ssec:GKLO}) with lifted minors associated to the $j$-th fundamental representation $L_{\varpi_j}$, thus yielding an alternative proof of \cite{Ilin-Rybnikov-19}*{Prop.~2.29}. To make this precise, for each $j\in \mbI$, let $v_j\in L_{\varpi_j}$ be a highest weight vector. Define $v_j^\ast\in (L_{\varpi_j})_{-\varpi_j}^\ast\subset (L_{\varpi_j})^\ast$ by $v_j^\ast(v_j)=1$. 
\begin{proposition}\label{P:t-vs-A}
For each $j\in \mbI$, we have 
\begin{equation*}
\tminor_{v_j^\ast,v_j}(u)=\tminor_{v_j^\ast,v_j}^0(u)=A_j(u-\hbar d_j).
\end{equation*}
Consequently, the series $B_j(u),C_j(u)$ and $D_j(u)$ are recovered as the lifted minors
\begin{gather*}
B_j(u)=\tminor_{e_j\cdot v_j^\ast, v_j}(u+\hbar d_j),\\
C_j(u)=-\tminor_{v_j^\ast, f_j\cdot v_j}(u+\hbar d_j) \quad \text{ and }\quad
D_j(u)=-\tminor_{e_j\cdot v_j^\ast, f_j \cdot v_j}(u+\hbar d_j).
\end{gather*}
\end{proposition}
\begin{proof}
The equality  $\tminor_{v_j^\ast,v_j}(u)=\tminor_{v_j^\ast,v_j}^0(u)$ follows easily from the Gauss decomposition \eqref{R:Gauss}. Hence, to prove the first assertion of the proposition it suffices to establish that $\tminor_{v_j^\ast,v_j}^0(u)=A_j(u-\hbar d_j)$. Let $\underline{\lambda}\in \Hom_{Alg}(\Yhg[0],\C)$ be the algebra homomorphism uniquely determined by 
\begin{equation*}
\underline{\lambda}(\xi_i(u))=1+\delta_{ij} \hbar d_j u^{-1} \quad \forall \; j\in \mbI.
\end{equation*}
Then the identity $\tminor_{v_j^\ast,v_j}^0(u)=A_j(u-\hbar d_j)$ is equivalent to 
\begin{equation*}
\mathscr{R}^0_j(u):=(\underline{\lambda}\otimes \mathrm{id})(\mathscr{R}^0(u))=A_j(-u-\hbar d_j). 
\end{equation*}
By  \eqref{diff:R^0}, the left-hand side is the unique series in $1+u^{-1}\Yhg[0][\![u^{-1}]\!]$  satisfying
\begin{equation}\label{diff:R_j^0}
(\qshift^{4\kappa}-1)\log\mathscr{R}^0_j(u)=\mcL_j(u),
\end{equation}
where $\mcL_j(u)$ is given by 
\begin{equation*}
\mcL_j(u):=(\underline{\lambda}\otimes \mathrm{Id})(\mcL_{21}(u))=\qshift^{2\kappa}\sum_{i,j\in \mbI} c_{ij}(\qshift) \underline{\lambda}\left( \bt_j(-\partial_u) \right) \bt_i(\partial_u)\cdot  (-u^{-2})
\end{equation*}
with all notation as specified below \eqref{diff:R^0}. It thus suffices to show that $A_j(-u-\hbar d_j)$ also satisfies the difference equation \eqref{diff:R_j^0}. This is the content of the following claim.

\noindent\textbf{Claim}. The series $a_j(u)=\log A_j(u)$ satisfies the relation  
\begin{equation*}
(\qshift^{4\kappa+2d_j}-\qshift^{2d_j})(a_j(-u))=\mcL_j(u).
\end{equation*}
\begin{proof}[Proof of claim] 
Since  $\underline{\lambda}(t_{j,r})=(-1)^r \frac{\hbar^r d_j^{r+1}}{r+1}$, we have 
\begin{equation*}
\underline{\lambda}(\bt_j(u))=-\sum_{r\geq 0}(-1)^{r+1} \hbar^{r+1} d_j^{r+1}\frac{u^{r+1}}{(r+1)!}=\frac{1-e^{-\hbar d_j u}}{u}.
\end{equation*}
As $-u^{-2}=-\partial_u(-u^{-1})$, this gives 
\begin{equation*}
\mcL_j(u)
=\qshift^{2\kappa}(\qshift^{2d_j}-1)\sum_{i\in \mbI} c_{ij}(\qshift) \bt_i(\partial_u) \cdot (u^{-1})
=\qshift^{2\kappa}(\qshift^{2d_j}-1) \sum_{i\in \mbI} c_{ij}(\qshift) ( t_i(-u)).
\end{equation*}
Since $t_i(u)=-\sum_{k\in \mbI}\qshift^{-d_k}[a_{ki}]_{\qshift^{d_k}}(a_k(u))$ and $e^{\hbar \partial_z} f(z)|_{z=-u}= e^{-\hbar \partial_u}f(-u)$ , we obtain 
\begin{equation*}
\mcL_j(u)=\qshift^{2\kappa}(\qshift^{2d_j}-1) \sum_{i,k\in \mbI} c_{ij}(\qshift) \qshift^{d_k}[a_{ki}]_{\qshift^{d_k}}(a_k(-u)).
\end{equation*}
Since $c_{ij}(\qshift)=\left[2\kappa/d_j\right]_{\qshift^{d_j}} v_{ij}(\qshift)$ where $v_{ij}(\qshift)$ is the $(i,j)$-th entry of $([a_{ij}]_{\qshift^{d_i}})_{i,j\in \mbI}^{-1}$, we have 
\begin{align*}
\sum_{i,k\in \mbI} \qshift^{d_k}c_{ij}(\qshift)[a_{ki}]_{\qshift^{d_k}}(a_k(-u))
&=
\sum_{i,k\in \mbI} \qshift^{d_k}\left[2\kappa/d_j\right]_{\qshift^{d_j}}[a_{ki}]_{\qshift^{d_k}}v_{ij}(\qshift)(a_k(-u))
\\
&
=\qshift^{d_j}\left[2\kappa/d_j\right]_{\qshift^{d_j}}(a_j(-u)).
\end{align*}
Thus, we obtain 
\begin{equation*}
\mcL_j(z)=\qshift^{2\kappa}\qshift^{2d_j}(\qshift^{d_j}-\qshift^{-d_j})\left[2\kappa/d_j\right]_{\qshift^{d_j}}(a_j(-u))
=
(\qshift^{4\kappa+2d_j}-\qshift^{2d_j})(a_j(-u)). 
\end{equation*}
This completes the proof of the claim, and thus the proof that the lifted minors $\tminor_{v_j^\ast,v_j}(u)$ and $\tminor_{v_j^\ast,v_j}^0(u)$ coincide with $A_j(u-\hbar d_j)$ for each $j\in \mbI$.  \qedhere
\end{proof}
The stated formulas for $B_j(u)$, $C_j(u)$ and $D_j(u)$ now follow from Corollary
\ref{C:[g,GKLO]} and the commutator relations for $[x,\tminor_{f,v}(u)]$ given in the second statement of Corollary \ref{C:minors}. For example, 
\begin{equation*}
B_j(u) = [e_j, A_j(u)] = [e_j, \tminor_{v_j^\ast, v_j}(u+\hbar d_j)] = \tminor_{e_j v_j^\ast, v_j}(u+\hbar d_j) + \tminor_{v_j^\ast, e_j  v_j}(u+\hbar d_j).
\end{equation*}
Since $e_j  v_j = 0$, this gives $B_j(u) = \tminor_{e_j  v_j^\ast, v_j}(u+\hbar d_j)$, as desired.\qedhere
\end{proof}
\begin{remark}\label{R:Ilin-Ryb}
As mentioned above, this gives a new proof of Proposition 2.29 from \cite{Ilin-Rybnikov-19}.  We note, however, that the statement of \cite{Ilin-Rybnikov-19} seems to be slightly incorrect: it is missing the shifts by $\hbar d_j$ which appear in the statement of Proposition \ref{P:t-vs-A}.  
\end{remark}
}


\section{The dual braid group action and weights}\label{sec:Weights}

In this section, we dualize the action of the braid group $\Bg$ on $\Yhg[0]$ constructed in the previous section to obtain a $\Bg$-action on the group  $(\C(\!(u^{-1})\!)^\times)^\mbI$ by automorphisms, which preserves the subgroups 
\begin{equation*}
(1+u^{-1}\C[\![u^{-1}]\!])^\mbI\quad\text{ and }\quad (\C(u)^\times)^\mbI.
\end{equation*}
In particular, this recovers the representation of $\Bg$ studied in the work \cite{tan-braid} of Tan, which was inspired by its $q$-analogue obtained in the earlier work \cite{Chari-braid} of Chari. We then show that this action naturally computes the eigenvalues of the series $\{\xi_i(u)\}_{i\in \mbI}$ on any extremal vector of any finite-dimensional highest weight representation  of $\Yhg$ (see Proposition \ref{P:extremal}), thus generalizing a result from \cite{tan-braid}.

\subsection{The braid group action on  $(1+u^{-1}\C[\![u^{-1}]\!])^\mbI$}\label{ssec:wts-dual}
We begin by explaining how the representation $\Yhg[0]$ of $\Bg$ from Theorem \ref{T:mbraid} can be dualized to obtain an action of $\Bg$ on $(1+u^{-1}\C[\![u^{-1}]\!])^\mbI$ by group automorphisms. 

Let $\upvartheta$ be the group anti-automorphism of $\Bg$ uniquely determined by $\upvartheta(\braid_i)=\braid_i$ for all $i\in \mbI$. Note that $\upvartheta(\braid_w) = \braid_{w^{-1}}$ for $w \in \Wg$ an element of the Weyl group. Using $\upvartheta$, we may define an action of $\Bg$ on the linear dual $\Yhg[0]^\ast$ by 
\begin{equation}\label{Bg-dual}
\upsigma(f)(y)=f(\upvartheta(\upsigma)\cdot y)
\end{equation}
for all $\upsigma\in \Bg$, $y\in \Yhg[0]$ and $f\in \Yhg[0]^\ast$. Note that this action is uniquely determined by the requirement that $\braid_i$ operates as the transpose/adjoint of $\mbraid_i$: 
\begin{equation*}
\braid_i(f)=\mbraid_i^*(f)=f\circ \mbraid_i\quad \forall \quad i\in \mbI\; \text{ and }\; f\in \Yhg[0]^\ast.
\end{equation*}

The Drinfeld coalgebra structure on $Y_\hbar^0(\mfg)$ induces a commutative algebra structure  on $Y_\hbar^0(\mfg)^\ast$ with unit given by the counit on $Y_\hbar^0(\mfg)$ and product given by the transpose of the Drinfeld coproduct $\Delta^0$ (see \eqref{def:Delta^0}). Since each modified braid group operator $\mbraid_i$ is a coalgebra homomorphism, $\Bg$ acts on $\Yhg[0]^\ast$ by algebra automorphisms. Moreover, as each $\mbraid_i$ is an algebra automorphism, the space of algebra homomorphisms
\begin{equation*}
\mathrm{Hom}_{Alg}(\Yhg[0],\C)\subset \Yhg[0]^\ast
\end{equation*}
is a subrepresentation of $\Yhg[0]^\ast$. This is a subgroup of the group of units in $\Yhg[0]^\ast$; in fact, one has an isomorphism of groups
\begin{equation*}
\mathrm{Hom}_{Alg}(\Yhg[0],\C)\iso (1+u^{-1}\C[\![u^{-1}]\!])^\mbI , \quad \gamma \mapsto (\gamma(\xi_i(u)))_{i\in \mbI}.
\end{equation*}
 From this point on, we will not distinguish between these two spaces, and the above identification will always be assumed. The first part of the following corollary summarizes the above discussion.
\begin{corollary}\label{C:T_1-rep}
The formula \eqref{Bg-dual} defines an action of $\Bg$ on $(1+u^{-1}\C[\![u^{-1}]\!])^\mbI$ by group automorphisms. Moreover, if $\underline{\lambda}=(\lambda_i(u))_{i\in \mbI}\in (1+u^{-1}\C[\![u^{-1}]\!])^\mbI$, then the $i$-th component $\braid_j(\underline{\lambda})_i$ of $\braid_j(\underline{\lambda})$ is given by
\begin{equation*}
\braid_j(\underline{\lambda})_i=
\lambda_i(u)\prod_{k=0}^{|a_{ji}|-1} \lambda_j\!\left(u-\frac{\hbar d_j}{2}(|a_{ji}|-2k)\right)^{(-1)^{\delta_{ij}}}.
 \end{equation*}
\end{corollary}
\begin{proof}
It remains to show that the $i$-th component $\braid_j(\underline{\lambda})_i=(\braid_j(\underline{\lambda}))(\xi_i(u))$ of $\braid_j(\underline{\lambda})$ is given by the claimed formula. By definition of $\braid_j$, we have
\begin{align*}
(\braid_j(\underline{\lambda}))(\xi_i(u))
=\underline{\lambda}(\mbraid_j(\xi_i(u)))
 =\lambda_i(u)\prod_{k=0}^{|a_{ji}|-1} \lambda_j\!\left(u-\frac{\hbar d_j}{2}(|a_{ji}|-2k)\right)^{(-1)^{\delta_{ij}}},
 \end{align*}
 where we have used the formula for $\mbraid_j(\xi_i(u))$ obtained in Corollary \ref{C:mbraid-t-xi} in the last equality. \qedhere
\end{proof}
\begin{remark}\label{R:Tan}
The action of $\Bg$ on $(1+u^{-1}\C[\![u^{-1}]\!])^\mbI$ constructed above is given by the same formulas as the action of  $\Bg$ on  $(\C(u)^\times)^\mbI$ defined by Tan in Proposition 3.1 of \cite{tan-braid}, except that $(a_{ji},d_j)$ and $(a_{ij},d_i)$ are interchanged; see also \cite{FrHer-23}*{\S3.2} and the discussion following Proposition \ref{P:Bg-difference} below. In the next subsection, we will explain how to recover the $\Bg$-action of \cite{tan-braid} from that given by Corollary \ref{C:T_1-rep} using formal, additive difference equations. 
We note this relation is also precisely why we have used the group anti-automorphism $\upvartheta$ of $\Bg$ in \eqref{Bg-dual} as opposed to the standard antipode on $\Bg$, given by inversion. 
\end{remark}
%
%

\subsection{Extending the representation space}\label{ssec:wts-diff}

Let $\Monic\subset \C(\!(u^{-1})\!)^\times$ be the subgroup consisting of monic Laurent series:
\begin{equation*}
\Monic=\bigcup_{k\in \Z} u^k (1+u^{-1}\C[\![u^{-1}]\!]).
\end{equation*}
 We will need the following elementary lemma. 
\begin{lemma}\label{L:factorization}
Let $\lambda(u)=1+\hbar\sum_{r\geq 0} \lambda_r u^{-r-1}$ be an arbitrary series in $1+u^{-1}\C[\![u^{-1}]\!]$ and suppose that $d\in \C^\times$. Then the formal difference equation 
\begin{equation*}
\frac{\mu(u+\hbar d)}{\mu(u)}=\lambda(u)
\end{equation*}
has a solution $\mu(u)\in \Monic$ if and only if $\lambda_0\in d\Z$. In this case, $\mu(u)$ is unique and belongs to $u^{\lambda_0/d}(1+u^{-1}\C[\![u^{-1}]\!])$. 
\end{lemma}
\begin{proof}
Write $\mu(u)=u^k \dot{\mu}(u)$, where $\dot{\mu}(u)\in 1+u^{-1}\C[\![u^{-1}]\!]$. Then the difference equation becomes 
\begin{equation*}
\left(1+\hbar \frac{d}{u} \right)^k\frac{\dot\mu(u+\hbar d)}{\dot\mu(u)}=\lambda(u).
\end{equation*}
Taking the formal logarithm of both sides, this gives the equivalent relation
\begin{equation*}
k\log\left(1+\hbar \tfrac{d}{u} \right) + b(u+\hbar d)-b(u)=\log(\lambda(u)),
\end{equation*}
where $b(u)=\sum_{r\geq 0} b_r u^{-r-1}=\log(\dot\mu(u))$. The coefficient of $u^{-1}$ on both sides is $k\hbar d = \hbar\lambda_0$. This shows that if a solution exists, we must have $\lambda_0\in d\Z$ and $\mu(u)\in u^{\lambda_0/d}(1+u^{-1}\C[\![u^{-1}]\!])$.

  Conversely, if $\lambda_0\in d\Z$ holds then we take $k=\lambda_0/d$, and solve the above equation for $b(u)$ recursively in its coefficients; taking the coefficients of $u^{-r-1}$ for $r>0$ will allow one to express $b_{r-1}$ in terms of $\{b_\ell\}_{0\leq \ell <r-2}$ and the coefficients of $k\log\left(1+\hbar \tfrac{d}{u} \right)$ and $\log(\lambda(u))$. \qedhere
\end{proof}

Recall that $\Lambda=\bigoplus_{i\in \mbI} \Z \varpi_i$ denotes the weight lattice of $\g$. We view $\Lambda$ as a representation of $\Bg$ via the canonical action of the Weyl group on $\mfh^\ast$.
Let us introduce the group homomorphism 
\begin{equation*}
\deg: \Monic^\mbI\to \Lambda, \quad (\lambda_i(u))_{i\in \mbI}\mapsto \sum_{i\in \mbI} \deg \lambda_i(u) \varpi_i,
\end{equation*}
where $\deg \lambda_i(u)\in \Z$ is the unique integer $k$ for which $u^{-k}\lambda_i(u)\in 1+u^{-1}\C[\![u^{-1}]\!]$. 

Next, following the notation of Remark \ref{R:logs}, let $\qshift$ be the group automorphism of $\Monic$ given by the translation $\msq=e^{\frac{\hbar}{2}\partial_u}$. For each $a\in \C$, we write $\qshift^a$ for $e^{\frac{a\hbar}{2}\partial_u}$, so that $\qshift^a(f(u))=f(u+a\frac{\hbar}{2})$.
 Given a diagonal matrix $A=\mathrm{Diag}(a_i)_{i\in \mbI}$, we write $\qshift^{A}$ for the group automorphism of $\Monic^\mbI$ given by 
\begin{equation*}
\qshift^{A} (\lambda_i(u))_{i\in \mbI}=(\qshift^{a_i}\lambda_i(u))_{i\in \mbI}=\left(\lambda_i\left(u+a_i\tfrac{\hbar}{2}\right)\right)_{i\in \mbI}.
\end{equation*}
In what follows, we will only be interested in the case where $A=2D$, where $D$ is the diagonal matrix $D=(d_i)_{i\in \mbI}$ of symmetrizing integers.  In addition, we note that if $\underline{\mu}\in \Monic^\mbI$, then the element $\left(\qshift^{2D}\underline{\mu}\right)\underline{\mu}^{-1}$ lies  in $(1+u^{-1}\C[\![u^{-1}]\!])^\mbI$, and hence we may apply $\upsigma$ to it for any $\upsigma\in \Bg$.
\begin{proposition}\label{P:Bg-difference}
Let $\upsigma\in \Bg$. Then for each $\underline{\mu}=(\mu_i(u))_{i\in \mbI}\in \Monic^\mbI$, there exists a unique element $\M[\upsigma](\underline{\mu})\in \Monic^\mbI$ satisfying
\begin{equation*}
\upsigma\!\left( \frac{\qshift^{2D}\underline{\mu}}{\underline{\mu}}\right)= \frac{\qshift^{2D}\M[\upsigma](\underline{\mu})}{\M[\upsigma](\underline{\mu})}
\end{equation*}
Moreover,  the formula $\upsigma\cdot\underline{\mu}=\M[\upsigma](\underline{\mu})$ determines an action of $\Bg$ on $\Monic^\mbI$ by group automorphisms for which $\deg:\Monic^\mbI\to \Lambda$ is a $\Bg$-equivariant map:
\begin{equation*}
\deg \M[\upsigma](\underline{\mu})= \upsigma\cdot \deg(\underline{\mu}).
\end{equation*}
\end{proposition} 
\begin{proof}
Let $\underline{\lambda}=(\lambda_i(u))_{i\in \mbI}$ denote the element $\left(\qshift^{2D}\underline{\mu}\right)\underline{\mu}^{-1}\in (1+u^{-1}\C[\![u^{-1}]\!])^\mbI$, and write $\lambda_i(u)=1+\hbar\sum_{r\geq 0} \lambda_{i,r}u^{-r-1}$. Note that
\begin{equation*}
\upsigma(\underline{\lambda})=1 + \hbar \upsigma\cdot (\lambda_{i,0})_{i\in \mbI} u^{-1} +O(u^{-2}),
\end{equation*}
where the action of $\upsigma$ on $\C^\mbI$ is inherited via the identification $\mfh^\ast \iso \C^\mbI$, $\lambda\mapsto (\lambda(\xi_{i,0}))_{i\in \mbI}=(\lambda,\alpha_i)_{i\in \mbI}$. Equivalently, one has
\begin{equation*}
\upsigma\cdot (\lambda_{i,0})_{i\in \mbI}:= \left(\sum_{j\in \mbI} \frac{1}{d_j}\lambda_{j,0}(\alpha_i,\upsigma(\varpi_j))\right)_{i\in \mbI} = (\alpha_i,\upsigma\cdot \deg(\underline{\mu}))_{i\in \mbI},
\end{equation*}
where we have used that, by the previous lemma, $\lambda_{j,0}=d_j \deg \mu_j(u)$ for each $j\in \mbI$. As the action of $\Bg$ on $\mfh^\ast$ preserves $\Lambda$, one has $(\alpha_i,\upsigma\cdot \deg(\underline{\mu}))\in d_i \Z$ for each $i\in \mbI$. The previous lemma therefore implies that $\M[\upsigma](\underline{\mu})$ exists, is unique, and has degree $\upsigma\cdot \deg(\underline{\mu})$. 

The rest of the  proposition now follows easily using the uniqueness assertion of the previous lemma.\qedhere
\end{proof}

Since the subgroup $(1+u^{-1}\C[\![u^{-1}]\!])^\mbI\subset \Monic^\mbI$ consists precisely of those $\underline{\mu}\in \Monic^\mbI$ for which $\deg(\underline{\mu})=0$, the last assertion of the proposition implies that it is an invariant subset. When $\g$ is simply laced, this recovers the action of $\Bg$ on $\mathrm{Hom}_{Alg}(\Yhg[0],\C)\cong (1+u^{-1}\C[\![u^{-1}]\!])^\mbI$ from Corollary \ref{C:T_1-rep}. Indeed, in this case $D$ is the identity matrix and so, for each  $\upsigma\in \Bg$ and $\underline{\mu}\in (1+u^{-1}\C[\![u^{-1}]\!])^\mbI$, one has
\begin{equation*}
\upsigma\!\left( \frac{\qshift^{2D}\underline{\mu}}{\underline{\mu}}\right)
=
\frac{\upsigma(\qshift^{2D}\underline{\mu})}{\upsigma(\underline{\mu})}
=
\frac{\qshift^{2D}\upsigma(\underline{\mu})}{\upsigma(\underline{\mu})}
\end{equation*}
and thus $\upsigma(\underline{\mu})=\M[\upsigma](\underline{\mu})$. 
When $\g$ is not simply laced, $\upsigma$ and $\qshift^D$ need not commute, and the equality $\upsigma(\underline{\mu})=\M[\upsigma](\underline{\mu})$ no longer holds in general. This is illustrated concretely by the following corollary, which computes $\M[\braid]_j$ explicitly for each $j\in \mbI$.
\begin{corollary}\label{C:Monic-braid-compute}
Let $\underline{\mu}\in \Monic^\mbI$, and fix an index $j\in \mbI$. Then the $i$-th component of the element $\M[\braid]_j(\underline{\mu})\in \Monic^\mbI$ is given by
\begin{equation*}
\M[\braid]_j(\underline{\mu})_i=
\mu_i(u)\prod_{k=0}^{|a_{ij}|-1} \mu_j\!\left(u-\frac{\hbar d_i}{2}(|a_{ij}|-2k)\right)^{(-1)^{\delta_{ij}}}.
 \end{equation*}
\end{corollary}
\begin{proof}
Set $\underline{\lambda}:=(\qshift^{2D}\underline{\mu})\underline{\mu}^{-1}$, so that $\lambda_i(u)=\mu_i(u+\hbar d_i)/\mu_i(u)$ for each $i\in \mbI$. Then, by Corollary \ref{C:T_1-rep}, we have 
\begin{align*}
\lambda_i(u)^{-1}\braid_j(\underline{\lambda})_i=
\prod_{k=0}^{|a_{ji}|-1} \lambda_j\!\left(u-\frac{\hbar d_j}{2}(|a_{ji}|-2k)\right)^{(-1)^{\delta_{ij}}}
=\frac{\mu_j(u+\frac{\hbar}{2} |d_j a_{ji}|)^{(-1)^{\delta_{ij}}}}{\mu_j(u-\frac{\hbar}{2} |d_j a_{ji}|)^{(-1)^{\delta_{ij}}}},
\end{align*}
where to obtain the last expression we have substituted $\mu_j(u+\hbar d_j)/\mu_j(u)$ for $\lambda_j(u)$ in the second expression and observed that the product is telescoping. On the other hand, since $d_j a_{ji}=d_i a_{ij}$, the same telescoping property with $i$ and $j$ interchanged implies that  
\begin{equation*}
\braid_j(\underline{\lambda})_i=\lambda_i(u)\frac{\mu_j(u+\frac{\hbar}{2} |d_j a_{ji}|)^{(-1)^{\delta_{ij}}}}{\mu_j(u-\frac{\hbar}{2} |d_j a_{ji}|)^{(-1)^{\delta_{ij}}}}=\frac{\qshift^{2d_i}\M[\braid]_j(\underline{\mu})_i^\prime}{\M[\braid]_j(\underline{\mu})_i^\prime},
\end{equation*}
where $\braid_j(\underline{\lambda})_i^\prime$ is the claimed expression for $\braid_j(\underline{\lambda})_i$ from the statement of the corollary.
Hence, by the uniqueness assertion of Proposition \ref{P:Bg-difference}, we have $\braid_j(\underline{\lambda})_i=\braid_j(\underline{\lambda})_i^\prime$ for each $i\in \mbI$, as desired. 
\qedhere
\end{proof}
Note that the above formula for $\M[\braid]_j(\underline{\mu})_i$ can be recovered from the formula for $\braid_j(\underline{\lambda})_i$ given in Corollary \ref{C:T_1-rep} by replacing $a_{ji}$ by  $a_{ij}$ and $d_j$ by $d_i$. This phenomenon has also been observed in the recent work \cite{FrHer-23} of E. Frenkel and D. Hernandez in the course of studying a $q$-analogue of the $\Bg$-action from Proposition \ref{P:Bg-difference}, where it has been attributed to a non-trivial manifestation of Langlands duality; see Definition 3.5 and Remark 3.7 therein.
\begin{remark}\label{R:big-extend}
Let us now explain how to further extend the underlying representation space from $\Monic^\mbI$ to $(\C(\!(u^{-1})\!)^\times)^\mbI$ in order to recover the action from \cite{tan-braid}*{Prop.~3.1}. First recall that the Langlands dual Lie algebra $\g^\vee$ is defined to have Cartan matrix given by the transpose $(a_{ji})_{i,j\in \mbI}$ of the Cartan matrix $(a_{ij})_{i,j\in \mbI}$ of $\g$.  Consider the corresponding algebraic group $G^\vee$ of adjoint type.  Then its maximal torus $T^\vee $ can be identified with $(\C^\times)^\mbI$.  More precisely, each fundamental weight $\varpi_i$ can be viewed as a homomorphism $\C^\times \rightarrow T^\vee$, and the product of these homomorphisms gives $(\C^\times)^\mbI \cong T^\vee$.   Note also that we may naturally identify $\Wg \cong \mathscr{W}_{\g^\vee}$ and $\Bg \cong \mathscr{B}_{\g^\vee}$. Thus $\Wg$ acts on $T^\vee \cong (\C^\times)^\mbI$ by the formulas
\begin{equation*}
s_j (\lambda_i)_{i\in \mbI}=(\lambda_i\lambda_j^{-a_{ij}})_{i\in \mbI}
\end{equation*}
for all $i,j\in \mbI$ and $(\lambda_i)_{i\in \mbI}\in(\C^\times)^\mbI$. It follows by Proposition \ref{P:Bg-difference} that $\Bg$ acts on the group $(\C^\times)^\mbI\times \Monic^\mbI$ by automorphisms. However, the groups $(\C^\times)^\mbI\times \Monic^\mbI$ and $(\C(\!(u^{-1})\!)^\times)^\mbI$ are isomorphic, with an isomorphism $(\C^\times)^\mbI\times \Monic^\mbI\iso (\C(\!(u^{-1})\!)^\times)^\mbI$  given by componentwise multiplication: 
\begin{equation*}
((\lambda_i)_{i\in \mbI},(\mu_i(u))_{i\in \mbI})\mapsto (\lambda_i\mu_i(u))_{i\in \mbI}.
\end{equation*}
Hence, we can conclude that the braid group $\Bg$ acts on $(\C(\!(u^{-1})\!)^\times)^\mbI$ by group automorphisms, and that this action is given explicitly by the formulas of Corollary \ref{C:Monic-braid-compute}. Moreover, the space $(\C(u)^\times)^\mbI$ is a subrepresentation, and coincides with the representation of $\Bg$ obtained in Proposition 3.1 of \cite{tan-braid}; see Remark \ref{R:Tan} above. 
\end{remark}

\begin{remark}
Considering the $\hbar = 0$ limit of the $\Bg$--action on $( \C(\!(u^{-1})\!)^\times)^\mbI$ from the previous remark, we simply recover the natural action of the Weyl group $\Bg\twoheadrightarrow \Wg$ on the $\C(\!(u^{-1})\!)$-points of the torus $T^\vee \cong (\C^\times)^\mbI$.   Namely, the group $\Wg$ naturally acts on the $A$-points $T^\vee(A)$ for any commutative $\C$-algebra $A$, since it acts on the scheme $T^\vee$ itself: for $(x_i)_{i\in\mbI} \in T^\vee(A) \cong (A^\times)^\mbI$ this action is given by $s_j (x_i)_{i \in \mbI} \mapsto (x_i x_j^{-a_{ij}})_{i \in \mbI}$.  Considering the case of $A = \C(\!(u^{-1})\!)$, this is easily seen to agree with the $\hbar=0$ limit of the formulas from Corollary \ref{C:Monic-braid-compute}.  
\end{remark}

\subsection{Weights of extremal vectors}\label{ssec:wts-extemal}

Recall from Section \ref{ssec:Bg} that the homomorphism $\Bg\to \Wg$ admits a canonical set-theoretic section $w\mapsto \braid_w$,  where $\braid_w$ is defined by taking any reduced expression $w=s_{i_1}\cdots s_{i_\ell}$ and setting
\begin{equation*}
\braid_w=\braid_{i_1}\cdots \braid_{i_\ell}\in \Bg.
\end{equation*} 
The following is the main result of this subsection; it provides a strengthening of \cite{tan-braid}*{Prop.~4.5}, and is the Yangian analogue of \cite{Chari-braid}*{Prop.~4.1}.
\begin{proposition}\label{P:extremal}
 Let $\underline{\lambda}$ be the highest weight of an irreducible finite-dimensional $Y_\hbar(\mfg)$ module $V$, and let $\lambda\in \mfh^\ast$ be the corresponding $\mfg$-weight. Then, for each $w\in \Wg$, we have 
 \begin{equation*}
 \xi_i(u)|_{V_{w(\lambda)}} = \braid_{w}(\underline{\lambda})_i \cdot \mathrm{Id}_{V_{w(\lambda)}} \quad \forall\quad i\in \mbI.
 \end{equation*}
 In particular, if $V\cong L(\underline{\mathrm{P}})$, then for each $i\in \mbI$ one has
 \begin{equation}\label{ext-DP}
  \xi_i(u)|_{V_{w(\lambda)}} = \frac{\qshift^{2d_i}\M[\braid]_w(\underline{\mathrm{P}})_i}{\M[\braid]_w(\underline{\mathrm{P}})_i} \cdot \mathrm{Id}_{V_{w(\lambda)}}.
  \end{equation}
\end{proposition}
 \begin{proof}
The second assertion is a consequence of the first and that, by Proposition \ref{P:Bg-difference}, if $\underline{\lambda}=(\qshift^{2D}\underline{\mathrm{P}})\cdot \underline{\mathrm{P}}^{-1}$, then 
$
\braid_w(\underline{\lambda})=(\qshift^{2D}\M[\braid]_w(\underline{\mathrm{P}})) \cdot \M[\braid]_w(\underline{\mathrm{P}})^{-1}$.

To establish the first assertion,  note that  we have 
$
 \braid_w(\underline{\lambda})_i = \underline{\lambda}(\mbraid_{w^{-1}}(\xi_i(u)))$ for all $i\in \mbI$. Hence, it is sufficient to prove that 
\begin{equation*}
\xi_i(u)|_{V_{w(\lambda)}} =\underline{\lambda}(\mbraid_{w^{-1}}(\xi_i(u)))\cdot \mathrm{Id}_{V_{w(\lambda)}} \quad \forall \; i \in \mbI.
\end{equation*}
This is equivalent to establishing that, for each $w\in \Wg$, the eigenvalue of $y\in \Yhg[0]$ on the one-dimensional weight space $V_{w^{-1}(\lambda)}$ coincides with the eigenvalue of $\mbraid_w(y)$ on the highest weight space  $V_\lambda$.  

Fix $y\in \Yhg[0]$, and let $\mu_w(y)\in \C$ be the eigenvalue of $y$ on  $V_{w (\lambda)}$, and let $v_\lambda\in V_\lambda$ be a highest weight vector.   Set $v:=\braid_w^V(\braid_{w^{-1}}^V(v_\lambda))$, where $\braid_w^V$ is the image of $\braid_w$ under the homomorphism $\braid_i\mapsto \braid_i^V$ from Proposition \ref{P:tau-wt-space}. Note that $\braid_{w^{-1}}^V(v_\lambda) \in V_{w^{-1}(\lambda)}$, while $v \in V_\lambda$ is itself a highest weight vector. Therefore $y \cdot \braid_{w^{-1}}^V(v_\lambda) = \mu_w(y) \braid_{w^{-1}}^V (v_\lambda)$ and $y \cdot v=\Pi(y) \cdot v$ for all $y\in \Yhg_0$. We thus have
\begin{equation*}
\mu_w(y)v= \braid_w^V( y \cdot \braid_{w^{-1}}^V(v_\lambda) )
=
\braid_w(y) \cdot v =\Pi(\braid_w(y))\cdot v = \mbraid_w(y)\cdot v,
\end{equation*}
where in the second equality  we have applied the last statement of Proposition \ref{P:tau-wt-space}, and in the fourth equality we have applied \eqref{mbraid_w}. This completes the proof of the proposition.
\end{proof}
\begin{corollary}\label{C:sigma-poly}
Let $\underline{\mathrm{P}}=(P_j(u))_{j\in \mbI}$ be any tuple of Drinfeld polynomials, and let $w\in \Wg$ and  $i\in \mbI$. Define $P_{i,w}(u)\in \Monic$ by 
\begin{equation*}
P_{i,w}(u)
:=
\begin{cases}
\M[\braid]_{w}(\underline{\mathrm{P}})_i & \text{ if }\; w^{-1}(\alpha_i)\in \Root^+\\[5pt]
{\displaystyle \frac{1}{\M[\braid]_{w}(\underline{\mathrm{P}})_i}}& \text{ if }\; w^{-1}(\alpha_i)\notin \Root^+
\end{cases}.
\end{equation*}
Then $P_{i,w}(u)$ is a monic polynomial in $u$.
\end{corollary}
\begin{proof}
By Remark 2.2 of \cite{ChPr1}, there is a monic polynomial $P_{i,w}^+(u)$ such that 
\begin{equation*}
\xi_i(u)|_{V_{w(\lambda)}} = \frac{P_{i,w}^+(u+\hbar d_i)^{f_i(w)}}{P_{i,w}^+(u)^{f_i(w)}} \cdot \mathrm{Id}_{V_{w(\lambda)}},
\end{equation*}
where $V=L(\underline{\mathrm{P}})$ and $f_i(w)$ is $1$ if $w^{-1}(\alpha_i)\in \Root^+$ and $-1$ otherwise. By Proposition \ref{P:extremal} and the uniqueness assertion of Lemma \ref{L:factorization}, we can conclude that  $P_{i,w}^+(u)^{f_i(w)}=\M[\braid]_w(\underline{\mathrm{P}})_i=P_{i,w}(u)^{f_i(w)}$. Hence, $P_{i,w}(u)$ coincides with the polynomial $P_{i,w}^+(u)$. \qedhere 
\end{proof}
\begin{remark}\label{R:Tan-extremal}
The statement of \cite{tan-braid}*{Prop.~4.5} is that if $i\in \mbI$ is such that $\ell(s_i w)=\ell(w)+1$, then  the relation \eqref{ext-DP} of Proposition \ref{P:extremal} is satisfied. This was proven by induction on the length of $w$ using a technical argument based on the defining relations of $\Yhg$. Moreover, in this case one always has $w^{-1}(\alpha_i)\in \Root^+$, and the polynomiality of $P_{i,w}(u)$ from Corollary \ref{C:sigma-poly} can be deduced from Proposition 4.5 and Lemma 4.3 of \cite{tan-braid} (see also \cite{guay-tan}*{Lem.~5.1}) using the uniqueness assertion of Lemma \ref{L:factorization}. 
\end{remark}

\section{Baxter polynomials and cyclicity}\label{sec:Baxter}

Our goal in this section is to prove a conjecture from \cite{GWPoles}*{\S7.4} which asserts that the two sufficient conditions for the cyclicity and irreducibility of any tensor product 
$
L(\underline{\mathrm{P}})\otimes L(\underline{\mathrm{Q}})$ obtained in \cite{GWPoles} and \cite{tan-braid} are identical; see Corollary \ref{C:Cyclic}. 

We will do this by applying the results of the previous two sections to factorize the so-called specialized Baxter polynomials associated to the extremal vectors of  $L(\underline{\mathrm{P}})$ (see Section \ref{ssec:Baxter-extremal}) as products of polynomials arising from the $\Bg$-orbit of $\underline{\mathrm{P}}$. This will be achieved in Theorem \ref{T:Baxter-braid} after recalling some preliminaries on Baxter polynomials and their relation to the sets of poles of a representation in Sections \ref{ssec:transfer} and \ref{ssec:Baxter-poles}.

\subsection{The transfer operator $T_i(u)$}\label{ssec:transfer}
Suppose that $V$ is any finite-dimensional highest-weight module of $\Yhg$. Let $\lambda\in \mfh^\ast$ denote the $\mfg$-weight of any highest-weight vector in $V$ and, for each $i\in \mbI$,  let $\lambda_i^A(u)\in 1+u^{-1}\C[\![u^{-1}]\!]$ denote the eigenvalue of the series $A_i(u)$ on $V_\lambda$.
We then introduce the normalized operator 
\begin{equation*}
A_i^V(u):=\lambda_i^A(u)^{-1}A_i(u)|_V\in \End(V)[\![u^{-1}]\!].
\end{equation*}
Then, by Theorem 4.4 of \cite{GWPoles} (see also Corollary 4.7 of \cite{GWPoles} together with Propositions 5.7 and 5.8 of \cite{HerZhang-shifted}), there is a unique monic polynomial $T_i(u)\in \End(V)[u]$ satisfying
\begin{equation}\label{transfer-op}
T_{i}(u+\hbar d_i)=A_i^V(u)T_{i}(u).
\end{equation}
\begin{remark}\label{R:transfer}
 Let us explain how this relates to the \textit{$i$-th abelianized transfer operator} $\mathscr{T}_i(u)$ introduced in \cite{GWPoles}*{\S4.3}.  For each $i\in \mbI$, let $\mathscr{A}_i(u)$ denote the product 
\begin{equation*}
\mathscr{A}_i(u):= A_i(u)A_i(u+\hbar d_i)\cdots A_i(u+(\tfrac{2\kappa}{d_i}-1)\hbar d_i)\in \Yhg[0][\![u^{-1}]\!],
\end{equation*}
where we recall that $\kappa\in \frac{1}{2}\Z$ has been defined below \eqref{diff:R^0}.  
The element $\mathscr{A}_i(u)$ coincides with the series introduced in (4.2) of \cite{GWPoles} (see  Remark 4.7 therein) and operates on $V$ as the expansion at infinity of a rational function of $u$. In Section 4.3 of \cite{GWPoles}, an $\End(V)$-valued meromorphic function $\mathscr{T}_i(u)$ was defined as one of the two canonical fundamental solutions of the difference equation 
\begin{equation*}
\mathscr{T}_i(u+2\kappa\hbar)=\mathscr{A}_i(u)|_V\mathscr{T}_i(u).
\end{equation*}
As explained in Remarks 4.3 and 7.5 of \cite{GWPoles}, one may recover $\mathscr{T}_i(u)$ heuristically by applying the formal substitution $\xi_j(u)\mapsto u^{\delta_{ij}}$ to the first tensor factor of the diagonal component $\mcR^0(u)$ of the universal $R$-matrix of $\Yhg$. 

It was proven in Theorem 4.4 of \cite{GWPoles} that if $\lambda_i^T(u)$ denotes the eigenvalue of $\mathscr{T}_i(u)$ on $V_\lambda$, then  $\lambda_i^T(u)^{-1} \mathscr{T}_i(u)$ is a monic, $\End(V)$-valued polynomial of $u$. This can be seen as a weak, rational counterpart of a polynomiality result established for quantum loop algebras in the work \cite{FrHer-15} of Frenkel and Hernandez. Finally, we note that by Corollary 4.5 and Remark 4.5 of \cite{GWPoles}, the polynomial $\lambda_i^T(u)^{-1} \mathscr{T}_i(u)$ is precisely $T_i(u)$ introduced above. 
\end{remark}
%
%

\subsection{Baxter polynomials and poles}\label{ssec:Baxter-poles}

The eigenvalues of the polynomial operator $T_i(u)$ defined by \eqref{transfer-op}  are called the specialized \textit{Baxter polynomials} associated to $V$. In \cite{GWPoles}, they were applied to compute the sets of poles $\sigma_i(V)\subset \C$ of $V$, as defined in Definition \ref{D:Poles}.

In order to describe the relation between Baxter polynomials and the sets $\Poles_i(V)$, let $\mathcal{Z}_i(V)$ denote the set of zeroes of all eigenvalues of $T_i(u)$, and let $\mcQ_{i,V}^\mfg(u)$ denote the eigenvalue of $T_i(u)$ on the lowest weight space $V_{w_0(\lambda)}$, where $w_0\in \Wg$ is the longest element. Then, by Theorem 4.4 of \cite{GWPoles}, one has 
\begin{equation}\label{Poles=Z}
\mcZ_i(V)=\Poles_i(V)=\mathsf{Z}(\mcQ_{i,V}^\mfg(u)) \quad \forall \quad i\in \mbI. 
\end{equation}

In the case where $V$ is irreducible, the polynomials $\mcQ_{i,V}^\mfg(u)$ (and thus the sets $\Poles_i(V)$) were computed explicitly in Theorem 5.2 of \cite{GWPoles} in terms of the entries of the quantum Cartan matrix $\mbE(z)=(v_{ij}(z))_{i,j\in \mbI}$ (see below \eqref{diff:R^0}).
More precisely, one has $v_{ij}(z)=\sum_{r\geq d_i}v_{ij}^{(r)}z^r\in \Z[\![z]\!]$ with $v_{ij}^{(r)}\geq 0$ for all $d_i\leq r\leq 2\kappa-d_i$, and $\mcQ_{i,V}^\mfg(u)$ is given by
\begin{equation*}\label{Q-explicit}
\mcQ_{i,V}^\mfg(u)=\prod_{j\in \mbI} \prod_{b=d_i}^{2\kappa-d_i} P_j\!\left( u-(b-d_j)\frac{\hbar}{2}\right)^{v_{ij}^{(b)}},
\end{equation*}
where $\underline{\mathrm{P}}=(P_j(u))_{j\in \mbI}$ is the tuple of Drinfeld polynomials associated to $V\cong L(\underline{\mathrm{P}})$. Consequently, one has 
\begin{gather*}
\Poles_i(V)=\bigcup_{j\in \mbI}\left( \mathsf{Z}(P_j(u))+\Poles_i(L_{\varpi_j})\right),\\
\Poles_i(L_{\varpi_j})=\left\{\frac{b\hbar}{2}: d_i-d_j\leq b\leq 2\kappa-d_i-d_j \; \text{ and }\; v_{ij}^{(b+d_j)}>0\right\}.
\end{gather*} 
 
\subsection{Baxter polynomials associated to extremal weights}\label{ssec:Baxter-extremal}

Let us now fix $V=L(\underline{\mathrm{P}})$, where $\underline{\mathrm{P}}=(P_j(u))_{j\in \mbI}$ is any tuple of Drinfeld polynomials. Let $\lambda=\sum_{i\in \mbI}\deg(P_i)\varpi_i\in \mfh^\ast$; this is the $\mfg$-weight of any highest weight vector in $V$. Given $w\in \Wg$, let $\mcQ_{i,\underline{\mathrm{P}}}^w(u)\in \C[u]$ denote the eigenvalue of $T_i(u)$ on $V_{w(\lambda)}$: 
\begin{equation*}
T_i(u)|_{V_{w(\lambda)}}=\mcQ_{i,\underline{\mathrm{P}}}^w(u)\cdot \mathrm{Id}_{V_{w(\lambda)}}.
\end{equation*}
In particular, if $w$ is the longest element $w_0$, then  $\mcQ_{i,\underline{\mathrm{P}}}^w(u)=\mcQ_{i,V}^\mfg(u)$. 

The following theorem provides the main result of this section. It gives a factorization of  $\mcQ_{i,\underline{\mathrm{P}}}^w(u)$ as a product of polynomials arising from the $\Bg$-orbit of $\underline{\mathrm{P}}$, with respect to the action of $\Bg$ on $\Monic^\mbI$ defined by Proposition \ref{P:Bg-difference}, starting from any reduced expression of $w$.
\begin{theorem}\label{T:Baxter-braid}
Let $\underline{\mathrm{P}}=(P_j(u))_{j\in \mbI}$ be any tuple of Drinfeld polynomials, and fix $w\in \Wg$. 
Let $w=s_{j_1}s_{j_2}\cdots s_{j_p}$ be any reduced expression for $w$. For each $0< r\leq p$, set $w_r:=s_{j_{r+1}}\cdots s_{j_p}$, where $w_p=\mathrm{Id}$. Then one has
\begin{equation*}
\mcQ_{i,\underline{\mathrm{P}}}^w(u)
=
\prod_{r: j_r=i}\M[\braid]_{w_r}(\underline{\mathrm{P}})_{i}.
\end{equation*}
Moreover, $\M[\braid]_{w_r}(\underline{\mathrm{P}})_{j_{r}}$ is a monic polynomial in $u$ for each $r$.
\end{theorem}
\begin{proof}
We first note that the polynomiality of  $\M[\braid]_{w_r}(\underline{\mathrm{P}})_{j_{r}}$ follows from Corollary \ref{C:sigma-poly} and that $w_r^{-1}(\alpha_{j_r})=s_{j_p}\cdots s_{j_{r+1}}(\alpha_{j_r})\in \Root^+$ for each $r$, as $s_{j_p}\cdots s_{j_{r+1}}s_{j_r}$ is a reduced expression. 

Let us now establish the decomposition of $\mcQ_{i,\underline{\mathrm{P}}}^w(u)$ given in the statement of the theorem. By definition, the polynomial $\mcQ_{i,\underline{\mathrm{P}}}^w(u)$ satisfies 
\begin{equation}\label{Baxter-w-form}
\lambda_i^A(u)^{-1}A_i(u)|_{V_{w(\lambda)}}=\frac{\mcQ_{i,\underline{\mathrm{P}}}^w(u+\hbar d_i)}{\mcQ_{i,\underline{\mathrm{P}}}^w(u)}\cdot \mathrm{Id}_{V_{w(\lambda)}},
\end{equation}
where we recall that  $\lambda_i^A(u)\in 1+u^{-1}\C[\![u^{-1}]\!]$ denotes the eigenvalue of $A_i(u)$ on the highest weight space $V_\lambda$. Moreover, by Proposition \ref{P:extremal}, one has 
\begin{equation*}
A_i(u)|_{V_{w(\lambda)}}=\braid_{w}(\underline{\lambda})\left(A_i(u)\right )\cdot \mathrm{Id}_{V_{w(\lambda)}},
\end{equation*}
where $\underline{\lambda}=(\qshift^{2D}\underline{\mathrm{P}})\cdot \underline{\mathrm{P}}^{-1}$ is the $\Yhg$-highest weight of $L(\underline{\mathrm{P}})$, and we recall that $\braid_{w}(\underline{\lambda})\in (1+u^{-1}\C[\![u^{-1}]\!])^\mbI=\Hom_{Alg}(\Yhg[0],\C)$. 

Therefore, by \eqref{Baxter-w-form} and the uniqueness assertion of Lemma \ref{L:factorization}, it suffices to prove the following claim.

\noindent \textbf{Claim}. The series $\braid_{w}(\underline{\lambda})\left(A_i(u)\right)$ satisfies 
\begin{equation*}
\braid_{w}(\underline{\lambda})\left(A_i(u)\right)
=
\lambda_i^A(u)\prod_{r: j_r=i}\braid_{w_r}(\underline{\lambda})_{i}=\lambda_i^A(u)\prod_{r: j_r=i}\frac{\qshift^{2d_i}\M[\braid]_{w_r}(\underline{\mathrm{P}})_{i}}{\M[\braid]_{w_r}(\underline{\mathrm{P}})_{i}}.
\end{equation*}
\begin{proof}[Proof of claim]
The second equality is an immediate consequence of Proposition \ref{P:Bg-difference}. By \eqref{Bg-dual}, the first equality is equivalent to 
\begin{equation}\label{tau_w(A)}
\mbraid_{w^{-1}}(A_i(u))=A_i(u)\prod_{r: j_r=i}\mbraid_{w_r^{-1}}(\xi_i(u)).
\end{equation}
This follows by a simple induction on the length $p$ of $w$ using that, by Proposition \ref{P:tau-dot(A)}, one has $\mbraid_j(A_i(u))=A_i(u) \xi_i(u)^{\delta_{ij}}$. Indeed, if $p=1$, then \eqref{tau_w(A)} reduces to exactly this identity. Suppose now that \eqref{tau_w(A)} holds whenever $w$ has length $p$, and let $w'=w s_{j_{p+1}}\in \Wg$ be an element of length $p+1$. Then
\begin{align*}
\mbraid_{(w')^{-1}}(A_i(u))
&=\mbraid_{j_{p+1}}\left(\mbraid_{w}(A_i(u))\right)
\\
&=
A_i(u)\xi_i(u)^{\delta_{j_{p+1},i}}\prod_{\substack{1\leq r\leq p\\ j_r=i}}\mbraid_{j_{p+1}}\left(\mbraid_{w_r^{-1}}(\xi_i(u))\right).
\end{align*}
Since $\mbraid_{j_{p+1}}\circ \mbraid_{w_r^{-1}}=\mbraid_{s_{j_{p+1}}w_r^{-1}}
=\mbraid_{(w_r')^{-1}}$, the right-hand side of the above can be rewritten as 
\begin{equation*}
A_i(u)\xi_i(u)^{\delta_{j_{p+1},i}}\prod_{\substack{1\leq r\leq p\\ j_r=i}}\mbraid_{(w_r')^{-1}}(\xi_i(u))
=A_i(u)\prod_{\substack{1\leq r\leq p+1\\ j_r=i}}\mbraid_{(w_r')^{-1}}(\xi_i(u)),
\end{equation*}
which gives the desired result. \qedhere
\end{proof}
\let\qed\relax
\end{proof}

\subsection{Cyclicity criteria for tensor products}\label{ssec:cyclic}
One important property of the discrete invariants $\sigma_i(V)$ is that they encode information about when the tensor product of two irreducible $\Yhg$-modules is cyclic (\textit{i.e.,} highest weight) or irreducible. Namely, by Theorem 7.2 of \cite{GWPoles}, $L(\underline{\mathrm{P}})\otimes L(\underline{\mathrm{Q}})$ is cyclic provided 
\begin{equation*}
\mathsf{Z}(Q_i(u+\hbar d_i))\subset \C\setminus\Poles_i(L(\underline{\mathrm{P}})) \quad \forall\; i\in \mbI
\end{equation*}
and, by Corollary 7.3 of \cite{GWPoles}, it is irreducible provided this condition holds in addition to $\mathsf{Z}(P_i(u+\hbar d_i))\subset \C\setminus\Poles_i(L(\underline{\mathrm{Q}}))$ for all $i\in \mbI$. 

As noted in Section \ref{ssec:I-baxter}, another set of sufficient conditions for the cyclicity of any such tensor product was obtained earlier in \cite{tan-braid}*{Thm.~4.8}, and is given in terms of the action of $\Bg$ on $(\C(u)^\times)^\mbI$, following \cite{Chari-braid} closely. More precisely, let 
\begin{equation*}
w_0=s_{j_1}s_{j_2}\cdots s_{j_p}
\end{equation*}
be any reduced expression for the longest element $w_0\in \Wg$ and, as in the previous section, set $w_r= s_{j_{r+1}}\cdots s_{j_p}$ for each $0<r\leq p$. Then by \cite{tan-braid}*{Thm.~4.8} and Remark \ref{R:big-extend}, $L(\underline{\mathrm{P}})\otimes L(\underline{\mathrm{Q}})$ is cyclic provided 
\begin{equation*}
\mathsf{Z}(Q_{j_r}(u+\hbar d_{j_r}))\subset \C\setminus \mathsf{Z}(\M[\braid]_{w_r}(\underline{\mathrm{P}})_{j_r}) \quad \forall \quad 0<r\leq p.
\end{equation*}

 It was conjectured in 	\cite{GWPoles}*{\S7.5} that the two sets of conditions highlighted above are identical. In this section, we apply Theorem \ref{T:Baxter-braid} to prove this. 
 
The key result is the following corollary, which is an immediate consequence of Theorem 4.4 of \cite{GWPoles} (see \eqref{Poles=Z}) and Theorem \ref{T:Baxter-braid}.
\begin{corollary}\label{C:Poles} Let $\underline{\mathrm{P}}=(P_j(u))_{j\in \mbI}$ be any tuple of Drinfeld polynomials. Then, for each $i\in \mbI$, the $i$-th set of poles of $V=L(\underline{\mathrm{P}})$ is given by
\begin{equation*}
\Poles_i(V)
=
\bigcup_{r:j_r=i}\mathsf{Z}(\M[\braid]_{w_r}(\underline{\mathrm{P}})_i),
\end{equation*}
where the union is taken over all $0<r\leq p$ for which $j_r=i$. 
\end{corollary}
%
%
%
%
By combining Corollary \ref{C:Poles} with Theorem 7.2 of \cite{GWPoles} and Theorem 4.8 of \cite{tan-braid}, we obtain the following corollary, which in particular establishes that the cyclicity criteria obtained in \cite{GWPoles} and \cite{tan-braid} are identical.  
\begin{corollary}\label{C:Cyclic}
Let $\underline{\mathrm{P}}=(P_j(u))_{j\in \mbI}$ and $\underline{\mathrm{Q}}=(Q_j(u))_{j\in \mbI}$ be any two tuples of Drinfeld polynomials. Then the following two conditions are equivalent:
\begin{enumerate}[font=\upshape]\setlength{\itemsep}{3pt}
\item $\mathsf{Z}(Q_i(u+\hbar d_i))\subset \C\setminus\Poles_i(L(\underline{\mathrm{P}}))$ for all $i\in \mbI$.
\item  For each $0<r\leq p$, the polynomials $\M[\braid]_{w_r}(\underline{\mathrm{P}})_{j_r}$ and $Q_{j_r}(u+\hbar d_{j_r})$ have no common roots:
\begin{equation*}
\mathsf{Z}(Q_{j_r}(u+\hbar d_{j_r}))\subset \C\setminus \mathsf{Z}(\M[\braid]_{w_r}(\underline{\mathrm{P}})_{j_r}).
\end{equation*}
\end{enumerate}
Moreover, if either of these two conditions hold, then the $\Yhg$-module $L(\underline{\mathrm{P}})\otimes L(\underline{\mathrm{Q}})$ is cyclic. 
\end{corollary}

\section{Quantum loop algebras}\label{sec:qLoop}

Our goal in this section is to switch from Yangians to the parallel situation of quantum loop algebras $\UqLg$. By Lusztig's general construction of braid group actions, there is an action of $\Bg$ on $\UqLg$.  We define a modification of this action on the subalgebra $\UqLg[0]$, whose dual action on weights recovers an action of the braid group previously studied by Chari \cite{Chari-braid}.  Finally, we compare these results to the case of Yangians via the homomorphism constructed by Gautam and Toledano Laredo in \cite{GTL1}.

\subsection{Quantum loop algebras}

Consider the field $\C(q)$.  For each $i\in \mbI$ we denote $q_i = q^{d_i}$, and use the following conventions for $q$-numbers (\textit{cf.} \eqref{z-number}):
$$
[n]_q = \frac{q^n - q^{-n}}{q-q^{-1}}, \quad [n]_q! = [n]_q [n-1]_q \cdots [1]_q, \quad \sbinom{n}{k}_q= \frac{[n]_q}{[k]_q [n-k]_q}.
$$
Denote by $U_q(\g)$ the (simply-connected) Drinfeld--Jimbo quantum group corresponding to $\g$: the $\C(q)$-algebra with generators $E_i, K_i^{\pm 1}, F_i$ for $i\in \mbI$, with relations as in \cite{Lusztig} or \cite{Jantzen}*{\S 4.3}.
\begin{definition}
\label{D:quantum-loop}
    The  \textbf{quantum loop algebra} $\UqLg$ is the unital associative algebra over $\C(q)$ with generators $K_i^{\pm 1}, H_{i,r}$ and $X_{i,s}^\pm$ for $i\in \mbI$, $r \in \Z \setminus \{0\}$ and $s \in \Z$, satisfying relations:
    \begin{gather*}
    K_i K_i^{-1} = K_i^{-1} K_i = 1, \quad K_i K_j = K_j K_i,
    \\
    [K_i, H_{j,r}] = [H_{i,r}, H_{j,s}] = 0,
    \\
    K_i X_{j,r}^\pm K_i^{-1}  = q^{\pm (\alpha_i, \alpha_j)} X_{j,r}^\pm, 
    \\
    [H_{i,r}, X_{j,s}^\pm] = \pm \frac{[r a_{ij}]_{q_i}}{r} X_{j, r+s}^\pm, 
	\\
    [X_{i,r}^+, X_{j,s}^-]  = \delta_{i,j} \frac{\varphi^+_{i,r+s} - \varphi^-_{i,r+s}}{q_i-q_i^{-1}}, 
	\\
    \sum_{\pi\in S_m} \sum_{k=0}^m (-1)^k \sbinom{m}{k}_{q_i} X_{i, r_{\pi(1)}}^\pm \cdots X_{i, r_{\pi(k)}}^\pm X_{j,s}^\pm X_{i,r_{\pi(k+1)}}^\pm \cdots X_{i,r_{\pi(m)}} = 0,
    \end{gather*}
    where in the final relation $m = 1 -a_{ij}$ and $S_m$ denotes the symmetric group of degree $m$. The elements $\varphi_{i,r}^\pm$ for $i \in \mbI$ and $r \in \Z$ are defined by
    \begin{equation*}
        \varphi_i^\pm(z) = \sum_{r \geq 0} \varphi_{i,\pm r}^\pm z^{\mp r} = K_i^{\pm 1} \exp \Bigg( \pm(q_i - q_i^{-1}) \sum_{s >0} H_{i,\pm s} z^{\mp s} \Bigg),
    \end{equation*}
    and $\varphi_{i,-r}^+ = \varphi_{i,r}^- = 0$ for $r \geq 1$.
\end{definition}
The algebra $\UqLg$ has a Poincar\'{e}--Birkhoff--Witt theorem due to Beck \cite{Beck2}, and in particular admits a triangular decomposition
\begin{equation*}
\label{eq: triangular quantum loop}
    \UqLg \ = \  \UqLg[-] \otimes_{\C(q)} \UqLg[0] \otimes_{\C(q)} \UqLg[+],
\end{equation*}
where $\UqLg[0]$ is the $\C(q)$-subalgebra generated by the elements $K_i^{\pm 1}$ and $H_{i,r}^\pm$, and $\UqLg[\pm]$ are the $\C(q)$-subalgebras generated by the elements $X_{i,r}^\pm$.

We may identify the Drinfeld--Jimbo quantum group $U_q(\g)$ as the $\C(q)$-subalgebra 
\begin{equation}
\label{eq: Uq into loop}
U_q(\g) \subset \UqLg
\end{equation}
generated by the elements $E_i = X_{i,0}^+, K_i^{\pm 1},F_i = X_{i,0}^- $ for $i\in\mbI$.
\begin{remark}
\label{rmk:Beck-UqLg}
The presentation  of $\UqLg$  given in Definition \ref{D:quantum-loop} is due to Drinfeld \cite{DrNew},  and was proven by Beck \cite{Beck1}.  More precisely, consider the Drinfeld--Jimbo quantum group $U_q(\widehat{\g})$ associated to the untwisted affine Kac--Moody algebra $\widehat{\g}$, generated over $\C(q)$ by $E_i, K_i^{\pm 1}, F_i$ for $i\in \widehat{\mbI} = \mbI \cup \{0 \}$ as well as $C^{\pm 1/2}, D^{\pm 1}$, satisfying the relations from \cite{Beck1}*{\S 1.3}.  Take the $\C(q)$-subalgebra $U_q(\widetilde{\g}) \subset U_q(\widehat{\g})$ generated by $E_i, K_i^{\pm 1}, F_i$ for $i\in \widehat{\mbI}$ and the (central) elements $C^{\pm 1/2}$. Then \cite{Beck1}*{Thm.~4.7} provides an explicit isomorphism
\begin{equation*}
    \UqLg \ \cong \ U_q(\widetilde{\g}) / \langle C^{\pm 1/2}- 1\rangle.
\end{equation*}
In particular, $\UqLg$ also has Drinfeld--Jimbo generators $E_i, K_i^{\pm 1}, F_i$ for $i\in \widehat{\mbI}$, and this description is compatible with the embedding (\ref{eq: Uq into loop}) in the obvious way.
\end{remark}

Finally, we note that the algebra $\UqLg$ admits a Hopf structure, as a subquotient of the Drinfeld--Jimbo quantum group $U_q(\widehat{\g})$. But as in the Yangian case, we will more interested in its \emph{deformed Drinfeld coproduct} $\Delta_\zeta^D$ defined by Hernandez \cites{Her05,Her07}.  Following the conventions of \cite{GTL3}*{\S4.1}, this coproduct restricts to a homomorphism
$$
\Delta_\zeta^D : \UqLg[0] \longrightarrow \UqLg[0]^{\otimes 2}[\zeta,\zeta^{-1}], \quad \varphi_i^\pm(z) \mapsto \varphi_i^\pm (\zeta^{-1} z)  \otimes \varphi_i^\pm(z).
$$
Evaluating at $\zeta = 1$, we obtain a homomorphism
\begin{equation}
\label{def:Delta^1}
\Delta^1 : \UqLg[0] \longrightarrow \UqLg[0] \otimes_{\C(q)} \UqLg[0].
\end{equation}
This defines a \emph{Drinfeld Hopf algebra structure} on $\UqLg[0]$, provided we equip it with the counit $\veps^1$ and antipode $S^1$ uniquely determined by 
$$
\veps^1\big( \varphi_i^\pm(z) \big) = 1 \ \ \ \text{ and } \ \ \ S^1\big( \varphi_i^\pm(z) \big) = \varphi_i^\pm(z)^{-1} \quad \forall \quad i\in \mbI.
$$
This is the analogue of the Drinfeld Hopf structure on $\Yhg[0]$ from \eqref{def:Delta^0}.

\subsubsection{Representation theory}

We will consider representations $V$ of $\UqLg$ over $\C(q)$  which are type 1, in the sense that $V = \bigoplus_{\mu\in \Lambda} V_\mu$ where for $\mu = \sum_i \mu_i \varpi_i \in \Lambda$ an integral weight we define
\begin{equation*}
\label{eq: weight spaces quantum loop}
V_\mu = \left\{ v \in V \ : \ K_i v = q_i^{\mu_i} v \quad \forall \; i \in \mbI\right\}.
\end{equation*}
We say that $V$ is an $\ell$--highest weight module if it is generated by a vector $v$ with all $X_{i,r}^+ v = 0$, and such that
\begin{equation*}
\varphi_{i,r}^\pm v = \qweight_{i,r}^\pm v
\end{equation*}
for some elements $\qweight_{i,r}^\pm \in \C(q)$. Note that necessarily we must have $\qweight_{i,0}^+ \qweight_{i,0}^- = 1$. For any such elements we define series
\begin{equation*}
\label{eq:ell-weights}
\qweight_i^\pm(z) = \sum_{r \geq 0} \qweight^\pm_{i,\pm r} z^{\mp r} \ \in \ \C(q)[\![z^{\mp 1}]\!]
\end{equation*}
and encode these as the tuple $\underline{\qweight} = ( \qweight_i^\pm(z) )_{i \in \mbI}$. We call $v$ the $\ell$--highest weight vector, and call $\underline{\qweight}$ its $\ell$--weight. There is a unique irreducible $\ell$--highest weight module with this highest $\ell$--weight $\underline{\qweight}$.
\begin{theorem}[\mbox{\cite{ChPr2}*{Theorem 3.3}}]
\label{thm:CP Drinfeld polys}
Let $V$ be the irreducible $\ell$--highest weight representation corresponding to $\underline{\qweight}$.  Then $V$ is finite-dimensional over $\C(q)$ if and only if there exist monic polynomials $P_i(z) \in \C(q)[z]$ with $P_i(0) \neq 0$ such that
$$
\qweight^+_i(z) = q_i^{-\deg(P_i)} \frac{P_i(q_i^2 z)}{P_i(z)} = \qweight_i^-(z),
$$
in the sense that the left- and right-hand sides are the Laurent expansion of the middle rational function at $z=\infty$ and $z = 0$, respectively.
\end{theorem}
In particular, in the above theorem the action of $K_i^{\pm 1} = \varphi_{i,0}^\pm$ on the $\ell$--highest weight vector $v\in V$ is given by $\qweight_{i,0}^\pm = q_i^{\pm \deg(P_i)}$, \textit{i.e.}~the highest weight of $V$ as a $U_q(\g)$--module is $\lambda = \sum_i \deg(P_i) \varpi_i \in \Lambda$.  We will refer to $V_\lambda$ as the $\ell$--highest weight space of $V$.

\subsection{Braid group actions}
The general results of Lusztig \cite{Lusztig} provide actions of braid groups on quantum groups and on their integrable representations. Two cases will be most relevant to us:

\begin{enumerate}\setlength{\itemsep}{3pt}
    \item The affine braid group $\mathscr{B}_{\widehat{\g}}$ acts on the quantum affine algebra $U_q(\widehat{\g})$, via the operators $T''_{i, 1}$ for $i\in \widehat{\mbI}$ from \cite{Lusztig}*{\S 37}. Their action on the Drinfeld--Jimbo generators of $U_q(\widehat{\g})$ (\textit{cf}.~Remark \ref{rmk:Beck-UqLg}) can be written explicitly,  see  \cite{Beck1}*{\S 1.3}.

    Consider the subgroup $\Bg \subset \mathscr{B}_{\widehat{\g}}$.  Using the explicit formulas for the action of $\mathscr{B}_{\widehat{\g}}$ on $U_q(\widehat{\g})$, it is easy to see that the action of $\Bg$ preserves the subalgebra $U_q(\widetilde{\g}) \subset U_q(\widehat{\g})$ and passes to the quotient $\UqLg$.  We will denote the action of its generators on $\UqLg$ by $\braid_i$ for $i\in\mbI$,  following our notation from \S \ref{ssec:Bg-modified}. For any $i,j \in\mbI$, these operators satisfy
    $$
    \braid_i( K_j) = K_j K_i^{-a_{ij}}.
    $$

    \item Consider an $\ell$--highest weight representation $V$ of $\UqLg$ as in the previous section, which is finite-dimensional over $\C(q)$.  Then $V$ is an integrable representation of $U_q(\g) \subset \UqLg$, so the operators $T''_{i,1}$ for $i\in \mbI$ from \cite{Lusztig}*{\S 5} define  an action of $\Bg$ on $V$.  We will denote the action of the generators by $\braid_i^V$ for $i \in \mbI$, following our notation from \S \ref{ssec:Bg}.  On weight spaces, these operators satisfy 
    $$ \braid_i^V( V_\mu) = V_{s_i \mu}.$$
    Moreover, the identities proven in  \cite{Lusztig}*{\S 37.2} show that 
\begin{equation}
\label{eq: braid quantum loop compatible}
    \braid^V_i(a \cdot v) = \braid_i(a) \cdot \braid^V_i(v)
\end{equation}
for all $a \in \UqLg$ and $v\in V$, analogous to Proposition \ref{P:tau-wt-space}\ref{braid:e}.

\end{enumerate}





Similarly to the Yangian case considered in Section \ref{ssec:Bg-modified}, consider the subsets $N^\pm_q = \left\{ X_{i,r}^\pm : i \in \mbI, r\in \Z\right\}$, and the projection
$$
\Pi : \UqLg = \UqLg[0] \oplus ( N^-_q \UqLg + \UqLg N^+_q) \ \longrightarrow \ \UqLg[0].$$
Mimicking Definition \ref{T:mbraid}, we make the following definition:
\begin{definition}
\label{T:mbraid2}
The \textbf{modified braid group operators} on $\UqLg$ are the maps
$$
\mbraid_i : \UqLg[0] \longrightarrow \UqLg[0] \quad \text{for } i \in \mbI
$$
defined by the composition $\mbraid_i = \Pi \circ \braid_i \big|_{U_q^0(L\g)}$.
\end{definition}
Each $\mbraid_i$ is a $\C(q)$-linear endomorphism of $\UqLg[0]$, and by the same arguments as used in Corollary \ref{cor: braid rels} and Lemma \ref{lemma: braid aut Yangian},  one can see that they are algebra automorphisms which satisfy the braid relations of $\Bg$, with inverses given by $\mbraid_i^{-1} = \Pi^{op} \circ \braid_i^{-1}\big|_{\UqLg[0]}$ where $\Pi^{op}$ is defined analogously to $\Pi$ with $N^+_q$ and $N^-_q$ interchanged.  


To state the next theorem precisely, let $\mathscr{U}_q^0(L\g)$ denote the $\C[q,q^{-1}]$-subalgebra of $U_q(L\g)$ generated by the elements $K_i^{\pm 1}$ and $[r]_{q_i}^{-1}H_{i,r}$, for each $i\in \mbI$ and $r\neq 0$. The following result establishes several key properties of the operators $\mbraid_i$, and provides the main result of this section. 
\begin{theorem} \label{thm:main-q-thm}
The operators $\{\mbraid_i\}_{i \in\mbI}$ define an action of $\Bg$ on $\UqLg[0]$ by $\C(q)$-Hopf algebra automorphisms, with inverses given by $\mbraid_i^{-1} = \Pi^{op} \circ \braid_i^{-1}\big|_{\UqLg[0]}$. 
Moreover, we have:
\begin{enumerate}[label=(\alph*)] \setlength{\itemsep}{3pt}

    \item\label{main-q:a} The action of  $\mbraid_j$ on $\UqLg[0]$ is uniquely determined by the formula:
    $$
    \mbraid_j(\varphi_i^\pm(z))=\varphi_i^\pm(z)\prod_{\ell=0}^{|a_{ji}|-1}\varphi_j^\pm(zq^{- d_j(|a_{ji}|-2\ell)})^{(-1)^{\delta_{ij}}} \quad \forall \quad i\in \mbI.
    $$
    In particular, we have $\mbraid_j(K_i) = K_i K_j^{-a_{ji}}$.
    
    \item\label{main-q:b} The dual action of $\Bg$ on $\mathrm{Hom}_{Alg}(\UqLg[0], \C(q))$ generalizes an action previously studied by Chari \cite{Chari-braid}.
    
    \item\label{main-q:c} The homomorphism constructed by Gautam--Toledano Laredo in \cite{GTL1}*{Thm.~1.4} restricts to an embedding
   \begin{equation*}
    \Phi: \mathscr{U}_q^0(L\g) \into 
    \Yhg[0][\![v]\!]
    \end{equation*}
    which intertwines the modified braid group actions from Definitions \ref{T:mbraid2} and \ref{T:mbraid},  respectively. 
\end{enumerate}
\end{theorem}
We will prove this result in several steps, which occupy the remainder of this paper.      More precisely, in Section \ref{ssec: GTL} we will show that under the Gautam--Toledano Laredo map, the $\Bg$ action of $\Yhg[0]$ from Definition \ref{T:mbraid} \emph{induces} an action of $\Bg$ on $\UqLg[0]$, by operators we denote $\mbraid_i^\Phi$. By its very definition, this action satisfies the analogue of Theorem \ref{thm:main-q-thm}\ref{main-q:c}; see Corollary \ref{cor: braid group via GTL}. We will then show explicitly that this braid group action dualizes Chari's action on weights, and thus satisfies Part \ref{main-q:b} of Theorem \ref{thm:main-q-thm}; see Remark \ref{rmk:Chari-braid-1} and Proposition \ref{prop: summary of dual UqLg action}.  Finally, in Section \ref{ssec:completing the proof} we  will compare these operators and show that $\mbraid_i = \mbraid_i^\Phi$, which will complete the proof of the above theorem.

\begin{remark}
    In \cite{FiTs19}*{\S 6}, Finkelberg and Tsymbaliuk introduced $\UqLg$ versions of the GKLO generators from Section \ref{ssec:GKLO}. It seems plausible that one could prove Part \ref{main-q:a} of the above theorem by a similar calculation to Proposition \ref{P:tau-dot(A)}, though we will not study this question here.
\end{remark}

\subsection{From Yangians to quantum loop algebras}
\label{ssec: GTL}

In this section, we let $v$ be a formal variable and define $\mathbb{Y}_{v}(\mfg)$ to be the Rees algebra of $\Yhg$ over $\C[v]$.  That is, it is the graded $\C[v]$-algebra 
\begin{equation*}
\mathbb{Y}_{v}(\mfg)=\bigoplus_{n\geq 0} v^n\mathcal{F}_n\Yhg \subset \Yhg[][v],
\end{equation*}
where $\deg(v)=1$. 
Similarly, we let $\mathbb{Y}_{v}^0(\mfg)$ denote the Rees algebra of $\Yhg[0]$ --- this is precisely the subalgebra of $\mathbb{Y}_{v}(\mfg)$ generated by $\{v^r \xi_{i,r}\}_{i\in \mbI,r\geq 0}$. 

Note that, for each $j\in \mbI$, the operator $\mbraid_j$ extends naturally to a $\C[v]$-algebra automorphism of $\Yhg[0][v]$ which preserves each component $\mathbb{Y}_{v}^0(\mfg)_n=v^n\mathcal{F}_n\Yhg$ of $\mathbb{Y}_{v}^0(\mfg)$ as $\mbraid_j$ is a filtered map of degree zero. These operators define an action of $\Bg$ on $\mathbb{Y}_{v}^0(\mfg)$ by graded algebra automorphisms, and thus an action of $\Bg$ on the formal completion 
\begin{equation*}
\widehat{\mathbb{Y}_{v}^0(\mfg)}=\prod_{n\geq 0}\mathbb{Y}_{v}^0(\mfg)_n \subset \Yhg[0][\![v]\!]
\end{equation*}
by $\C[\![v]\!]$-algebra automorphisms. In what follows, we shall also view $\widehat{\mathbb{Y}_{v}^0(\mfg)}$ as a $\C[q,q^{-1}]$-algebra via the embedding 
$
\C[q,q^{-1}]\into \C[\![v]\!]$ given by $q\mapsto e^{\frac{v}{2}}$.

Recall from above the statement of Theorem \ref{thm:main-q-thm} that $\mathscr{U}_q^0(L\g)$ is the $\C[q,q^{-1}]$-subalgebra of $U_q(L\g)$ generated by $K_i^{\pm 1}$ for each $i\in \mbI$, together with the elements 
\begin{equation*}
\msh_{i,r}:=\frac{H_{i,r}}{[r]_{q_i}}
\end{equation*}
for all $i\in \mbI$ and nonzero $r\in \Z$. By Theorem 1.4 of \cite{GTL1}, there is an embedding of $\C[q,q^{-1}]$-algebras 
\begin{equation*}
\Phi:\mathscr{U}_q^0(L\mfg)\into \widehat{\mathbb{Y}_{v}^0(\mfg)}
\end{equation*}
uniquely determined by the formulas
\begin{equation}
\label{eq:Phi-explicitly}
\Phi(K_i)=\exp\big((v/2)\xi_{i,0}\big) \quad \text{ and }\quad \Phi(\msh_{i,r})=\frac{v}{q_i^r-q_i^{-r}}\bt_i\big(rv/\hbar\big)
\end{equation}
for all $i\in \mbI$ and nonzero $r\in \Z$, where we recall that $\bt_i(u)$ is the Borel transform of $t_i(u)$; see below \eqref{diff:R^0}. More precisely, $\Phi$ is obtained by restricting the embedding $\mathbb{U}_v(L\mfg)\into \widehat{\mathbb{Y}_v(\mfg)}$ constructed in \cite{GTL1}*{Thm.~1.4} to $\mathscr{U}_q^0(L\mfg)\subset \mathbb{U}_v(L\mfg)$, where $\mathbb{U}_v(L\mfg)$ is the $v$-adic analogue of $U_q(L\mfg)$; see \cite{GTL1}*{\S2.3}.

The homomorphism $\Phi$ is easily seen to intertwine Drinfeld Hopf algebra structures on both sides.  Indeed, on the one hand, the subalgebra $\mathscr{U}_q^0(L\g) \subset \UqLg[0]$ inherits a Hopf algebra structure over $\C[q,q^{-1}]$ by restricting the Drinfeld Hopf structure (\ref{def:Delta^1}) from $\UqLg[0]$.  On the other hand, $\widehat{\mathbb{Y}_{v}^0(\mfg)}$ is naturally a topological Hopf algebra over $\C[v]$ under the Drinfeld Hopf algebra structure determined by (\ref{def:Delta^0}).  These two structures are intertwined by $\Phi$ since this is true on all generators in equation (\ref{eq:Phi-explicitly}): $K_i$ is grouplike and $\xi_{i,0}$ is primitive, while the elements $\msh_{i,r}$ and $t_{i,s}$ are all primitive. (Similarly, the counits and antipodes are intertwined.)

\begin{proposition}\label{P:GTL}
Let $j\in \mbI$. Then there is a unique $\C[q,q^{-1}]$-Hopf algebra automorphism $\mbraid_j^\Phi$ of $\mathscr{U}_q^0(L\g)$ satisfying $\Phi\circ\mbraid_j^\Phi=\mbraid_j \circ \Phi$. It is given explicitly by the formulas
\begin{equation*}
\mbraid_j^\Phi(K_i)=K_{s_j(\alpha_i)}=K_iK_j^{-a_{ji}}\quad \text{ and }\quad \mbraid_j^\Phi(\msh_{i,r})=\msh_{i,r}-q_j^r\frac{[ra_{ij}]_{q_i}}{[r]_{q_i}} \msh_{j,r}
\end{equation*}
for all $i\in \mbI$ and nonzero $r\in \Z$. 

\end{proposition}
\begin{proof}
Let $j\in \mbI$. Since $\Phi$ is injective, to establish the existence of a (necessarily unique) $\C[q,q^{-1}]$-Hopf algebra automorphism $\mbraid_j^\Phi$ of $\mathscr{U}_q^0(L\g)$ satisfying $\Phi\circ\mbraid_j^\Phi=\mbraid_j \circ \Phi$, it is sufficient to show that, for each $i\in \mbI$ and $r\in \Z\setminus\{0\}$, one has
\begin{equation}\label{GTL-suff}
\mbraid_j(\Phi(K_i))=\Phi(K_iK_j^{-a_{ji}})\quad \text{ and }\quad
\mbraid_j(\Phi(\msh_{i,r}))=\Phi\left(\msh_{i,r}-q_j^r\frac{[ra_{ij}]_{q_i}}{[r]_{q_i}} \msh_{j,r}\right).
\end{equation}
Note that a proof of these relations will simultaneously prove the second assertion of the proposition.
The first relation of \eqref{GTL-suff} follows from the fact that, for each $i\in \mbI$, we have 
\begin{equation*}
\mbraid_j(\Phi(K_i))=\exp\Big(\frac{v}{2}s_i(\xi_{i,0})\Big)=\exp\Big(\frac{v}{2}(\xi_{i,0}-a_{ji}\xi_{j,0})\Big)=\Phi(K_iK_j^{-a_{ji}}).
\end{equation*}
Hence, we are left to verify the second relation of \eqref{GTL-suff}. 
To this end, recall from Corollary \ref{C:mbraid-t-xi} that $\mbraid_j(t_{i,k})$ is uniquely determined by the formula 
\begin{equation*}
\mbraid_j(t_i(u))=t_i(u)-\qshift^{-d_j}[a_{ji}]_{\qshift^{d_j}}(t_j(u)),
\end{equation*}
 where $\qshift=e^{\frac{\hbar}{2}\partial_u}$, viewed as an operator on $\Yhg[0][\![u^{-1}]\!]$.  From this and the formula \eqref{qshift-poly}, 
we deduce that $\mbraid_j(t_{i,k})$ is given explicitly by 
 \begin{equation*}
 \mbraid_j(t_{i,k})=t_{i,k}+(-1)^{\delta_{ij}}\sum_{\ell=0}^{|a_{ji}|-1}\sum_{b=0}^k \binom{k}{b}t_{j,b} (a_{ji;\ell})^{k-b},
 \end{equation*}
 where we have set $a_{ji;\ell}:=\frac{\hbar d_j }{2}(|a_{ji}|-2\ell)$. Using the definition of $\Phi$, we then obtain
\begin{align*}
(\mbraid_j-\id)\Phi(\msh_{i,r})&=\frac{v (-1)^{\delta_{ij}}}{q_i^r-q_i^{-r}}\sum_{k\geq 0} \frac{(rv/\hbar)^k}{k!}\sum_{\ell=0}^{|a_{ji}|-1}\sum_{b=0}^k \binom{k}{b}t_{j,b} (a_{ji;\ell})^{k-b}\\
&
=\frac{v (-1)^{\delta_{ij}}}{q_i^r-q_i^{-r}} \sum_{b\geq 0} t_{j,b}\frac{(rv/\hbar)^b}{b!} \sum_{\ell=0}^{|a_{ji}|-1} \sum_{k\geq b}\frac{((rv/\hbar)a_{ji;\ell})^{k-b}}{(k-b)!}
\\
&
=\left( (-1)^{\delta_{ij}} \frac{q_j^r-q_j^{-r}}{q_i^r-q_i^{-r}} \sum_{b=0}^{|a_{ji}|-1} q^{d_j r(|a_{ji}|-2\ell)} \right)\cdot  \Phi(\msh_{j,r})
\end{align*}
where in the last equality we have used the definition of $a_{ji;\ell}$ and $q=e^{\frac{v}{2}}$. Next, observe that since
$\sum_{b=0}^{|a_{ji}|-1} q^{r d_j (|a_{ji}|-2\ell)}=q_j^r [ |a_{ji}|]_{q^r_j}$, we have 
\begin{align*}
(-1)^{\delta_{ij}}\frac{q_j^r-q_j^{-r}}{q_i^r-q_i^{-r}} \sum_{b=0}^{|a_{ji}|-1} q^{d_j r(|a_{ji}|-2\ell)} 
&=
-\frac{q_j^r-q_j^{-r}}{q_i^r-q_i^{-r}}q_j^r [ a_{ji}]_{q^r_j}=-q_j^r \frac{[ra_{ij}]_{q_i}}{[r]_{q_i}}.
\end{align*}
Combining this with the above expression for $(\mbraid_j-\id)\Phi(\msh_{i,r})$ yields the second relation of \eqref{GTL-suff} and thus completes the proof of the proposition. \qedhere
\end{proof}

\begin{corollary}
\label{cor: braid group via GTL}
 For each $j\in \mbI$, the operator
$\mbraid_j^\Phi$ uniquely extends to a $\C(q)$-Hopf algebra automorphism of $U_q^0(L\g)$ satisfying 
\begin{equation*}
\mbraid_j^\Phi(\varphi_i^\pm(z))=\varphi_i^\pm(z)\prod_{\ell=0}^{|a_{ji}|-1}\varphi_j^\pm(zq^{- d_j(|a_{ji}|-2\ell)})^{(-1)^{\delta_{ij}}} \quad \forall \quad i\in \mbI.
\end{equation*}
Moreover, the assignment $\braid_j\mapsto \mbraid_j^\Phi$ defines an action of $\Bg$ on $U_q^0(L\g)$. 
\end{corollary}
\begin{proof}
That the $\mbraid_j^\Phi$ uniquely extend to $\C(q)$-Hopf algebra automorphisms of $U_q^0(L\g)$ satisfying the defining braid relations of $\Bg$ follows from the tensor decomposition $U_q^0(L\g)\cong\mathscr{U}_q^0(L\g)\otimes_{\C[q,q^{-1}]}\C(q)$, Proposition \ref{P:GTL}, and that the operators $\{\mbraid_j\}_{j\in \mbI}$ define an action of $\Bg$ on $\widehat{\mathbb{Y}_v^0(\mfg)}$. Hence, we are left to check that $\mbraid_j^\Phi(\varphi_i^\pm(z))$ is given as claimed in the statement of the corollary. To this end, observe that by Proposition \ref{P:GTL}, we have 
\begin{equation*}
\mbraid_j^\Phi(H_{i,r})=H_{i,r}-q_j^r\frac{[ra_{ij}]_{q_i}}{[r]_{q_j}} H_{j,r}=H_{i,r}- \frac{q_j-q_j^{-1}}{q_i-q_i^{-1}} q_j^r[a_{ji}]_{q_j^r}H_{j,r}.
\end{equation*}
It follows that 
\begin{align*}
\mbraid_j^\Phi(\varphi_i^\pm(z))
&
=\mbraid_j^\Phi(K_i)^{\pm 1} \exp\!\left(\pm(q_i-q_i^{-1})\sum_{r>0} \mbraid_j^\Phi(H_{i,\pm r})z^{\mp r}\right)\\
&
=\varphi_i^\pm(z)  K_j^{\mp a_{ji}} \exp\!\left(\mp(q_j-q_j^{-1})\sum_{r>0} q_j^{\pm r}[a_{ji}]_{q_j^{\pm r}}H_{j,\pm r}z^{\mp r}\right).
\end{align*}
Substituting the equality $-q^{\pm r}_j[a_{ji}]_{q^{\pm r }_j}=(-1)^{\delta_{ij}}\sum_{\ell=0}^{|a_{ji}|-1}q^{\pm r d_j(|a_{ji}|-2\ell)}$ into this formula outputs the claimed expression for $\mbraid_j^\Phi(\varphi_i^\pm(z))$.
\end{proof}

\begin{remark}
\label{rmk:Chari-braid-1}
Following \cite{Chari-braid}, define $\msh_i(u)=\sum_{r>0}\msh_{i,r}u^r$ for each $i\in \mbI$. Then the formulas of Proposition \ref{P:GTL} yield 
\begin{equation*}
\mbraid^\Phi_j(\msh_i(u))
=\msh_i(u)+(-1)^{\delta_{ij}}\sum_{b=0}^{|a_{ij}|-1} \msh_j(q_i^{(|a_{ij}|-1-2b)}q_ju) \quad \forall\quad i,j\in \mbI.
\end{equation*}
Equivalently, one has $\mbraid^\Phi_j(\msh_i(u))=\msh_i(u)$ if $a_{ij}=0$, $\mbraid^\Phi_j(\msh_j(u))=-\msh_j(q_j^2u)$, while 
\begin{equation*}
\mbraid^\Phi_j(\msh_i(u))
=
\begin{cases}
\msh_i(u)+\msh_j(q_j u) &\; \text{ if }\; a_{ij}=-1,\\
\msh_i(u)+\msh_j(q^3u)+\msh_j(qu)  &\; \text{ if }\; a_{ij}=-2,\\
\msh_i(u)+\msh_j(q^5u)+\msh_j(q^3u)+\msh_j(qu)&\; \text{ if }\; a_{ij}=-3.
\end{cases}
\end{equation*}
We note that these are exactly the formulas from \cite{Chari-braid}*{(3.4)}, up to a missing subscript $j$ in \cite{Chari-braid} in the case where $a_{ij}=-1$. More precisely, the action of $\Bg$ constructed therein is realized on the space 
\begin{equation*}
\mathcal{A}^n:=\Hom_{Alg}(U_q^0(L\g)^+,\C(q))\cong \big\{(\mu_i(u))_{i\in \mbI}\in \C(q)[\![u]\!]: \mu_i(0)=0 \quad \forall \; i\in \mbI\big\},
\end{equation*}
where $U_q^0(L\g)^+$ is the subalgebra of $U_q^0(L\g)$ generated by $\{\msh_{i,r}\}_{i\in \mbI,r>0}$. The relation between the two actions will be discussed in more detail in the next section below.  
\end{remark}
\begin{remark}
\label{rmk:q-Hecke}
Similarly to Section \ref{ssec:Hecke}, the  action of $\Bg$ on $\UqLg[0]$ from Corollary \ref{cor: braid group via GTL} gives rise to an action of the Hecke algebra $\mathscr{H}_z(\mfg)$.  To see this, define a $\C(q)$-algebra automorphism $\msz$ of $\UqLg[0]$ by $\msz(K_i) = K_i$ and $\msz( \msh_{i,r}) = q_i^{-r} \msh_{i,r}$. Next, similarly to Proposition \ref{P:Hecke}, define operators $T_i = \msz^{d_i} \circ \mbraid_i^\Phi$ on $\UqLg[0]$, for each $i\in\mbI$.  Explicitly, using the formulas from Proposition \ref{P:GTL} and Corollary \ref{cor: braid group via GTL}, these algebra automorphisms are determined by:
$$
T_j( \msh_{i,r}) = q_j^{-r} \msh_{i,r} - [a_{ij}]_{q_i^r} \msh_{j,r}.
$$
For any non-zero integer $r\in \Z$, consider the finite-dimensional $\C(q)$-vector space 
$$
\mathbb{U}_r = \mathrm{span}_{\C(q)}\{ \msh_{i,r} : i \in \mbI\} = \mathrm{span}_{\C(q)}\{ H_{i,r} : i \in \mbI\}.
$$
Then the operators $\{T_i\}_{i \in \mbI}$ define an action of $\mathscr{H}_z(\mfg)$ on $\mathbb{U}_r$, where $z$ acts by $\msz$.  This can be proven by explicit calculations similar to the proof of Proposition \ref{P:Hecke}.  
\end{remark}

\subsection{Dual action on $\ell$--weights}
 \label{ssec:wts-dual2}
 By Corollary \ref{cor: braid group via GTL}, there is an action of $\Bg$ on $\UqLg[0]$ defined by the operators $\{\mbraid_i^\Phi\}_{i \in \mbI}$. For $\upsigma \in \Bg$, let us denote the corresponding operator on $\UqLg[0]$ by $\upsigma^\Phi$.  Analogously to Ssection \ref{ssec:wts-dual}, we dualize to obtain an action of $\Bg$ on $\mathrm{Hom}_{Alg}(\UqLg[0], \C(q))$, via the rule
\begin{equation} \label{Bg-dual2}
    \upsigma(\gamma)(y)=\gamma\big(\upvartheta(\upsigma)^\Phi \cdot y\big)
\end{equation}
for any $\upsigma \in \Bg, y \in \UqLg[0]$, and $\gamma \in \mathrm{Hom}_{Alg}(\UqLg[0], \C(q))$. In other words, this action is determined by the requirement that $\braid_i \in \Bg$ acts by the transpose/adjoint of $\mbraid_i^\Phi$, so that $\braid_i(\gamma) = (\mbraid_i^\Phi)^\ast (\gamma) = \gamma \circ \mbraid_i^\Phi$.

Each homomorphism $\gamma \in \mathrm{Hom}_{Alg}(\UqLg[0], \C(q))$ may be encoded by a tuple $\underline{\qweight} = (\qweight_i^\pm(z))_{i \in\mbI}$, where $\qweight_i^\pm(z) = \gamma( \varphi_i^\pm(z))$. This correspondence allows us to identify $\mathrm{Hom}_{Alg}(\UqLg[0], \C(q))$ with the following space of tuples of formal series:
$$
\left\{ \big( \qweight_i^\pm(z) \big)_{i \in \mbI} \in \C(q)[\![z^{\mp 1}]\!]^\mbI \ :\  \forall i\in \mbI, \ \qweight_{i,0}^+ \qweight_{i,0}^- = 1 \text{ where } \qweight_i^\pm(z) = \sum_{r \geq 0} \qweight_{i, \pm r}^\pm z^{\mp r}\right\}.
$$
The Drinfeld Hopf algebra structure (\ref{def:Delta^1}) on $\UqLg[0]$ induces a group structure on this set, which is a subgroup of the group $\big( \C(q)[\![z^{-1}]\!]^\times \times \C(q)[\![z]\!]^\times\big)^\mbI$ under componentwise multiplication. 

With this identification in mind, we have the following analogue of Corollary \ref{C:T_1-rep}, which follows by an analogous proof:
\begin{corollary}\label{C:braid-qweight}
The formula (\ref{Bg-dual2}) defines an action of the braid group $\Bg$ on $\mathrm{Hom}_{Alg}(\UqLg[0], \C(q))$ by group automorphisms. Moreover, for any element $\underline{\qweight} = ( \qweight_i^\pm(z))_{i\in\mbI}\in \mathrm{Hom}_{Alg}(\UqLg[0], \C(q))$, the components $\braid_j( \underline{\qweight})_i^\pm$ of $\braid_j(\underline{\qweight})$ are given by
$$
\braid_j ( \underline{\qweight})_i^\pm = \qweight_i^\pm(z)\prod_{\ell=0}^{|a_{ji}|-1}\qweight_j^\pm(zq^{- d_j(|a_{ji}|-2\ell)})^{(-1)^{\delta_{ij}}}.
$$
\end{corollary}
By restricting to the subalgebra $\UqLg[0]^+ \subset \UqLg[0]$ from Remark \ref{rmk:Chari-braid-1}, we also obtain an induced action of $\Bg$ on 
$$\mathcal{A}^n = \mathrm{Hom}_{Alg}( \UqLg[0]^+, \C(q)) \cong \big\{(\mu_i(u))_{i\in \mbI}\in \C(q)[\![u]\!]: \mu_i(0)=0 \quad \forall \; i\in \mbI\big\}.$$
More explicitly, for any $\gamma \in \mathrm{Hom}_{Alg}(\UqLg[0], \C(q))$ we consider the images 
$$\mu_i(u) := \gamma( \msh_i(u)) \in \C(q)[\![u]\!]$$ 
of the series $\msh_i(u) = \sum_{r>0} \msh_{i,r} u^r$, and encode these as a tuple $\underline{\mu} = (\mu_i(u))_{i \in \mbI} \in \mathcal{A}^n$. Recall that the action of $\Bg$ on the generating series $\msh_i(u)$ for $\UqLg[0]^+$ was described in Remark \ref{rmk:Chari-braid-1}. We immediately deduce expressions for the components $\braid_j(\underline{\mu})_i$ of the image $\braid_j( \underline{\mu}) \in \mathcal{A}^n$ under any generator $\braid_j \in \Bg$: we have $\braid_j(\underline{\mu})_i =\mu_i(u)$ if $a_{ij}=0$, $\braid_j(\underline{\mu})_j=-\mu_j(q_j^2u)$, and finally
\begin{equation*}
\braid_j(\underline{\mu})_i
=
\begin{cases}
\mu_i(u)+\mu_j(q_j u) &\; \text{ if }\; a_{ij}=-1,\\
\mu_i(u)+\mu_j(q^3u)+\mu_j(qu)  &\; \text{ if }\; a_{ij}=-2,\\
\mu_i(u)+\mu_j(q^5u)+\mu_j(q^3u)+\mu_j(qu)&\; \text{ if }\; a_{ij}=-3.
\end{cases}
\end{equation*}
This exactly recovers Chari's braid group action on $\mathcal{A}^n$ from \cite{Chari-braid}*{(3.4)}, up to a missing subscript $j$ in \cite{Chari-braid} in the case $a_{ij}=-1$. 

This discussion puts us in place to apply Chari's result \cite{Chari-braid}*{Prop.~4.1} on the $\ell$--weights of extremal vectors in representations of $\UqLg$, which provides the quantum loop analogue of Proposition \ref{P:extremal}.  We summarize this as follows:

\begin{proposition}
\label{prop: summary of dual UqLg action}
The action of $\Bg$ on $\mathcal{A}^n = \mathrm{Hom}_{Alg}(\UqLg[0]^+, \C(q))$ agrees with Chari's action \cite{Chari-braid}*{(3.4)}. 
Consequently, for any  finite-dimensional irreducible $\ell$--highest weight representation $V$ of $\UqLg$ with $\ell$--highest weight space $V_\lambda$, the following two assertions hold: 
\begin{enumerate}[label=(\alph*)]\setlength{\itemsep}{3pt}
\item\label{q-dual:a} Suppose that the action of $\msh_i(u)$ on the $\ell$--highest weight space $V_{\lambda}$ is encoded by the tuple $\underline{\mu} = (\mu_i(u))_{i \in \mbI} \in\mathcal{A}^n$.  Then for any $w\in \Wg$,  we have
$$
\msh_i(u)|_{V_{w(\lambda)}} = \braid_w(\underline{\mu})_i \cdot \mathrm{Id}_{V_{w(\lambda)}},
$$
where $\braid_w(\underline{\mu}) \in \mathcal{A}^n$ is the image of $\underline{\mu}$ under $\braid_w \in \Bg$.

\item\label{q-dual:b} Suppose that the action of the series $\varphi_i^\pm(z)$ on the $\ell$--highest weight space is encoded by the tuple $\underline{\qweight} = ( \qweight_i^\pm(z))_{i \in \mbI} \in \mathrm{Hom}_{Alg}(\UqLg[0], \C(q))$.  Then, for each $w\in \Wg$, we have
$$
\varphi_i^\pm(z) |_{V_{w(\lambda)}} = \braid_w(\underline{\qweight})_i \cdot \mathrm{Id}_{V_{w(\lambda)}},
$$
where $\braid_i( \underline{\qweight}) \in \mathrm{Hom}_{Alg}(\UqLg[0], \C(q))$ is the image of $\underline{\qweight}$ under $\braid_w \in \Bg$.
\end{enumerate}
\end{proposition}
\begin{proof}
We have already seen above that the action of $\Bg$ on $\mathcal{A}^n$ agrees with Chari's.  Therefore, Part \ref{q-dual:a} of the proposition follows immediately from \cite{Chari-braid}*{\S4.1}.  For Part \ref{q-dual:b}, we first observe that Chari's proof applies equally well to the elements $\msh_{i,r}$ for $r < 0$.  By taking exponentials and using the fact that $\braid_j(K_i) = K_i K_j^{-a_{ji}}$, we deduce that Part \ref{q-dual:b} holds.
\end{proof}

\subsection{Completing the proof of Theorem \ref{thm:main-q-thm}}
\label{ssec:completing the proof}

Recall that in Definition \ref{T:mbraid2} we introduced modified braid group operators $\mbraid_i$ on $\UqLg[0]$, defined via Lusztig's braid group operators $\braid_i$ on $\UqLg$. Recall also that $\Bg$ acts on any irreducible finite-dimensional $\ell$--highest weight representation $V$ of $\UqLg$, via operators $\braid_i^V$.  

\begin{lemma}
\label{lemma: Lusztig-braid-extremal}
Let $V$ be as above, with $\ell$--highest weight space $V_\lambda$.  Then for any $i\in \mbI$ and any $a\in \UqLg[0]$, the eigenvalue of $a$ acting on $V_{s_i(\lambda)}$ is the same as the eigenvalue of $\mbraid_i(a)$ on $V_\lambda$.
\end{lemma}
\begin{proof}
This follows from the same argument as Proposition \ref{P:extremal}:

 Choose an $\ell$--highest weight vector $v \in V_\lambda$.  Then  $\braid_i^V(v) \in V_{s_i(\lambda)}$, while $v': = (\tau_i^V)^2(v)\in V_{s_i^2(\lambda)} = V_\lambda$ is another $\ell$--highest weight vector. By the compatibility condition  (\ref{eq: braid quantum loop compatible}), we have:
$$
\braid_i^V\big( a \cdot \braid_i^V(v)) = \braid_i(a) \cdot (\braid_i^V)^2(v) = \braid_i(a) \cdot v' = \mbraid_i(a) \cdot v'.$$
The left-hand side reflects the eigenvalue of $a$ on $V_{s_i(\lambda)}$, and the right-hand side the eigenvalue of $\mbraid_i(a)$ on $V_\lambda$, proving the claim.
\end{proof}

On the other hand, in the previous two sections we used the Gautam--Toledano Laredo homomorphism to obtain an action of $\Bg$ on $\UqLg[0]$ by operators $\mbraid_i^\Phi$ (Proposition \ref{P:GTL} and Corollary \ref{cor: braid group via GTL}), and dualized this to an action of $\Bg$ on $\mathrm{Hom}_{Alg}(\UqLg[0], \C(q))$. We have seen that this dual action extends an action of $\Bg$ on $\mathcal{A}^n = \mathrm{Hom}_{Alg}( \UqLg[0]^+, \C(q))$ studied by Chari, and thereby allows us to express the action of $\UqLg[0]$ on extremal $\ell$--weight spaces $V_{w(\lambda)}$ (Proposition \ref{prop: summary of dual UqLg action}).  

In particular, this shows that the operators $\mbraid_i^\Phi$ satisfy the same property from the previous lemma: for any finite-dimensional irreducible representation $V$ with $\ell$--highest weight space $V_\lambda$, and for any $a \in \UqLg[0]$, the eigenvalue of $a$ acting on $V_{s_i(\lambda)}$ is equal to the eigenvalue of $\mbraid_i^\Phi(a)$ on $V_\lambda$.  Indeed, this is an immediate consequence of Proposition \ref{prop: summary of dual UqLg action} and the definition of (Chari's) dual action on $\mathrm{Hom}_{Alg}(\UqLg[0], \C(q))$.  

\begin{proof}[Proof of Theorem \ref{thm:main-q-thm}]
The fact that the $\mbraid_i$ define an action of the braid group was explained just after Definition \ref{T:mbraid2}, as was the equality $\mbraid_i^{-1} = \Pi^{op} \circ \braid_i^{-1} \big|_{\UqLg[0]}$.  We note that the analogue of equation (\ref{eq: braid estimate}) follows for example from the Saito's description \cite{Saito} of $\braid_i^{\pm 1}$ as conjugation by certain explicit elements of a completion of $U_{q_i}(\mathfrak{sl}_2)$. Next we claim that, for each $i\in \mbI$, the operators $\mbraid_i$ and $\mbraid_i^\Phi$ on $\UqLg[0]$ coincide:
\begin{equation}
\label{eq: claim final proof}
\mbraid_i = \mbraid_i^\Phi \quad \forall \quad i \in \mbI.
\end{equation}
Assuming this claim for the moment, we can complete the proof of Theorem \ref{thm:main-q-thm}.   Part \ref{main-q:a} of the theorem  follows from Corollary \ref{cor: braid group via GTL}, as does the fact that the action respects the $\C(q)$--Hopf structure. Part \ref{main-q:b} follows from Proposition \ref{prop: summary of dual UqLg action}, and Part \ref{main-q:c} is precisely Proposition \ref{P:GTL} and Corollary \ref{cor: braid group via GTL}.

Finally, to prove the claim (\ref{eq: claim final proof}), we assume towards a contradiction that there exists an element $a\in \UqLg[0]$ such that $\mbraid_i(a) \neq \mbraid_i^\Phi(a)$.  Then by Lemma \ref{lem: technical} below, there exists a representation $V$ such that the difference $\mbraid_i(a) - \mbraid_i^\Phi(a)$ acts \emph{non-trivially} on its $\ell$--highest weight space $V_\lambda$.   But by Lemma \ref{lemma: Lusztig-braid-extremal} and the discussion afterwards, the action of $\mbraid_i(a)- \mbraid_i^\Phi(a)$ on $V_\lambda$ is given by the same scalar as $a-a = 0$ acting on $V_{s_i(\lambda)}$, a contradiction. It follows that $\mbraid_i = \mbraid_i^\Phi$ on $\UqLg[0]$, completing the proof.
\end{proof}

\begin{lemma}
\label{lem: technical}
    For any non-zero element $a \in U_q^0(L\g)$, there exists an irreducible finite-dimensional $\ell$--highest weight representation $V$ of $\UqLg$ such that $a$ acts non-trivially on its $\ell$--highest weight space $V_\lambda$.
\end{lemma}
\begin{proof}
Let $V$ be an irreducible finite-dimensional $\ell$--highest weight representation of $\UqLg$, with Drinfeld polynomials $P_i(z)$, \textit{cf}.~Theorem \ref{eq: braid quantum loop compatible}. For any $i\in \mbI$ and nonzero $r \in \Z$, the action of the element $H_{i,r} \in \UqLg$ on the $\ell$-highest weight space $V_\lambda$ is given (up to rescaling by a constant independent of $V$) by the power sum symmetric (Laurent) polynomial $p_r = \sum_k a_{i,k}^r$ in the roots $\{a_{i,1},\ldots, a_{i, \deg( P_i)}\}$ of $P_i(z)$, see \cite{Chari-braid}*{Thm.~2.4} or \cite{GTL1}*{Prop.~4.4(1)}.  Therefore, the lemma is a consequence of the following result about rings of symmetric Laurent polynomials.
\end{proof}

\begin{lemma}
    Let $k$ be a field of characteristic 0, and consider the ring $A = k[x_1^{\pm 1}, \ldots, x_n^{\pm 1}]^{S_n}$ of symmetric Laurent polynomials in $n$ variables. For each integer $r \geq 1$, consider its elements:
    $$p_r^+ = x_1^r + \ldots + x_n^r, \quad p_r^- = x_1^{-r} + \ldots + x_{n}^{-r}$$
    If $n \geq r+s$, then the following elements are algebraically independent:
    $$
    p_1^+,\ldots, p_r^+, p_1^-,\ldots, p_s^- 
    $$
\end{lemma}
\begin{proof}
    Consider the subrings $A^+ = k[p_1^+,\ldots,p_r^+]$ and $A^- = k[p_1^-,\ldots,p_s^-]$.  By standard results about symmetric polynomials, these are both polynomial rings with alternative  generators
    $$
    A^+ = k[e_1^+,\ldots, e_r^+], \qquad A^- = k[e_1^-,\ldots,e_s^-],
    $$
    where $e_\ell^+$ (resp.~$e_\ell^-$) denotes the $\ell$-th elementary symmetric polynomial in the variables $x_1,\ldots,x_n$ (resp.~the variables $x_1^{-1},\ldots,x_n^{-1}$).  

    Now, note that $e_\ell^- = e_{n-\ell}^+ / e_n^+$. Thus $A^- $ is generated by $e_{n-1}^+/e_n^+,\ldots, e_{n-s}^+/ e_n^+$. So long as $r+s \leq n$, it is easily seen that the elements 
    $$e_1^+,\ldots, e_r^+, e_{n-s}^+/e_n^+,\ldots, e_{n-1}^+/e_n^+$$
    are algebraically independent in $A$. It follows that $p_1^+,\ldots, p_r^+, p_1^-, \ldots, p_s^-$ must also be algebraically independent, which completes the proof.
\end{proof}

%
%




\bibliography{Yangians}

\newcommand{\noop}[1]{}
\begin{bibdiv}
\begin{biblist}

\bib{AK-qAffine}{article}{
      author={Akasaka, T.},
      author={Kashiwara, M.},
       title={Finite-dimensional representations of quantum affine algebras},
        date={1997},
        ISSN={0034-5318},
     journal={Publ. Res. Inst. Math. Sci.},
      volume={33},
      number={5},
       pages={839\ndash 867},
}

\bib{Beck1}{article}{
      author={Beck, J.},
       title={Braid group action and quantum affine algebras},
        date={1994},
        ISSN={0010-3616,1432-0916},
     journal={Comm. Math. Phys.},
      volume={165},
      number={3},
       pages={555\ndash 568},
}

\bib{Beck2}{article}{
      author={Beck, J.},
       title={Convex bases of {PBW} type for quantum affine algebras},
        date={1994},
        ISSN={0010-3616,1432-0916},
     journal={Comm. Math. Phys.},
      volume={165},
      number={1},
       pages={193\ndash 199},
}

\bib{Chari-braid}{article}{
      author={Chari, V.},
       title={Braid group actions and tensor products},
        date={2002},
        ISSN={1073-7928},
     journal={Int. Math. Res. Not.},
      number={7},
       pages={357\ndash 382},
}

\bib{ChPr1}{article}{
      author={Chari, V.},
      author={Pressley, A.},
       title={Fundamental representations of {Y}angians and singularities of
  {$R$}-matrices},
        date={1991},
        ISSN={0075-4102},
     journal={J. Reine Angew. Math.},
      volume={417},
       pages={87\ndash 128},
}

\bib{chari-pressley-qaffine}{article}{
      author={Chari, V.},
      author={Pressley, A.},
       title={Quantum affine algebras},
        date={1991},
     journal={Comm. Math. Phys.},
      volume={142},
       pages={261\ndash 283},
}

\bib{CPBook}{book}{
      author={Chari, V.},
      author={Pressley, A.},
       title={A guide to quantum groups},
   publisher={Cambridge University Press, Cambridge},
        date={1994},
        ISBN={0-521-43305-3},
}

\bib{ChPr2}{incollection}{
      author={Chari, V.},
      author={Pressley, A.},
       title={Quantum affine algebras and their representations},
        date={1995},
   booktitle={Representations of groups ({B}anff, {AB}, 1994)},
      series={CMS Conf. Proc.},
      volume={16},
   publisher={Amer. Math. Soc., Providence, RI},
       pages={59\ndash 78},
}

\bib{chari-pressley-dorey}{article}{
      author={Chari, V.},
      author={Pressley, A.},
       title={Yangians, integrable quantum systems and {D}orey's rule},
        date={1996},
        ISSN={0010-3616},
     journal={Comm. Math. Phys.},
      volume={181},
      number={2},
       pages={265\ndash 302},
}

\bib{chari-pressley-RepsChar}{article}{
      author={Chari, V.},
      author={Pressley, A.},
       title={Yangians: their representations and characters},
        date={1996},
        ISSN={0167-8019},
     journal={Acta Appl. Math.},
      volume={44},
      number={1-2},
       pages={39\ndash 58},
        note={Representations of Lie groups, Lie algebras and their quantum
  analogues},
}

\bib{DiKh00}{article}{
      author={Ding, J.},
      author={Khoroshkin, S.},
       title={Weyl group extension of quantized current algebras},
        date={2000},
        ISSN={1083-4362,1531-586X},
     journal={Transform. Groups},
      volume={5},
      number={1},
       pages={35\ndash 59},
}

\bib{Dr}{article}{
      author={Drinfel'd, V.},
       title={Hopf algebras and the quantum {Y}ang-{B}axter equation},
        date={1985},
     journal={Soviet Math. Dokl.},
      volume={32},
      number={1},
       pages={254\ndash 258},
}

\bib{DrNew}{article}{
      author={Drinfel'd, V.},
       title={A new realization of {Y}angians and quantum affine algebras},
        date={1988},
     journal={Soviet Math. Dokl.},
      volume={36},
      number={2},
       pages={212\ndash 216},
}

\bib{FrHer-15}{article}{
      author={Frenkel, E.},
      author={Hernandez, D.},
       title={Baxter's relations and spectra of quantum integrable models},
        date={2015},
        ISSN={0012-7094},
     journal={Duke Math. J.},
      volume={164},
      number={12},
       pages={2407\ndash 2460},
}

\bib{FrHer-22}{unpublished}{
      author={Frenkel, E.},
      author={Hernandez, D.},
       title={Weyl group symmetry of {$q$}-characters},
        date={2022},
        note={\tt arXiv:2211.09779},
}

\bib{FrHer-23}{unpublished}{
      author={Frenkel, E.},
      author={Hernandez, D.},
       title={Extended {B}axter relations and {QQ}-systems for quantum affine
  algebras},
        date={2023},
        note={To appear in {C}omm. {M}ath. {P}hys. {\tt arXiv:2312.13256}},
}

\bib{frenkel-mukhin}{article}{
      author={Frenkel, E.},
      author={Mukhin, E.},
       title={{Combinatorics of $q$--characters of finite--dimensional
  representations of quantum affine algebras}},
        date={2001},
     journal={Comm. Math. Phy.},
      volume={216},
      number={1},
       pages={23\ndash 57},
}

\bib{frenkel-reshetikhin-qchar}{inproceedings}{
      author={Frenkel, E.},
      author={Reshetikhin, N.~Yu.},
       title={{The $q$-characters of representations of quantum affine algebras
  and deformed $\mathcal{W}$ algebras}},
        date={1999},
   booktitle={{Recent Developments in Quantum Affine Algebras and Related
  Topics}},
      series={Contemp. Math.},
      volume={248},
   publisher={AMS, Providence, RI},
     address={Raleigh, NC},
       pages={163\ndash 205},
}

\bib{FiTs19}{article}{
      author={Finkelberg, M.},
      author={Tsymbaliuk, A.},
       title={Shifted quantum affine algebras: integral forms in type {$A$}},
        date={2019},
        ISSN={2199-6792},
     journal={Arnold Math. J.},
      volume={5},
      number={2-3},
       pages={197\ndash 283},
}

\bib{GKLO}{article}{
      author={Gerasimov, A.},
      author={Kharchev, S.},
      author={Lebedev, D.},
      author={Oblezin, S.},
       title={On a class of representations of the {Y}angian and moduli space
  of monopoles},
        date={2005},
        ISSN={0010-3616,1432-0916},
     journal={Comm. Math. Phys.},
      volume={260},
      number={3},
       pages={511\ndash 525},
}

\bib{GNW}{article}{
      author={Guay, N.},
      author={Nakajima, H.},
      author={Wendlandt, C.},
       title={Coproduct for {Y}angians of affine {K}ac-{M}oody algebras},
        date={2018, In},
        ISSN={0001-8708},
     journal={Adv. Math.},
      volume={338},
       pages={865\ndash 911},
}

\bib{guay-tan}{article}{
      author={Guay, N.},
      author={Tan, Y.},
       title={Local {W}eyl modules and cyclicity of tensor products for
  {Y}angians},
        date={2015},
        ISSN={0021-8693},
     journal={J. Algebra},
      volume={432},
       pages={228\ndash 251},
}

\bib{GTL1}{article}{
      author={Gautam, S.},
      author={Toledano~Laredo, V.},
       title={Yangians and quantum loop algebras},
        date={2013},
        ISSN={1022-1824},
     journal={Selecta Math. (N.S.)},
      volume={19},
      number={2},
       pages={271\ndash 336},
}

\bib{GTL2}{article}{
      author={Gautam, S.},
      author={Toledano~Laredo, V.},
       title={Yangians, quantum loop algebras, and abelian difference
  equations},
        date={2016},
        ISSN={0894-0347},
     journal={J. Amer. Math. Soc.},
      volume={29},
      number={3},
       pages={775\ndash 824},
}

\bib{GTL3}{article}{
      author={Gautam, S.},
      author={Toledano~Laredo, V.},
       title={Meromorphic tensor equivalence for {Y}angians and quantum loop
  algebras},
        date={2017},
        ISSN={0073-8301},
     journal={Publ. Math. Inst. Hautes \'{E}tudes Sci.},
      volume={125},
       pages={267\ndash 337},
}

\bib{GTLW19}{article}{
      author={Gautam, S.},
      author={Toledano~Laredo, V.},
      author={Wendlandt, C.},
       title={The meromorphic {$R$}-matrix of the {Y}angian},
        date={2021},
     journal={Progr. Math.},
      volume={340},
       pages={201\ndash 269},
}

\bib{GWPoles}{unpublished}{
      author={Gautam, S.},
      author={Wendlandt, C.},
       title={Poles of finite-dimensional representations of {Y}angians},
        date={2020},
        note={To appear in {S}electa. {M}ath. ({N.S}). {\tt arXiv:2009.06427}},
}

\bib{Her05}{article}{
      author={Hernandez, D.},
       title={Representations of quantum affinizations and fusion product},
        date={2005},
        ISSN={1083-4362},
     journal={Transform. Groups},
      volume={10},
      number={2},
       pages={163\ndash 200},
}

\bib{Her07}{article}{
      author={Hernandez, D.},
       title={Drinfeld coproduct, quantum fusion tensor category and
  applications},
        date={2007},
        ISSN={0024-6115},
     journal={Proc. Lond. Math. Soc. (3)},
      volume={95},
      number={3},
       pages={567\ndash 608},
}

\bib{Hern-simple}{article}{
      author={Hernandez, D.},
       title={Simple tensor products},
        date={2010},
        ISSN={0020-9910},
     journal={Invent. Math.},
      volume={181},
      number={3},
       pages={649\ndash 675},
}

\bib{Hern-cyclic}{article}{
      author={Hernandez, D.},
       title={Cyclicity and {$R$}-matrices},
        date={2019},
        ISSN={1022-1824},
     journal={Selecta Math. (N.S.)},
      volume={25},
      number={2},
       pages={Paper No. 19, 24pp},
}

\bib{HerZhang-shifted}{unpublished}{
      author={Hernandez, D.},
      author={Zhang, H.},
       title={Shifted {Y}angians and polynomial {R}-matrices},
        date={2022},
        note={To appear in {P}ubl. {R}es. {I}nst. {M}ath. {S}ci. {\tt
  arXiv:2103.10993}},
}

\bib{Ilin-Rybnikov-19}{article}{
      author={Ilin, A.},
      author={Rybnikov, L.},
       title={Bethe subalgebras in {Y}angians and the wonderful
  compactification},
        date={2019},
        ISSN={0010-3616,1432-0916},
     journal={Comm. Math. Phys.},
      volume={372},
      number={1},
       pages={343\ndash 366},
}

\bib{Jantzen}{book}{
      author={Jantzen, J.~C.},
       title={Lectures on quantum groups},
      series={Graduate Studies in Mathematics},
   publisher={American Mathematical Society, Providence, RI},
        date={1996},
      volume={6},
        ISBN={0-8218-0478-2},
}

\bib{KacBook90}{book}{
      author={Kac, V.},
       title={Infinite-dimensional {L}ie algebras},
     edition={Third edition},
   publisher={Cambridge University Press, Cambridge},
        date={1990},
        ISBN={0-521-37215-1; 0-521-46693-8},
}

\bib{kashiwara-level-zero}{article}{
      author={Kashiwara, M.},
       title={On level-zero representations of quantized affine algebras},
        date={2002},
        ISSN={0012-7094,1547-7398},
     journal={Duke Math. J.},
      volume={112},
      number={1},
       pages={117\ndash 175},
}

\bib{Kod19b}{article}{
      author={Kodera, R.},
       title={Braid group action on affine {Y}angian},
        date={2019},
     journal={SIGMA Symmetry Integrability Geom. Methods Appl.},
      volume={15},
       pages={Paper No. 020, 28},
}

\bib{KTWWY}{article}{
      author={Kamnitzer, J.},
      author={Tingley, P.},
      author={Webster, B.},
      author={Weekes, A.},
      author={Yacobi, O.},
       title={Highest weights for truncated shifted {Y}angians and product
  monomial crystals},
        date={2019},
     journal={J. Comb. Algebra},
      volume={3},
      number={3},
       pages={237\ndash 303},
}

\bib{Kumar}{book}{
      author={Kumar, S.},
       title={Kac-{M}oody groups, their flag varieties and representation
  theory},
      series={Progress in Mathematics},
   publisher={Birkh\"{a}user Boston, Inc., Boston, MA},
        date={2002},
      volume={204},
        ISBN={0-8176-4227-7},
}

\bib{Lusztig-Hecke}{book}{
      author={Lusztig, G.},
       title={Hecke algebras with unequal parameters},
      series={CRM Monograph Series},
   publisher={American Mathematical Society, Providence, RI},
        date={2003},
      volume={18},
        ISBN={0-8218-3356-1},
}

\bib{Lusztig}{book}{
      author={Lusztig, G.},
       title={Introduction to quantum groups},
      series={Modern Birkh\"{a}user Classics},
   publisher={Birkh\"{a}user/Springer, New York},
        date={2010},
        ISBN={978-0-8176-4716-2},
        note={Reprint of the 1994 edition},
}

\bib{MoBook}{book}{
      author={Molev, A.},
       title={Yangians and classical {L}ie algebras},
      series={Mathematical Surveys and Monographs},
   publisher={American Mathematical Society, Providence, RI},
        date={2007},
      volume={143},
        ISBN={978-0-8218-4374-1},
}

\bib{NaTa02}{article}{
      author={Nazarov, M.},
      author={Tarasov, V.},
       title={On irreducibility of tensor products of {Y}angian modules
  associated with skew {Y}oung diagrams},
        date={2002},
        ISSN={0012-7094},
     journal={Duke Math. J.},
      volume={112},
      number={2},
       pages={343\ndash 378},
}

\bib{NaTa98}{article}{
      author={Nazarov, M.},
      author={Tarasov, V.},
       title={On irreducibility of tensor products of {Y}angian modules},
        date={1998},
        ISSN={1073-7928},
     journal={Internat. Math. Res. Notices},
      number={3},
       pages={125\ndash 150},
}

\bib{Saito}{article}{
      author={Saito, Yoshihisa},
       title={P{BW} basis of quantized universal enveloping algebras},
        date={1994},
     journal={Publ. Res. Inst. Math. Sci.},
      volume={30},
      number={2},
       pages={209\ndash 232},
}

\bib{tan-braid}{unpublished}{
      author={Tan, Y.},
       title={Braid group actions and tensor products for {Y}angians},
        date={2015},
        note={{\tt arXiv:1510.01533}},
}

\bib{vv-standard}{article}{
      author={Varagnolo, M.},
      author={Vasserot, E.},
       title={Standard modules of quantum affine algebras},
        date={2002},
        ISSN={0012-7094,1547-7398},
     journal={Duke Math. J.},
      volume={111},
      number={3},
       pages={509\ndash 533},
}

\bib{WeekesThesis}{book}{
      author={Weekes, A.},
       title={Highest {W}eights for {T}runcated {S}hifted {Y}angians},
   publisher={ProQuest LLC, Ann Arbor, MI},
        date={2016},
        ISBN={978-1369-85412-1},
        note={Thesis (Ph.D.)--University of Toronto (Canada)},
}

\bib{WRTT}{article}{
      author={Wendlandt, C.},
       title={The {$R$}-matrix presentation for the {Y}angian of a simple {L}ie
  algebra},
        date={2018},
        ISSN={0010-3616},
     journal={Comm. Math. Phys.},
      volume={363},
      number={1},
       pages={289\ndash 332},
}

\end{biblist}
\end{bibdiv}

\end{document}